\documentclass[12pt]{amsart}
\usepackage{latexsym}
\usepackage{amssymb}
\usepackage{amsmath}

\newtheorem{thm}{Theorem}[section]
\newtheorem{lem}[thm]{Lemma}
\newtheorem{cor}[thm]{Corollary}
\newtheorem{prop}[thm]{Proposition}

\theoremstyle{remark}
\newtheorem{rem}[thm]{Remark}

\numberwithin{equation}{section}

\newcommand{\cs}{{\mathcal S}}

\newcommand{\cb}{{\mathcal B}}
\newcommand{\ca}{{\mathcal A}}

\newcommand{\R}{{\mathbb{R}}}
\newcommand{\T}{{\mathbb{T}}}
\newcommand{\C}{{\mathbb{C}}}
\newcommand{\Z}{{\mathbb{Z}}}
\newcommand{\N}{{\mathbb{N}}}
\newcommand{\xbm}{(X,\mathcal B,p)}

\newcommand{\ycn}{(Y,{\mathcal C},\nu)}

\newcommand{\tfs}{T_{\varphi,{\mathcal S}}}

\newcommand{\ov}{\overline}
\newcommand{\beq}{\begin{equation}}
\newcommand{\eeq}{\end{equation}}
\newcommand{\vep}{\varepsilon}
\newcommand{\va}{\varphi}

\newcommand{\ot}{\otimes}
\newcommand{\la}{\lambda}

\begin{document}

\title{Rigidity and Non-recurrence along Sequences}
\author{V. Bergelson, A. del Junco, M. Lema\'nczyk, J. Rosenblatt}
\date{February, 2011}

\begin{abstract}
Two properties of a dynamical system, rigidity and non-recurrence,
are examined in detail. The ultimate aim is to characterize the sequences
along which these properties do or
do not occur for different classes of transformations.  The
main focus in this article is to characterize explicitly the
structural properties of sequences which can be rigidity sequences
or non-recurrent sequences for some weakly mixing dynamical system.
For ergodic transformations generally and for weakly mixing transformations
in particular there are both parallels and distinctions between the class of
rigid sequences and the class of non-recurrent sequences.
A variety of classes of sequences with various properties are considered
showing the complicated and rich structure of rigid and
non-recurrent sequences.
\end{abstract}

\maketitle

\section{\bf Introduction}
\label{intro}

Let $(X,\mathcal B,p,T)$ be a dynamical system: that is, we have
a non-atomic probability space $(X,\mathcal B,p)$ and an invertible measure-preserving transformation
$T$ of $(X,\mathcal B,p)$.   We consider here two properties of the dynamical system
$(X,\mathcal B,p,T)$, rigidity and non-recurrence. Ultimately we would
like to characterize the sequences along which these properties do, or
do not occur, for different classes of transformations. The
main focus here is to characterize which subsequences $(n_m)$ in
$\mathbb Z^+$ can be a sequence for rigidity, and which
can be a sequence for non-recurrence, for some
weakly mixing dynamical system. In the process of doing this, we will
see that there are parallels and distinctions between the class of
rigid sequences and the class of non-recurrent sequences, both for
ergodic transformations in general and for weakly mixing transformations
in particular.

The properties of rigidity and non-recurrence along a given sequence $(n_m)$
are opposites of one another. By {\em  rigidity along the
sequence} $(n_m)$ we mean that the powers $(T^{n_m})$ are converging in
the strong operator topology to the identity; that is, $\|f\circ
T^{n_m} - f\|_2 \to 0$ as $m \to \infty$, for all $f \in L_2(X,p)$.
So rigidity along $(n_m)$ means that $p(T^{n_m}A\cap A) \to p(A)$
as $m\to \infty$ for all $A \in \mathcal B$. On the other hand,
{\em non-recurrence along the sequence} $(n_m)$ means that for some
$A \in \mathcal B$ with $p(A) > 0$, we have $p(T^{n_m}A \cap A) = 0$
for all $m \ge 1$. Nonetheless, there are structural parallels between these
two properties of a sequence $(n_m)$. For example, neither property can occur
for an ergodic transformation unless
the sequence is sparse. Also, these two properties
cannot occur without the sequence $(n_m)$ having (or avoiding)  
various combinatorial or algebraic structures. 
These properties can occur simultaneously for a given transformation
if the sequences are disjoint.  For example, we are
able to use Baire category results to show that the generic
transformation $T$ is weak mixing and rigid along some sequence $(n_m)$
such that it is also non-recurrent along $(n_m-1)$.  In proving this,
one sees a connection between rigidity and non-recurrence.  The non-recurrence
along $(n_m-1)$ is created by first using rigidity to take a rigid sequence
$(n_m)$ for $T$ and a set $A, p(A) > 0$,
such that $\sum\limits_{m=1}^\infty p(T^{n_m}A\Delta A) \le \frac 1{100}p(A)$.
This allows one to prove that $T$ is non-recurrent along $(n_m-1)$ for some
subset $C$ of $TA$.  One can extend this argument somewhat and show that
for every whole number $K$, there is a weakly mixing transformation $T$ that is rigid
along a sequence $(n_m)$, such that also for some set $C$, $p(C) > 0$,
$T$ is non-recurrent for $C$ along $(n_m+k)$ for all $k \not= 0, |k| \le K$.

First, in Section~\ref{rigidity}, we discuss some generalities about
rigidity and weak mixing.  We also consider the more restrictive property
of IP-rigidity. We will see that both rigidity and IP-rigidity can be viewed as a
spectral property and therefore characterized in terms of the behavior of
the Fourier transforms $\widehat \nu$ of the positive
Borel measures $\nu$ on $\mathbb T$ that are the spectral measures
of the dynamical system. We will see that rigidity
sequences must be sparse, but later in Section~\ref{ratesofgrowth},
it is made clear that they are not necessarily very sparse.
In addition, we show that rigidity sequences cannot
have certain types of
algebraic structure for rigidity to occur even for an ergodic
transformations, let alone a weakly mixing one.

After this in Section~\ref{methods}, we prove a variety of results
about rigidity that serve to demonstrate how rich and complex is
the structure of rigid sequences.  Here is a sample of what we prove:
\medskip

\noindent {\bf a)} In Proposition~\ref{ratiogrows} we show that
if $\lim\limits_{m \to \infty} \frac {n_{m+1}}{n_m} = \infty$, then
$(n_m)$ is a rigidity sequence for some weakly mixing transformation $T$.
This result uses the Gaussian measure space construction.
Also, by a cutting and stacking construction, we construct an infinite measure preserving
rank one transformation $S$
for which $(n_m)$ is a rigidity sequence.
Under some additional assumptions on $(n_m)$, we can use the cutting and
stacking construction to produce a
weakly mixing rank one transformation $T$ on a probability space for which $(n_m)$
is a rigidity sequence.  See specifically
Proposition~\ref{specialinfrankone} and generally Section~\ref{rankone}.
\medskip

\noindent {\bf b)} In contrast, we show that sequences like $(a^m: m \ge 1)$, and $a
\in \mathbb N, a\not= 1$,
are also rigidity sequences for weakly mixing transformations.  However,
perturbations of them, like
$(a^m+p(m): m \ge 1)$ with $p \in \mathbb Z[x], p \not=0$, are never rigidity
sequences for ergodic transformations,
let alone weakly mixing transformations. See Proposition~\ref{integerratios}
and Remark~\ref{linformeg} c).
\medskip

\noindent {\bf c)} We prove a number of results in
Section~\ref{ratesofgrowth} that show that rigidity
sequences do not necessarily have to grow quickly, but rather can have
their density decreasing to zero
infinitely often slower than any given rate. One consequence may
illustrate what this tells us: we
show that there are rigidity sequences for weakly mixing transformations
that are not Sidon sets.
See Proposition~\ref{ergodicratewm} and Corollary~\ref{notSidon}.
\medskip

\noindent {\bf d)} In Section~\ref{secdisjoint}, we show that there is
no universal rigid sequence.   That is, we show that given a weakly
mixing transformation $T$ that is rigid along some sequence, there is
another weakly mixing transformation $S$ which is rigid along some other
sequence such that $T\times S$ is not rigid along any sequence.
\medskip

\noindent {\bf e)} In Section~\ref{seccocycle}, we show how cocycle
construction can be used to construct rigidity sequences
for weakly mixing transformations. One particular result is
Corollary~\ref{denominators}: if $(\frac {p_n}{q_n}:n\ge 1)$
are the convergents associated with the continued fraction expansion of
an irrational number, then
$(q_n)$ is a rigidity sequence for a weakly mixing transformation.
\medskip

We then consider non-recurrence in Section~\ref{nonrecurrence}.
We show that the sequences exhibiting non-recurrence must be
sparse and cannot have certain types of algebraic structure for there to be
non-recurrence even for ergodic transformation, let alone a
weakly mixing one.
We conjecture that any lacunary sequence is a sequence of non-recurrence
for some
weakly mixing transformation, but we have not been able to prove this
result at this time.
Here are some specific results on non-recurrence that we prove:
\medskip

\noindent {\bf a)}
It is well-known that the generic transformation $T$ is weakly mixing and
rigid.  We show that in addition, there is a
rigidity sequence $(n_m)$ for such a generic $T$, so that for any whole
number $K$, each
$(n_m+k), 0 < |k|\le K$, is a
 non-recurrent sequence for $T$. See Proposition~\ref{revise} and Remark~\ref{notrecshift}.
\medskip

\noindent {\bf b)} We observe in Proposition~\ref{Chacon} that some
weakly mixing transformations,
like Chacon's transformation, are non-recurrent along a lacunary sequence with
bounded ratios.
\medskip

\noindent {\bf c)} We also show that for any increasing sequence $(n_m)$
with $\sum\limits_{m=1}^\infty \frac {n_m}{n_{m+1}} < \infty$,
and a whole number $K$, there is
a weakly mixing
transformation $T$ and
a set $C, p(C) > 0$, such that $(n_m)$ is a rigidity sequence for $T$ and $T$
non-recurrent for $C$ along $(n_m +k)$ for all $k, 0 < |k| \le K$.
See Proposition~\ref{fastworks} and Remark~\ref{notrecshiftagain}.
\medskip

When considering both rigidity and non-recurrence of measure-preserving transformations,
there are often unitary versions of the results that are either almost identical in
statement and proof, or worth more consideration.  When possible, we will take note of
this.  See Krengel~\cite{Krengel} for a general reference on this and other aspects of
ergodic theory used in this article.

There is also a larger issue of considering both rigidity and non-recurrence for
general groups of invertible measure-preserving transformations.  This
will require a careful look at general spectral issues, including the irreducible
representations of the groups.  We plan to pursue this in a later
paper.

\section{\bf Generalities on Rigidity and Weak Mixing}
\label{rigidity}

Suppose we consider a dynamical system $(X,\mathcal B,p,T)$. Unless it
is noted otherwise, we will be assuming that $(X,\mathcal B,p)$ is a
standard Lebesgue probability space i.e. it is measure theoretically
isomorphic to $[0,1]$ in Lebesgue measure. In particular it is non-atomic
and $L_2(X,p)$ has a countable dense subset in the norm topology. We say that the
dynamical system is separable in this case. We also assume that
$T$ is an invertible measure-preserving transformation
$(X,\mathcal B,p)$.

This section provides the background information needed in this article.
First, in Section~\ref{rigidwmboth} we deal with the basic properties of rigidity and
weak mixing in order to give a general version of the well-known fact that the generic transformation
is both weakly mixing and rigid along some sequence.  Second, in Section~\ref{weakmixsec} we
look at weak mixing and aspects of it that are important to this article.  Third, in Section~\ref{rigidonly} we
consider rigidity itself in somewhat more detail.  See Furstenberg and Weiss~\cite{FurstWeiss}
and Queffelec~\cite{Queff}, especially Section 3.2.2,
for background information about rigidity as we consider it, and other types of rigidity
that have been considered by other authors.

\subsection{\bf Rigidity and Weak Mixing in General}\label{rigidwmboth}

Given an increasing sequence $(n_m)$ of integers we consider the
family $$
\ca(n_m)=\{A\in\mathcal B:\:p\left(T^{n_m}A\triangle
A\right)\to 0\}.$$ We now recall some basic and well-known facts
about $\ca(n_m)$. See Walters~\cite{Walters2} for the following
result.

\begin{prop}
$\ca(n_m)\subset\mathcal B$ is a sub-$\sigma$-algebra which is also
$T$-invariant. $\ca(n_m)$ is the maximal $\sigma$-algebra
$\ca\subset\mathcal B$ such that
$$
T^{n_m}|_{\ca}\to Id|_{\ca}\;\;\mbox{as}\;\;m\to\infty.$$
Moreover \beq\label{ww1} L_2(X,\ca(n_m),p)=\{f\in L_2(X,\mathcal
B,p):\:f\circ
T^{n_m}\to f\;\;\mbox{in}\;\;L_2(X,\mathcal B,p)\}.\eeq
\end{prop}

\begin{rem} An approach to the above result different than in ~\cite{Walters2}
begins by observing that $\{f\in L_\infty(X,p):
\|f\circ T^{n_m} - f\|_2 \to 0\}$ is an algebra. So there is a corresponding
factor map of $(X,\mathcal B,p,T)$ for which there is an associated
$T$ \-invariant sub-$\sigma$-algebra, namely
$\ca(n_m)\subset\mathcal B$.
\end{rem}

If $\ca(n_m)=\mathcal B$ then one says that $(n_m)$ is a {\em
rigidity sequence} for $(X,\mathcal B,p,T)$. Systems possessing rigidity
sequences are called {\em rigid}. The
fact that $(n_m)$ is a rigidity sequence for $T$ is a {\em spectral property}; that is, it is a
unitary invariant of the associated Koopman operator $U_T$ on
$L_2(X,p)$ given by the formula $U_T(f)=f\circ T$. The following
discussion should make this clear.

First, recall some basic notions of spectral theory (see e.g.
\cite{CFS}, \cite{Ka-Th}, \cite{Pa}).  For each $f\in L_2(X,p)$, the
function $\rho(n) = \langle f\circ T^n,f\rangle$ is a positive-definite
function and hence, by the Herglotz Theorem, is the Fourier
transform of a positive Borel measure on the circle $\mathbb T$.  So
for each $f \in L_2(X,p)$, there is a unique positive Borel
measure $\nu_f^T$ on $\mathbb T$, called the
{\em spectral measure} for $T$ corresponding to $f$
which is determined by $\widehat {\nu_f^T}(n) = \langle f\circ
T^{n},f\rangle$ for all $n \in \mathbb Z$. Spectral measures are
non-negative and have $\nu_f^T(\mathbb T) =\|f\|_2^2$.
We will also need to use the adjoint $\nu^*$ given by $\nu^*(E) = \overline
{\nu(E^{-1})}$ for all Borel sets $E\subset \mathbb T$. The
adjoint has $\widehat {\nu^*}(n) =\overline {\widehat {\nu}(-n)}$
for all $n \in \mathbb Z$.

Absolute continuity of measures is important
here: given two positive Borel measures $\nu_1$ and $\nu_2$ on $\mathbb T$,
we say $\nu_1$ is absolutely continuous with respect to $\nu_2$,
denoted by $\nu_1 \ll \nu_2$, if $\nu_1(E) = 0$ for all Borel sets
$E$ such that $\nu_2(E) = 0$. Now, among all spectral
measures there exist measures $\nu_F^T$ such that all other
spectral measures are absolutely continuous with respect to
$\nu_F^T$. Any one of these is called a {\em maximal spectral measure}
of $T$.  These measures are all mutually
absolutely continuous with respect to one another.
The equivalence class of the maximal spectral measures is denoted by $\nu^T$.
By abuse of notation, we refer to $\nu^T$ as a measure too.
Recall that the type of a finite positive measure (e.g. whether the
measure is singular, absolutely continuous with respect to Lebesgue measure, etc.)
is a property of the equivalence
class of all finite positive measures $\omega$ such that $\omega \ll \nu$
and $\nu \ll \omega$. The type of $\nu^T$ (that is, of a maximal spectral measure $\nu^T_F$)
has a special role in the structure of the transformation. For this
reason the type of $\nu^T$ is called the {\em
maximal spectral type} of $T$. For example, rigid transformations must
have singular maximal spectral type; see Remark~\ref{Rokhlinsingular}. Also, Bernoulli transformations
must have Lebesgue type i.e. their maximal spectral measures are
equivalent to
Lebesgue measure. In general, a strongly mixing transformation does not need to be
of Lebesgue type.  It could be of singular type (this occurs when every
maximal spectral measure is singular but yet has the Fourier transform
tending to zero at infinity).

For a given sequence $(n_m)$, a transformation $T$
and a function $f \in L_2(X,p)$, we say {\em $(n_m)$ is
a rigidity sequence of $T$ for $f$} if
$f\circ T^{n_m}\to f$ in $L_2$-norm. Recall that
$\widehat{\nu^T_f}(n_m)=\langle f\circ
T^{n_m},f\rangle$.

\begin{prop} \label{rigidfacts} Fix the transformation $T$. The
following are equivalent for $f \in L_2(X,p)$:
\begin{enumerate}
\item The sequence $(n_m)$ is
a rigidity sequence for the function $f$.
\item $\langle f\circ T^{n_m},f\rangle=\int_X f\circ T^{n_m}\cdot
\overline{f}\,dp\to \|f\|_2^2$.
\item $\widehat{\nu^T_f}(n_m)\to\|f\|_2^2$.
\item $z^{n_m} \to 1$ in
$L_2(\mathbb T,\nu_f^T)$.
\item $z^{n_m} \to 1$ in measure with respect to
$\nu_f^T$.
\end{enumerate}
\end{prop}
\begin{proof} We have $\|f \circ T^{n_m} - f\|_2^2 =2\|f\|_2^2 -
2\text{Re}\langle f\circ T^{n_m},f\rangle$.
Since $|\langle f\circ T^{n_m},f\rangle| \le \|f\|_2^2$, we see that (1)
is equivalent to (2). Now
(2) is equivalent to (3) by the definition of the spectral measure
$\nu_f^T$. We also have
$\int |z^{n_m}-1|^2\, d\nu_f^T(z) = 2\|f\|_2^2 - 2\text{Re}(\widehat
{\nu_f^T}(n_m))$. Since
$|\widehat {\nu_f^T}(n_m)|\le \|f\|_2^2$, we see that (3) is equivalent
to (4). It is clear
that (4) is equivalent to (5) because $|1-z^{n_m}| \le 2$ and
$\nu_f^T$ is a finite, positive measure.
\end{proof}

\begin{rem} This result is really a fact about a unitary operator
$U$ on a Hilbert space $H$.  That is, a sequence $(n_m)$ and
vector $v \in H$ satisfy
$\lim\limits_{m \to \infty} ||U^{n_m} v - v \|_H = 0$ if and only
if the spectral measure $\nu_v^U$ determined by $\widehat {\nu_v^U}(k)
= \langle U^kv,v\rangle $ for all $k \in \mathbb Z$ has the property that
$\lim\limits_{m \to \infty} \widehat {\nu_v^U}(n_m) = \|v\|_H^2$.
\end{rem}

Proposition ~\ref{rigidfacts} shows that
if $(n_m)$ is a rigidity sequence for $T$ for
a given function $F$, then for any spectral measure
$\nu_f^T \ll \nu_F^T$, we would also have $z^{n_m} \to 1$
in measure with respect to $\nu_f^T$. Hence, $(n_m)$ would
be a rigidity sequence for $T$ for the function $f$ too. It
follows then easily that $T$ is rigid and has a rigidity
sequence $(n_m)$ if and only if $(n_m)$ is a rigidity sequence for
$F$ where $\nu_F^T$ is a maximal spectral measure for
$T$.
\begin{cor}\label{thouvenot}
$T$ is rigid if and only if for each
function $f\in L_2(X,p)$ there exists $(n_m)=(n_m(f))$ such that
$f\circ T^{n_m}\to f$ in $L_2(X,p)$.
\end{cor}

\begin{rem}  It is clear that an argument like this works equally
well for a unitary transformation $U$ of a separable Hilbert space
$H$.  That is, there is one sequence $(n_m)$ such that for all $v \in H$,
$\|U^{n_m}v - v\|_H \to 0$ as $m \to \infty$ if and only if
for every vector $v \in H$, there exists a
sequence $(n_m)$ such that
$\|U^{n_m}v - v\|_H \to 0$ as $m \to \infty$
\end{rem}

\begin{rem} J.-P.
Thouvenot was the first to observe that $T$ is rigid if and only if
for each $f \in L_2(X,p)$ (or just for each characteristic
function $f = 1_A, A \in \mathcal B$), there exists $(n_m)$ depending on
$f$ such
that $\|f\circ T^{n_m} - f\|_2 \to 0$ as $m \to \infty$.
There are a number of different ways to prove this. We have
given one such argument above. Another argument
would use the characterization up to isomorphism of
unitary operators as multiplication operators.  Here is
an interesting approach via Krieger's Generator Theorem;
see Krieger~\cite{Krieger}.
It is sufficient to prove rigidity holds assuming that one has the weaker
condition of there being rigidity sequences for each characteristic
function. Suppose that an automorphism $T$ has the property
that for each set
$A\in\mathcal B$ there exists $(n_m)=(n_m(A))$ such that
$p\left(T^{-n_m}A\triangle A\right)\to 0$.
Then all spectral measures of functions of the
form $1_A$, $A\in\mathcal B$ are singular, and since the family of
such functions is linearly dense, the maximal spectral type of $T$
is singular. It follows that $T$ has zero entropy;
see Remark~\ref{Rokhlinsingular}
for an explanation of this point.   Hence, by Krieger's
Generator Theorem, there exists a two element partition
$P=\{A,A^c\}$ which generates $\mathcal B$. Now, let
$(n_m)=(n_m(A))$ and notice that for each $k\geq1$ and for each
$B\in\bigvee_{i=0}^{k-1}T^iP$ we have $p(T^{n_m}B\triangle
B)\to0$. Hence by approximating the $L_2(X,p)$ functions by
simple functions, $(n_m)$ is a rigidity sequence for $T$.
\end{rem}

\begin{rem}\label{Rokhlinsingular}
 From Proposition~\ref{rigidfacts}, we see that a maximal spectral
measure $\nu^T$ of a rigid transformation is
a {\em Dirichlet measure}.  This means that for some increasing sequence
$(n_m)$, we have $\gamma^{n_m} \to 1$ in measure with respect to $\nu^T$
as $m \to \infty$.  Hence, as in Proposition~\ref{rigidfacts},
we have $\widehat{\nu^T}(n_m)\to \nu^T(\mathbb T)$ as $m \to \infty$. A measure
with this property is also sometimes called a {\em rigid measure}.  Note that a
measure absolutely continuous with respect to a Dirichlet measure
is a Dirichlet measure. So by the Riemann-Lebesgue Lemma,
there is no
non-zero positive measure $\nu$ which is absolutely continuous with
respect to Lebesgue measure
such that $\nu \ll\nu_f^T$ for a non-zero spectral measure $\nu_f^T$ of a
rigid transformation.
Therefore, for a rigid transformation, all spectral measures, and
$\nu^T$ itself, are Dirichlet measures and hence
singular measures.  So $T$ has singular maximal spectral type.
Rokhlin shows in his
classical paper ~\cite{Rokhlin} that if $T$ has positive
entropy, then for every maximal spectral measure $\nu_F^T$, there is a
non-zero spectral
measure $\nu_f^T \ll \nu_F^T$ that is equivalent
to (mutually absolutely continuous with respect to) Lebesgue measure.
Therefore, all rigid transformations have zero entropy.
\end{rem}

One can often use Baire category arguments to distinguish the
behavior of transformations. For this we use the Polish group
$Aut(X,\mathcal B,p)$ of invertible measure-preserving transformations
on $(X,\mathcal B,p)$, with the topology of strong operator convergence.  That
is, a sequence $(S_n)$ in $Aut(X,\mathcal B,p)$ converges to $S \in
Aut(X,\mathcal B,p)$ if and only if $\|f\circ S_n- f\circ S\|_2 \to 0$
as $n\to \infty$ for all $f \in L_2(x,p)$.  By a {\em generic property}, we mean that the
property holds on at least a dense $G_\delta$ subset of $Aut(X,\mathcal
B,p)$, and any set containing a dense $G_\delta$
set is called a {\em generic set}.  So a generic property is one that holds on a set whose
complement is first category.
For example, it is well-known that the generic dynamical system is
weakly mixing. See Halmos~\cite{Halmos} where this was used to give
a Baire category argument for the existence of weakly mixing transformations
that are not strongly mixing. Also, the generic transformation has
a rigidity sequence. See Katok and Stepin~\cite{Ka-St} and
Walters~\cite{Walters2}.
Hence, the generic transformation is
weakly mixing, rigid, and has zero entropy (see Remark~\ref{Rokhlinsingular}).
We will show this in a slightly more
general setting.

First, in order to see that a generic transformation is rigid we will
prove the following stronger result. This result may be well-known,
but we provide a proof because we could not find
a good reference for it. In this proof, and then later in
Section~\ref{rankone},
{\em rank one transformations} arise. These are transformations obtained
by cutting and stacking where at each inductive stage only one Rokhlin
tower is used.  See Nadkarni~\cite{Nadkarni} and
Ferenczi~\cite{Ferenczi} for background information about rank one transformations.

\begin{prop}\label{folklore1} Given an increasing sequence $(n_m)$ of
natural
numbers, let ${\mathcal G}_{(n_m)}$ be the set consisting of all
$S\in Aut\xbm$ such that
$S^{n_{m_k}}\to Id$ in the strong operator topology
for some subsequence $(n_{m_k})$ of
$(n_m)$. Then ${\mathcal G}_{(n_m)}$ is a generic subset of
$Aut\xbm$.
\end{prop}
\begin{proof}
We can obtain a metric $d$ for $Aut\xbm$ that is
compatible with the strong topology as follows.  Take
$\{A_i:\:i\geq1\}$ which is a dense subset
in $(\mathcal B,p)$.  Let $d$ be given by
\beq\label{metric}d(R,S)=\sum_{i=1}^\infty\frac1{2^i}
(p\left(RA_i\triangle SA_i\right)+p\left(R^{-1}A_i\triangle
S^{-1}A_i\right))\eeq
It follows that given $n\in\Z$ and $\vep>0$ the set
$$
\{S\in Aut\xbm:\:d\left(S^n,Id\right)<\vep\}$$ is open, and
therefore the set
$$
A_{k,\vep}:=\{S\in Aut\xbm:\: d\left(S^{n_q},Id\right)<\vep \
\text{for some}\ q \ge k\}$$ is open as well.
Also, $A_{k,\vep}$  is dense.
Indeed, given $A_1,\ldots,A_m$ and $R\in Aut\xbm$, we can construct $S\in
A_{k,\vep}$ so that \beq\label{p1} p(SA_i\triangle RA_i)\;\mbox{is
as close to zero as we like,}\eeq and also for some $q\geq k$
\beq\label{p2}p(S^{n_q}A_i\triangle A_i)\;\mbox{is as close to
zero as we like}\eeq for $i=1,\ldots,m$.  Actually, as needed,
the argument below will show the
same facts hold if we replace $R$ and $S$ by their inverses. With no loss of
generality, we can assume that $R$ is of rank one as this family
is dense in $Aut\xbm$; see ~\cite{Nadkarni} and ~\cite{Ferenczi}.
This allows us to approximate
the sets $A_1,\ldots,A_m$ by unions of levels of large Rokhlin towers for $R$. Now fix
$n_q$.  We will see that the only condition on $n_q$ will be that
$n_q\to\infty$.  Let $h_s$ be the height of a Rokhlin tower for
$R$ so that the levels of the tower can be used to
 approximate the sets $A_1,\ldots,A_m$. We also assume
 without loss of generality that $h_s$ is a (large) multiple of
 $n_q$.  We now divide
this tower into consecutive subtowers (without changing the
levels) of height $n_q$.  This is done by taking the first $n_q$ levels, then
the next $n_q$ levels, etc.  Then define $S$ in the following way:
inside each subtower of height $n_q$, the automorphism $S$ acts as
$R$ except on the top level of the subtower on which we require that $S$ sends
this level into the bottom level of that subtower.  For example,
 for the first
subtower, the first level is sent into the second, the second to the third,
and so on, but the $n_q$-th level is mapped to the first.  This same
pattern is used on the rest of the subtowers.  Now, if $h_s/n_q$ is sufficiently
large, then we can see that we can well approximate each $A_i$ by a union of levels
of some of the $h_s/n_q$ Rokhlin subtowers (of height
$n_q$).   Taking this approach, the errors in
(\ref{p1}),~(\ref{p2}) come only from the fact that these subtowers are
cyclically permutated by $S$. The total error is hence of order
$$
\frac{h_s}{n_q}\cdot\frac1{h_s}=\frac1{n_q}.$$
Hence, to get the approximations
 we need, we only need to
know that $n_q\to \infty$.
Now take $0<\vep_l\to0$ and notice that the set
$$
\bigcap_{l=1}^\infty\bigcap_{k=1}^\infty A_{k,\vep_l}
$$
is included in ${\mathcal G}_{(n_m)}$.
\end{proof}

\begin{rem} This argument should be compared with the beginning of
the proof of Proposition~\ref{revise}.  Also, see the end of
Section~\ref{contfrac} for another
approach using continued fractions that gives generic results.
\end{rem}

\begin{rem}\label{metric1} We also notice that the metric $d$
defined in~(\ref{metric}) has the following properties:
$d(R,S)=d(R^{-1},S^{-1})$ and $d(TR,TS)=d(R,S)$ once $T$ commutes
with $R$ and $S$. Denote $\|T\|=d(T,Id)$. Then
$$
\|T^{n+m}\|=d(T^{n},T^{-m}\|\leq
d(T^n,Id)+d(Id,T^{-m})=\|T^n\|+\|T^m\|.$$
\end{rem}

\vspace{1ex}

\subsection{Weak Mixing Specifically} \label{weakmixsec}
Now we consider weakly mixing transformations.  Recall
that $T$ is {\em weakly mixing} if and only if for all $A,B \in \mathcal B$,
we have
$$\lim\limits_{N\to\infty}
\frac 1N\sum\limits_{n=1}^N |p(T^nA\cap B) - p(A)p(B)| = 0.$$
So $T$ is weakly mixing if and only if
for all  mean-zero $f \in L_2(X,p)$, we have
$$\lim\limits_{N\to \infty} \frac 1N
\sum\limits_{n=1}^N |\langle f\circ T^{n},f\rangle| = 0.$$
Now recall Wiener's Lemma: given a positive Borel measure
$\nu$ on $\T$ we have
$$\lim\limits_{N \to \infty} \frac 1{2N+1}
\sum\limits_{n=-N}^N |\widehat{\nu}(n)|^2 =
\sum_{\gamma \in\T}\nu^2(\{\gamma\}).$$
It follows that $\nu$ is continuous (i.e. has no point masses)
if and only if $\lim\limits_{N \to \infty} \frac 1{2N+1}
\sum\limits_{n=-N}^N |\widehat{\nu}(n)|^2 =0$. The latter
condition is well-known to be equivalent to the fact that
$\widehat{\nu}(n)$ tends to zero along a subsequence of
density $1$. Here we say that a set $S \subset \mathbb N$
of density one if
\[\lim\limits_{N\to\infty} \frac 1N\#(S\cap
\{1,2,\ldots,N\}) = 1.\]

So a transformation $T$ is weakly mixing if and only if $\nu_f^T$ is
continuous for each $f \in L_2(X,p)$ which is mean-zero, that is
(by Wiener's Lemma) we have $\langle f\circ T^{n},f\rangle$ tends
to zero along a sequence of density one. Denote a given such
density one sequence by $\mathcal N^T_f$. As $f$ changes, this
sequence generally might need to change. However, it is a well-known
fact that because $(X,\mathcal B,p)$ is separable,
we can choose a subsequence of density~$1$ which works for
all $L_2$-functions. We give a proof of this fact for the reader's
convenience.  This proof is different than the one in
Petersen~\cite{Petersen}.  See also Jones~\cite{LeeJones}.

\begin{prop} \label{wmone} Assume that $T$
 is weakly mixing. Then
there is a sequence $(n_m)$ in $\mathbb Z^+$ of density one
such that for all mean-zero $f \in L_2(X,p)$, one has
$\lim\limits_{m \to \infty} \langle f\circ T^{n_m},f\rangle = 0$.
\end{prop}
\begin{proof} Let $(f_s)$ be a sequence of non-zero mean-zero functions
which is dense in the subspace of $L_2(X,p)$
consisting of the mean-zero functions. Consider the measure
$\omega = \sum\limits_{s=1}^\infty \frac 1{2^s\|f_s\|_2^4}\,
\nu_{f_s}^T\ast(\nu_{f_s}^T)^*$. This is a continuous measure with a
positive Fourier transform.
Hence, there is a sequence $(n_m)$ in $\mathbb Z^+$
of density one such that $\widehat {\omega}(n_m) \to 0$ as $m \to
\infty$. For every $s\geq1$, we have $2^s\|f_s\|_2^4 \,\widehat
{\omega}(n)\ge |\widehat {\nu_{f_s}^T}(n)|^2$ for all
$n \in \mathbb Z$. So it follows that for every $s\geq1$, we also
have $\widehat {\nu_{f_s}^T}(n_m) \to 0$ as $m \to \infty$. 
Then, by a standard approximation argument, for any mean-zero function
$f \in L_2(X,p)$, we have $\widehat
{\nu_f^T}(n_m) \to 0$ as $m \to \infty$.
\end{proof}

\begin{rem}  This result also holds for a unitary operator $U$ on
a separable Hilbert space $H$.  That is, if all the spectral measures
$\nu_v^U$ for $v \in H$ are continuous, then there exists a sequence
$(n_m)$ of density one such that $\widehat {\nu_v^U}(n_m) \to 0$ as
$m \to \infty$.
\end{rem}

Let $L_{2,0}(X,p)$ denote the mean-zero functions in
$L_2(X,p)$.  We can rewrite the assertion of Proposition~\ref{wmone} as
\beq\label{ww3} \mbox{$U_T^{n_m}\to 0$ weakly in the space
$L_{2,0}(X,p)$}.\eeq Each sequence $(n_m)$ (not necessarily of
density~$1$) of integers for which~(\ref{ww3}) holds is called a
{\em mixing sequence} for $T$. Any
transformation possessing a mixing sequence is weakly mixing.

The following result about mixing subsequences is also folklore.

\begin{prop}\label{mixresidual}Given an increasing sequence $(n_m)$ of
natural
numbers, consider the set $\mathcal M_{(n_m)}$ that consists of all
$S\in Aut\xbm$
such that $S^{n_{m_k}}\to 0$ weakly in $L_{2,0}(X,p)$, for some
subsequence $(n_{m_k})$
of $(n_m)$. Then $\mathcal M_{(n_m)}$ is a generic subset of
$Aut\xbm$.\end{prop}
\begin{proof} Let $\{A_i:\:i\geq1\}$ be a dense family in
$(\mathcal B,d)$. Take $\vep>0$ and set
$$
{\mathcal M}(k,\vep)=\{S\in Aut(X,\mathcal B,p):\:\sum_{i,j=1}^\infty
\frac1{2^{i+j}}\left|p(S^{-n_k}A_i\cap
A_j)-p(A_i)p(A_j)\right|<\vep\}.
$$
Notice that ${\mathcal M}(k,\vep)$ is open and, for
$0<\vep_i\to0$, consider the set
$$
{\mathcal M}= \bigcap_{i=1}^\infty\bigcup_{k=i}^\infty {\mathcal
M}(k,\vep_i).$$ It is not hard to check that $\bigcup_{k=i}^\infty
{\mathcal M}(k,\vep_i)$ is dense (each mixing
transformation belongs to it), so $\mathcal M$ is a $G_\delta$ and
dense. If $T\in{\mathcal M}$ then for each $i\geq1$ there exists
$k_i\geq i$ such that
$$
\sum\limits_{r,s=1}^\infty\frac1{2^{r+s}}\left| p(T^{-n_{k_i}}A_r\cap
A_s)-p(A_r)p(A_s)\right|<\vep_i.$$ Hence for each $r,s\geq1$
$$
\left| p(T^{-n_{k_i}}A_r\cap
A_s)-p(A_r)p(A_s)\right|\to0\;\mbox{when}\;i\to\infty$$ and
therefore $(n_{k_i})$ is a mixing sequence for $T$.
\end{proof}

\begin{rem}\label{wmisgeneric}  It is easy to see that
Proposition~\ref{mixresidual} shows that weakly mixing transformations $\mathcal W$
are generic set because
they can be characterized as having only the trivial eigenvalue $1$ with
the eigenvectors being the constant functions.  Now taking $n_m = m$ for
all $m$, we have any transformation with a non-trivial eigenvalue
must be in  $\mathcal M(n_m)^c$.  So $\mathcal W^c \subset \mathcal M(n_m)^c$,
and $\mathcal M(n_m) \subset \mathcal W$.  Actually, it is also well-known
that $\mathcal W$ itself is a $G_\delta$ set.
\end{rem}

Combining our two basic category results, Proposition~\ref{folklore1} and
Proposition~\ref{mixresidual}, gives the following.

\begin{prop} \label{together} Given an increasing sequence $(n_m)$ of
natural
numbers, consider the set $\mathcal B_{(n_m)}$ that consists of all
$S\in Aut\xbm$
such that $S^{n_{m_k(1)}}\to Id$ weakly in $L_{2,0}(X,p)$, for some
subsequence $(n_{m_k(1)})$
of $(n_m)$ and
such that $S^{n_{m_k(2)}}\to 0$ weakly in $L_{2,0}(X,p)$, for some
subsequence $(n_{m_k(2)})$
of $(n_m)$.  Then $\mathcal B_{(n_m)}$ is a generic subset of
$Aut\xbm$.
\end{prop}

\begin{rem}\label{wmandrigid} The category result in Proposition~\ref{together}
also holds if we ask
for the stronger property that  $S^{\sigma}\to Id$ weakly in $L_{2,0}(X,p)$,
as $\sigma \to \infty(IP)$ for the IP set generated by some
subsequence $(n_{m_k})$ of $(n_m)$.  See Proposition~\ref{spectralIP} and the
discussion before it for the definition and basic characterization of IP rigidity.
\end{rem}

\begin{rem}  We can also formulate unitary versions of Proposition~\ref{folklore1},
Proposition~\ref{mixresidual}, and Proposition~\ref{together}.
\end{rem}

\subsection{Rigidity Specifically} \label{rigidonly}

In addition to the examples given inherently by Proposition~\ref{together}, each
ergodic transformation with discrete spectrum is rigid. One
can see this in several ways. One way is to note that $T$ is rigid
for each eigenfunction $f$ since if $\gamma \in \mathbb T$ there is
a sequence $(n_m)$ such that $\gamma^{n_m} \to 1$ in $\mathbb T$. Then
use the principle of Corollary~\ref{thouvenot}. Alternatively, in
order to see this via the Halmos-von Neumann Theorem, consider an ergodic
rotation $Tx=x+x_0$ where $X$ is a compact metric monothetic
group, $x_0$ is its topological cyclic generator, and $p$ stands
for Haar measure of $X$. Take any increasing sequence $(n_t)$ of
integers, and consider $(n_t\cdot x_0)$. By passing to a
subsequence if necessary, we can assume that $n_tx_0\to y\in X$.
This is equivalent to saying that $T^{n_t}\to S$, where $Sx=x+y$.
Because the convergence is taking part in the strong operator
topology, it is not hard to see that we will obtain
$$
T^{n_{t_{k+1}}-n_{t_k}}\to S\circ S^{-1}=Id,
$$
and therefore $T$ is rigid (indeed, $n_{t_{k+1}}-n_{t_k}\to\infty$ by
Proposition~\ref{sparse} below). These arguments show that
each purely atomic measure is a Dirichlet measure.
Moreover, we have also shown that in the discrete spectrum case
the closure of $\{T^n: n \in \mathbb Z\}$ in the strong
operator topology is compact.  The converse is also true.
See for example Bergelson and Rosenblatt~\cite{bergros} and
Ku\v{s}hnirenko~\cite{Ku}.  It also is not difficult to see that
the centralizer of $T$ in $Aut\xbm$, denoted by $C(T)$, can be identified with
this closure and so is compact in the strong operator topology.
The converse of this is also true (see again, e.g.\ \cite{Ku}).
Moreover, $T$ is isomorphic to the translation by $T$ on $C(T)$ considered
with Haar measure.
We will see
later that ergodic transformations with discrete spectrum are completely determined by their
rigidity sequences (see Corollary~\ref{AA2} below).  There we will be
using the information summarized here.

A positive finite Borel measure $\nu$ on $\T$ is called a {\em
Rajchman measure} if its Fourier transform vanishes at infinity,
that is
\beq\label{ww7} \widehat{\nu}(n)\to 0\;\;\mbox{when}\;
|n|\to \infty.\eeq
So the spectral measures of a strongly mixing transformation
are Rajchman measures, and $T$ is strongly mixing if and only if the
maximal spectral type $\nu^T$ is a
Rajchman measure. Moreover, by the {\em Gaussian measure space construction}
(GMC) discussed in Remark~\ref{GMC}, any
Rajchman measure is one of the spectral measures for some strongly mixing
transformation. It is not hard to
see that a measure absolutely continuous with respect to a
Rajchman measure is Rajchman. Also, Rajchman measures and 
Dirichlet measures are mutually singular.

\begin{rem} It would be interesting to characterize
the sets $\mathcal N = \{n_m\}$ of density one that occur in
Proposition~\ref{wmone}. This means
characterizing sets $\mathcal L$ of density zero that are the
complements of such sets. Characterizing rigidity
sequences for weakly mixing transformations means
characterizing certain types of sets $\mathcal L$.
However, this may not capture all sets in $\mathcal N$.
For example, it may be possible for a set $\mathcal N$
to fail to have a rigidity sequence in its complement,
but contain a set of the form $\{n\ge 1: |\widehat {\mu}(n)| \ge \delta \}$
for some $\delta > 0$, e.g. with $\mu$ that is a spectral measure
for a mildly mixing, not strongly mixing, transformation.
\end{rem}

Proposition~\ref{wmone} certainly shows that rigidity
sequences for weakly mixing transformations are density zero. 
Proposition~\ref{sparse} below shows also that more than this is
true without the assumption that $T$ is weakly mixing.
There is a general principle in play here, but the
argument has to be different when there are eigenfunctions. If the
system is not ergodic, or if some power $T^n$ is not ergodic, then
there can exist a non-zero, mean-zero function $f \in L_2(X,p)$ and
a periodic sequence $(n_m)$ such that $f\circ T^{n_m} = f$ for all
$m \ge 1$. Otherwise, the only way a sequence can exhibit
rigidity for a function, or for the whole dynamical system, is
when the sequence has gaps tending to $\infty$, and hence is
certainly of density zero.  We recall that our probability spaces
are standard Lebesgue spaces and so have no atoms.  This is important
in the next result where the Rokhlin Lemma is used.

\begin{prop}
\label{sparse} Let $(n_m)$ be an increasing sequence of integers.
\medskip

\noindent a) Let $T$ be totally ergodic. If $\|f_0\circ T^{n_m} -f_0
\|_2 \to 0$ as $m \to \infty$ for some non-zero, mean-zero $f_0 \in
L_2(X,p)$, then the sequence $(n_m)$ has gaps
tending to $\infty$ and hence has zero density.
\medskip

\noindent b) Suppose $T$ is ergodic. If $\|f\circ T^{n_m} -f \|_2
\to 0$ as $m \to \infty$ for all $f \in L_2(X,p)$, then $(n_m)$ has
gaps tending to $\infty$ and hence has zero density.
\end{prop}
\begin{proof} In a) we claim that $n_{m+1} - n_m
\to \infty$ as $m \to \infty$. Otherwise, there would be a value
$d \ge 1$ such that $d = n_{m+1} - n_m$ infinitely often. It
follows that $f_0\circ T^d = f_0$. This is not possible since
$f_0$ is non-zero and mean-zero, and $T$ is totally ergodic. To
prove b), one again argues that $n_{m+1} - n_m \to \infty$ as $m
\to \infty$ since otherwise there exists $d \ge 1$ such that $d =
n_{m+1} - n_m$ for infinitely many $m$, and hence $f \circ T^d =
f$ for all $f \in L_2(X,p)$. But this is impossible since our
system is ergodic. Indeed, for any $d_0$, using the Rokhlin Lemma,
there is a set $B$ of positive measure such that $T^jB$ are
pairwise disjoint for all $j, 1 \le j \le d_0$. Take $f_0$
supported on $B$ that is non-zero and mean-zero. Then $f_0\circ
T^j \not= f_0$ for all $j, 1 \le j \le d_0$. Hence, once $d_0 > d$,
we cannot have $f_0\circ T^d = f_0$.
\end{proof}

\begin{rem}  Consider part b) above in the case of suitable unitary
operators.  Since we used the Rokhlin Lemma, we would need a different
proof to show that a rigidity sequence for a unitary operator
has gaps tending to infinity.  This can be seen by the above if
the operator has an infinite discrete spectrum.  An
additional argument is needed in case all the non-trivial spectral
measures are continuous.  Then using the GMC (see Remark~\ref{GMC})
and the result in part a) gives
the result in this case too.
\end{rem}

We will be constructing various examples of rigidity sequences
in Section~\ref{methods}.  To have some contrast with these
constructions, it is worthwhile to make some remarks now about
sequences that cannot be rigidity sequences. We have seen from the
above, that rigidity sequences must have gaps growing to infinity.
But much more structural information is needed to guarantee that the
sequence can be a rigidity sequence.
For example, we have the following basic result.

\begin{prop} \label{unifdist}
Suppose $(n_m)$ is an increasing sequence such that
$(n_mx\mod 1)$
is uniformly distributed for all but a countable set of values $x \in \mathbb R`$. Then
$(n_m)$ cannot be a rigidity sequence for a weakly mixing transformation.
\end{prop}
\begin{proof} For any continuous measure $\nu$ on $\mathbb T$, we would
have $\int \frac 1M\sum\limits_{m=1}^M \gamma^{n_m} \,d\nu(\gamma) \to 0$ as
$M \to \infty$ because $\frac 1M\sum\limits_{m=1}^M \gamma^{n_m} \to 0$
as $M\to \infty$ for all but countable many $\gamma$. Hence, we cannot
have $\frac 1M\sum\limits_{m=1}^M \gamma^{n_m} \to 1$ in measure with
respect
to $\nu$ as $M \to \infty$.
\end{proof}

\begin{rem} \label{UDeg}  See Kuipers and
Niederreiter~\cite{KN} for information about uniform distribution of sequences.
For example, Vinogradov proved that the prime numbers $(p_m)$ in
increasing order satisfy the hypothesis in Proposition~\ref{unifdist}.
So the prime numbers cannot
be a rigidity sequence for a weakly mixing dynamical system. Of
course, the property in Proposition~\ref{unifdist} also shows that
they cannot be a rigidity sequence for an ergodic rotation
of $\mathbb T$.  Actually, they cannot be
a rigidity sequence for any ergodic transformation
with discrete spectrum on a Lebesgue space by Proposition~\ref{linform}.  However,
the property in Proposition~\ref{unifdist} cannot be used to argue this
because there are ergodic transformations with discrete spectrum
whose spectral measures are supported on the roots of unity.  Other examples
using Proposition~\ref{unifdist} include polynomial sequences $(p(m): m \ge 1)$, with $p$
a non-zero polynomial with integer coefficients.
However, we can reach the same conclusion for
polynomial sequences $(p(m): m \ge 1)$ by a simpler argument using successive
differences. See Remark~\ref{linformeg} b) below.
\end{rem}

The proof that certain sequences, like the prime numbers, satisfy the
hypothesis in Proposition~\ref{unifdist} is linked to another property
that prohibits rigidity.  First, we consider what happens when a linear
form on the sequence has bounded values.  Recall that our underlying
probability space is a standard Lebesgue space.

\begin{prop}\label{linform}  Suppose $F(x_1,\ldots,x_K) = \sum\limits_{k=1}^K c_kx_k$
where $c_1,\ldots,c_K \in \mathbb Z\backslash \{0\}$. Suppose $(n_m)$ is
a sequence of whole numbers.  Assume that for some non-zero $d\in \mathbb Z$,
we know that for any $M \ge 1$, there are
$m_k \ge M$ for all $k=1,\ldots,K$, such that
$F(n_{m_1},\ldots,n_{m_K}) = d$.  Then $(n_m)$ is not a rigidity sequence
for an ergodic transformation.
\end{prop}
\begin{proof}  Suppose $\nu$ is a Borel probability measure on $\mathbb T$
such that $\widehat {\nu}(n_m) \to 1$ as $m \to \infty$.  Hence, $\gamma^{n_m} \to 1$
in measure with respect to $\nu$ as $m\to \infty$.  Then also
$\gamma^d = \gamma^{F(n_{m_1},\ldots,n_{m_K})} \to 1$ in measure with respect
to $\nu$ as $m_k \to \infty$, for all $k=1,\ldots,K$.  That is, $\nu$ is
supported on the $d$-th roots of unity.  However, if $T$ is ergodic, we cannot have all
the spectral measures supported in the $d$-th roots of unity.
\end{proof}

\begin{rem}\label{linformeg} a) The proof is showing that given the
hypothesis of Proposition~\ref{linform}, $(n_m)$ cannot
be a rigidity sequence for a transformation $T$ and a specific function $f\in L_2(X,p)$
unless $f \circ T^d = f$.  So if $T$ is totally ergodic, $(n_m)$ cannot be
a rigidity sequence for any non-trivial function, let alone a rigidity sequence
for all functions.
\medskip

\noindent b)  It is easy to see that polynomial sequences $(p(n): n \ge 1)$ satisfy
the hypothesis of Proposition~\ref{linform} and therefore cannot be rigidity
sequences for weakly mixing transformations.  The easiest way to see this is to
note that the difference $q(n) = p(n+1)-p(n)$ is a polynomial of less degree than $p$.
So successive differences will eventually lead to a constant.  For example,
if $p(n) = n^2$, then let $q(n) = p(n+1) - p(n)$.  Then $q(n+1)- q(n) = 2$.
So $p(n+2) - 2p(n+1) + p(n) = p(n+2)-p(n+1) - (p(n+1) - p(n)) = q(n+1) - q(n) = 2$.  So
let $F(x,y,z) = x - 2y +z$.  Then for all (large) $n$, $F(p(n+2),p(n+1),p(n)) = 2$.
Hence, $(n^2)$ cannot be a rigidity sequence for a weakly mixing transformation.
\medskip

\noindent c)  Here are other simple examples of the above.  Suppose $n_m= 2^m+1$.
Then $2n_m - n_{m+1} = 1$.  So if $F(x,y) = 2x-y$, then
$F(n_m,n_{m+1}) =1$ for all $m$.  Also, suppose $n_m = 2^m+m$.
Then $2n_m - n_{m+1} = m -1$.  Hence, $2n_{m+1} - n_{m+2} - (2n_m - n_{m+1}) = 1$
for all $m$.  So with $F(x,y,z) = 2y-z-(2x - y) = 3y - z -2x$, we have
$F(n_m,n_{m+1},n_{m+2}) = 1$ for all $m$.
Therefore, both $(2^n +1)$ and $(2^n+n)$ are not
rigidity sequences for a weakly mixing transformation.  It is not hard to see that
a calculation of this sort can also be carried out for any sequence $(2^n+p(n))$ where
$p$ is a non-zero polynomial with integer coefficients.  Hence, such sequences
are generally not rigidity sequences for weakly mixing transformations.
See Proposition~\ref{integerratios} which shows that $(2^m)$ itself
is a rigidity sequence for a weakly mixing transformation.
\medskip

\noindent d) Another example will give some idea of other issues that can arise
in using Proposition~\ref{linform}.  Take the sequence $(2^n+p_n)$ where $(p_n)$
is the prime numbers in increasing order.  If this is a rigidity sequence, then
so is $(a_n) = (p_{n+1} - 2p_n)$ since $a_n = 2^{n+1} + p_{n+1} - 2(2^n +p_n)$.
It is not clear what holds for this resulting
sequence.  For example, does this $(a_n)$ satisfy the hypothesis in Proposition~\ref{unifdist}?
\medskip

\noindent e) Proposition~\ref{linform} can also be used in a positive way.
For example, start with the fact that $(2^n)$ is a rigidity sequence for a weakly
mixing transformation proved in Proposition~\ref{integerratios}.  It follows
that for any $m_1,\ldots,m_K$ and $N_1,\ldots,N_K$, the sequence $(n_m) =
(\sum\limits_{k=1}^K m_k2^{N_k+m})$ is also a rigidity sequence.
\end{rem}

Proposition~\ref{linform} in turn can be used to prove the following result.
This result was suggested by the fact that certain sequences, like the prime
numbers, were originally seen to satisfy the hypothesis of Proposition~\ref{unifdist}
by using analytic number theory arguments which also gave the hypothesis
of Proposition~\ref{sumset}.  We will use in the next proposition the notion
of upper density: given a set $A = \{a_n:n \ge 1\}$ of  integers, the
upper density of $A$ is
\[\limsup\limits_{N\to \infty} \frac {\#(\{a_n: n \ge 1\}\cap \{-N,\ldots,N\})}{2N+1}.\]

\begin{prop}\label{sumset}  Suppose that $\mathbf a$ is a sequence of whole numbers.  Assume that
the set $A = \{a_n:n\ge 1\}$ has the property
that for some integers $c_1,\ldots,c_L$ the set of sums $c_1A+\ldots+c_LA =
\{\sum\limits_{l=1}^L c_la_{n_l}: n_1,\ldots,n_L \in A\}$ has
positive upper density.  Then $\mathbf a$ cannot be a rigidity sequence for an
ergodic transformation.
\end{prop}
\begin{proof}
Suppose  that the condition on the set of sums holds and $L$ is the smallest possible value
for which this holds for some $c_1,\ldots,c_L$.  Let $B = c_1A+\ldots+c_LA$ and
assume it has upper density  $D > 0$.
Consider the set of sums $B_N =\{\sum\limits_{l=1}^L c_l a_{n(l)}:
n(1),\ldots,n(L) \ge N\}$.  $B_N$ can be obtained from $B$ by deleting a finite set and
a finite number of translates
of sets of sums  of the form $c_{i_1}A+\ldots+c_{i_{L^\prime}}A$ with $L^\prime < L$.   Since these sets of sums
are assumed to all be of upper density zero, $B_N$ also has upper density $D$.
Then there are infinitely many pairs $\sigma_1 < \sigma_2$ with $\sigma_1,\sigma_2 \in B_N$
and $\sigma_2 = d + \sigma_1$ for some non-zero $d \le 2\frac 1D$.  So let $K = 2L$ and $F(x_1,\ldots,x_K) =
\sum\limits_{k=1}^Lc_kx_k - \sum\limits_{k=L+1}^Kc_{k-L}x_k$.  With this linear form $F$, we have shown
that $A$ satisfies the hypothesis of Proposition~\ref{linform}.
\end{proof}

\begin{rem} \label{sumseteg}
Proposition~\ref{sumset}
 clearly applies to squares $S = \{n^2: n \ge 0\}$.  Indeed, $(n+1)^2 - n^2 = 2n+1$,
 so the odd numbers are a subset of $S - S$. It also applies to the prime numbers $P$,
because of the well-known fact that for some whole number $K$, the sum of $K$
copies of $P$ contains all of the whole numbers $n \ge 3$.
\end{rem}
\bigskip

By using spectral measures and the {\it Gaussian measure space
construction} (denoted here by GMC), we can see that our basic
desire to characterize rigidity sequences
is equivalent to a fact about Fourier transforms of
measures.

\begin{rem}\label{GMC} The GMC is a standard method of creating a weakly mixing transformation $G_{\nu}$
such that one of its spectral measures is a given continuous measure $\nu$.
The transformation $G_{\nu}$ itself is a coordinate shift on an infinite product space,
but the probability measure on the product space that it leaves invariant must be constructed
specifically with $\nu$ in mind.  See Cornfeld, Fomin, and Sinai~\cite{CFS} for details
about the GMC. We will use the notation $G_{\nu}$ for the transformation
obtained by applying GMC to the positive Borel measure $\nu$. One
can actually see that this method applies to all measures by using
complex scalars, but it is traditional to apply it to symmetric
measures whose Fourier transform is real-valued so that GMC gives
a real centered stationary Gaussian process. In all of our
applications, we can symmetrize the measures $\nu$ that we construct
by replacing them with $\nu_s = \nu\star\nu^*$, or with $\nu_s = \nu +
\nu^*$. This
will allow us to use the GMC in its traditional form, while still
preserving the properties that we need the GMC to give us.
\end{rem}

\begin{rem}\label{Poisson}  We will also have occasion to use another
general method, that of Poisson suspensions.   See
Cornfeld, Fomin, and Sinai~\cite{CFS},
Kingman~\cite{Poisson}, and
Neretin~\cite{Poisson1}
for information about Poisson suspensions.
Here is briefly the idea.
Let $T$ be a transformation of a standard Lebesgue space $(X,{\mathcal
B},\mu)$, where $\mu$ is a $\sigma$-finite, infinite positive measure.  We define a
probability space $(\widetilde{X},\widetilde{\mathcal B},
\widetilde{\mu})$. The points of the configuration space
$\widetilde{X}$ are infinite countable subsets
$\widetilde{x}=\{x_n:\:n\geq1\}$ of $X$. Given a set $A\in{\mathcal
B}$ of finite measure we define $N_A:\widetilde{X}\to\N\cup \{\infty\}$ by
setting
$$
N_A(\widetilde{x})=\#\{n\in \N:\; x_n\in A\}.
$$
Then $\widetilde{\cb}$ is defined as the smallest $\sigma$-algebra
of subsets of $\widetilde{X}$ making all variables $N_A$,
$\mu(A)<+\infty$, measurable. The measure $\widetilde{\mu}$ is the
only probability measure (see \cite{Poisson} for details) such
that \begin{itemize}\item the variables $N_A$ satisfy the Poisson
law with parameter $\mu(A)$;
\item
for each family $A_1,\ldots,A_k$ of pairwise disjoint subsets of
$X$ of finite measure the corresponding variables
$N_{A_1},\ldots,N_{A_k}$ are independent.
\end{itemize}
The space $(\widetilde{X},\widetilde{\mathcal B}, \widetilde{\mu})$ is
a standard Lebesgue probability space. Then, we define
$\widetilde{T}$ on $\widetilde{X}$ by setting
$$\widetilde{T}(\{x_n\})= (\{Tx_n\})$$ and obtain a transformation
of  $(\widetilde{X},\widetilde{\mathcal B}, \widetilde{\mu})$ which is
called the Poisson suspension of $T$. Then $\widetilde{T}$ is
ergodic if and only if $T$ has no non-trivial invariant sets of
finite measure. In this case, $\widetilde{T}$ turns out to be
weakly mixing and moreover $\widetilde{T}$ is spectrally
isomorphic to the GMC transformation $G$ given by the unitary
operator $U_T$ acting on $L_2(X,{\mathcal B},\mu)$, i.e.\ the unitary
operators $U_{\widetilde{T}}$ and $U_G$ are equivalent.
\end{rem}

\begin{prop} \label{spectral}
The sequence $(n_m)$ is a rigidity sequence for some
weakly mixing dynamical system if and only if there is a
continuous Borel probability measure $\nu$ on $\mathbb T$ such
that $\lim\limits_{m \to \infty} \widehat {\nu}(n_m)= 1$.
\end{prop}
\begin{proof}
First, if $f \in L_2(X,p)$ is norm one, then rigidity along the
sequence $(n_m)$ for $f$ means that we will have $\lim\limits_{m \to
\infty} \widehat {\nu_f^T}(n_m) = 1$. So when $T$ is weakly
mixing and $f$ is mean-zero, then $\nu_f^T$ is continuous.

Conversely, if there is a continuous Borel probability measure
$\nu$ on $\mathbb T$ such that $\widehat {\nu}(n_m) \to 1$ as $m
\to \infty$, then the GMC gives us a weakly mixing
dynamical system $(X,\mathcal B,p, T)$, with $T = G_{\nu}$, and a mean-zero function $f \in
L_2(X,p)$ with $\|f\|_2^2 = 1$ and $\widehat {\nu}(n) = \langle
f\circ T^{n},f\rangle$ for all $n$. Indeed, it is not hard to see
that this construction gives us a weakly mixing dynamical system
such that $T$ is rigid along the sequences $(n_m)$.
\end{proof}

\begin{rem}  The second part of this proof is saying that
the GMC preserves rigidity. In particular, if $\nu=\nu^S$
for some $S\in Aut(Y,\mathcal B_Y,p_Y)$ then $S$ and $T:=G_\nu$ have
the same rigidity sequences.
\end{rem}

\begin{rem}  The unitary version of this result is that there is a
continuous Borel probability measure $\nu$ on $\mathbb T$ such that
$\lim\limits_{m \to \infty} \widehat {\nu}(n_m) = 1$ if and only
if there is a unitary operator $U$ on a Hilbert space $H$ such
that 1) for no non-zero $v \in H$ is the orbit $\{U^kv: k \mathbb Z\}$
precompact in the strong operator topology, and
2) for all $v \in H$, $\lim\limits_{m\to \infty} \|U^{n_m} v -v \|_H = 0$.
The first property here 1) is saying that $U$ is called a {\em weakly mixing unitary
operator}.  See Bergelson and Rosenblatt~\cite{bergros}.
\end{rem}

The result in Proposition~\ref{spectral} holds
with appropriate changes if we ask for the
stronger property that we have rigidity along the IP
set generated by a sequence. Some notation is useful for
understanding this. The {\em IP set generated by a sequence} consists
of all finite sums of elements with distinct indices in the sequence. So the
notation $\Sigma = FS(n_m)$ for this IP set makes sense. Here
given
$m_1 < \ldots < m_k$, let $\sigma
= \sigma(n_{m_1},\ldots,n_{m_k}) = n_{m_1}+\ldots+n_{m_k}$. We say
$\sigma \to \infty (IP)$ if $\sigma = m_1$ tends to $\infty$.
Also, given the dynamical system, we say that $T$ is {\em
IP-rigid} along $\Sigma$ if $T^{\sigma} \to Id$ in the strong
operator topology as $\sigma \to \infty (IP)$. With this
understood, it is easy to see the following.

\begin{prop}\label{spectralIP}  There is a weakly mixing
dynamical system that is IP-rigid along $\Sigma$ if and only if there
is a continuous Borel probability measure $\nu$ on $\mathbb T$
such that $\widehat {\nu}(\sigma) \to 1$ as $\sigma \in FS(n_m: m
\ge 1)$ tends to $\infty$ (IP).
\end{prop}

There is another way of phrasing the Fourier transform condition
above for rigidity, and IP-rigidity, that is very useful. We have
already pointed this out in terms of spectral measures in Proposition
~\ref{rigidfacts}. We leave
the routine proof to the reader.

\begin{prop} \label{inmeasure} Given a sequence $(n_m)$ and a
positive Borel measure on $\mathbb T$, we have $\widehat
{\nu}(n_m)$ tends to $\nu(\mathbb T)$ if and only if $\gamma^{n_m} \to
1$ in measure with respect to $\nu$ as $m \to \infty$. Also,
we have $\widehat {\nu}(\sigma)$ tends to $1$ as $\sigma \to \infty$ (IP)
if and only if $\gamma^{\sigma}\to 1$ in measure
with respect to $\nu$ as $\sigma \to \infty$ (IP).
\end{prop}

\begin{rem}\label{moreIP} a) The Fourier transform characterizations of rigidity sequences
above suggests that we might be able to take this further by
finding the correct growth/sparsity condition on a strictly
increasing sequence $(n_m)$ to be a rigidity sequence or
IP-rigidity sequence for a weakly mixing dynamical system. First,
consider the property of IP-rigidity. It is easy to see, from the
spectral measure characterization of IP-rigidity above, that it is
sufficient to have a criterion that guarantees there is an
uncountable Borel set of points $K \subset \mathbb T$ such that
for all $\gamma \in K$, we have $\gamma^{\sigma}\to 1$ as $\sigma
\to \infty (IP)$. See the beginning of the proof of Proposition~\ref{ABCD}
where the same point is made.  Let us consider this in the parametrization of
$\mathbb T$ where $\gamma = \exp(2\pi i x)$ for $x \in [0,1)$. We
also denote by $K$ the set of $x$ corresponding to $\exp(2\pi ix)
\in K$. Our pointwise criterion then means that for all $x \in K$,
we have both $\cos(2\pi \sigma x) \to 1$ and $\sin(2 \pi \sigma x)
\to 0$ as $\sigma \to \infty\,(IP)$. It is enough to just have
$\sin(2 \pi \sigma x) \to 0$ as $\sigma \to \infty\,(IP)$. Indeed,
with the usual notation that
$\{z\} = z - \lfloor z \rfloor$ is the fractional
part of a real number $z$, the $\{y = \{2x\}: x \in K\}$ will give
an uncountable set of values $y$ such that both $\cos(2\pi \sigma
y) \to 1$ and $\sin(2 \pi \sigma y) \to 0$ as $\sigma \to
\infty\,(IP)$
\medskip

\noindent b) The pointwise spectral property above is equivalent to
having rigidity along $\Sigma$ for ALL functions whose spectral
measure in a given dynamical system is supported in $K$. This is
stronger than what is needed for rigidity along $\Sigma$ for some
weakly mixing dynamical system. This weaker notion is equivalent
to having a continuous positive measure $\nu$ supported on $K$
such that $\exp(2\pi i\sigma x) \to 1$ in measure with respect to
$\nu$ as $\sigma \to \infty\, (IP)$.
\end{rem}

\section{\bf Constructions of Rigid Sequences for Weakly Mixing
Transformations}
\label{methods}

We give a number of different approaches here for constructing weakly
mixing transformations that have a specific type of sequence as a rigidity
sequence. These methods are sometimes overlapping, but the different
approaches give us insights into the issues nonetheless. Also, there
are a variety of number theoretic and harmonic analysis connections
with some of these methods; these are also explored in this section.

When this paper was in the final draft, we learned of the work of
Eisner and Grivaux~\cite{EG}.  Our papers are largely complementary,
although we do cover some of the same basic issues.
We will cite their work more in place later in this section.

Besides the question of the structure of rigidity sequences in general, we would
like to be able to answer the following questions:
\medskip

\noindent{\bf Questions}: Which rigidity sequences of an ergodic
transformation with discrete spectrum can be rigidity sequences
for weakly mixing transformation?  Which rigidity sequences for
weakly mixing transformations can be rigidity sequences for an
ergodic transformation with discrete spectrum?
\medskip

\begin{rem}\label{whyquestion}  a) At this time, we do not know
if there is a counterexample to either of the questions above.
\medskip

\noindent b) As discussed later in this section,
one viewpoint to answering these questions is to consider, for
fixed $(n_m)$, the group $\mathcal R(n_m) =
\{\gamma \in \mathbb T: \lim\limits_{m\to \infty} \gamma^{n_m} =1\}$.
We will see that if $\mathcal R(n_m)$ is uncountable, then $(n_m)$ is a rigidity
sequence for an ergodic rotation of $\mathbb T$ and for some weakly
mixing transformation.  If $\mathcal R(n_m)$ is countably infinite,
then there is an ergodic transformation $T$ with discrete spectrum such that $(n_m)$
is a rigidity sequence for $T$, but we do not know in general if there
is a weakly mixing transformation with $(n_m)$ as a rigidity sequence.
For example, sequences
like $(2^n)$ cannot be a rigidity sequence for an ergodic
rotation of the circle, but can be a rigidity sequence for a weakly
mixing transformation.  This does not make this sequence a
counterexample to the second question above because this sequence is a rigidity sequence for
the ergodic generator of another compact abelian group $G$.
Just take $G$ to be the inverse limit of the finite groups
$\mathcal R_n = \{\gamma \in \mathbb T: \gamma^{2^n} = 1\}$.
\medskip

\noindent c) Our results, in particular Proposition~\ref{integerratios} or Proposition~\ref{intlac},
and the structure of subgroups of $\mathbb T$ in the discrete topology,
show that all rigidity sequences for transformations
with discrete spectrum are rigidity sequences for some weakly mixing transformation
if the following is true: given an element $\gamma \in \mathbb T$
of infinite order and a sequence $(n_m)$ such that $\gamma^{n_m} \to 1$
as $m \to \infty$, the sequence $(n_m)$ is a rigidity sequence for some
weakly mixing transformations.
\medskip

\noindent d) If $\mathcal R(n_m)$ is finite, then no ergodic
transformation with discrete spectrum has this sequence as a rigidity
sequence, but it might be possible to construct weakly mixing transformations
with $(n_m)$ as a rigidity sequence.  This is not at all clear yet.  A good
example of a candidate sequence for this case is $n_m = 2^m+3^m$ for
which $\mathcal R(n_m) = \{1\}$.  To see this, note that
$n_{m+1} - 2n_m = 3^m$ and $3n_m-n_{m+1} = 2^m$.  So any $\gamma \in \mathcal R(n_m)$
must have both $\gamma^{3^m}$ and $\gamma^{2^m}$ tending to $1$
as $m \to \infty$.  Hence, $\gamma$ is simultaneously a root of unity
for a power of $2$ and a power of $3$, and hence $\gamma = 1$.
\end{rem}
\medskip

A sequence being just density zero is clearly not
enough for rigidity for an ergodic transformation
because a sequence can be density zero and
have infinitely many pairs of terms $(n_m,n_{m+1})$ with, say,  $n_{m+1}
- n_m \le 10$. More sparsity is needed than just density zero. In this
direction, lacunary sequences might seem to be good candidates to be
rigidity sequences for some weakly mixing transformations because
they certainly have the necessary sparseness. We take lacunarity
here to be as usual: $(n_m)$ is lacunary if and only if there exists
$\rho >1$ such that $n_{m+1}/n_m \ge \rho$ for all $m \ge 1$.
However, even in this class of sequences the situation is not
clear as there are lacunary sequences which cannot even be rigidity
sequences for any ergodic transformations.  For example, $(2^n+1)$
is lacunary but cannot be a rigidity sequence
even for an ergodic transformation with discrete spectrum.
See Remark~\ref{linformeg} c).  This answers the question in
Eisner and Grivaux~\cite{EG}, p. 5.

\subsection{\bf Diophantine approach}
\label{diophantine}

\subsubsection{\bf Existence of Supports}\label{existenceofsupports}

Denote by $[x]$ the nearest integer to $x \in \mathbb R$, choosing
$\lfloor x \rfloor$ if $\{x\} = 1/2$. So $|x - [x]|$ would be the
distance of $x$ to $\mathbb Z$. We will denote this by $\|x\|$.
The distance $\|x\| = \{x\}$ if $\{x\} \le \frac 12$ and $\|x\| =
1 - \{x\}$ if $\frac 12 \le \{x\}$. We have the following result.

\begin{prop} \label{ABCD}  The following are equivalent:

\noindent a) there exists some
infinite perfect compact set $K \subset [0,1)$,  such that
$\sin(2\pi\sigma x) \to 0$ as $\sigma \to \infty
(IP)$ for all $x \in K$,

\noindent b) for some uncountable set of $x$ values, we have
\begin{equation}\label{ABS}
\sum\limits_{m=1}^\infty \|n_mx\| < \infty.
\end{equation}
\end{prop}
\begin{proof}
The condition $\sum\limits_{m=1}^\infty \|n_mx\| < \infty$ describes
a Borel set of values $x$.  Hence, the set theoretic aspects of this proposition work
because any uncountable Borel set in $[0,1)$ contains an
infinite perfect compact subset. See Sierpi\'nski~\cite{Sierpinski},
p. 228.

Assume that
$\sum\limits_{m=1}^\infty \|n_mx\| < \infty$
for an uncountable set of $x$ values.
Consider separately the
values of $\|n_mx\|$ where it is $\{n_mx\}$ or it is $1 - \{n_mx\}$.
In the first case, as $m \to \infty$, $0\le \sin(2\pi n_mx) \sim
2\pi \{n_mx\} = 2\pi\|n_mx\|$.
In the second case, as $m \to \infty$, $0\ge \sin(2\pi n_mx) \sim
-2\pi (1 -\{n_mx\}) = -2\pi\|n_mx\|$.
So Equation\,(\ref{ABS}) implies that
$\sum\limits_{m=1}^\infty \sin(2\pi n_mx)$ converges absolutely
for an uncountable set. Using the formula
\[\sin(\alpha+\beta) = \sin(\alpha)\cos(\beta)+\cos(\alpha)\sin(\beta)\]
repeatedly, we see that $\sin(2\pi x\sigma) = \sum\limits_{j=1}^k
f_j \sin(2\pi x n_{m_j})$ with coefficients $f_j$ that
suitable products of cosines. Here $|f_j| \le 1$. So the
convergence in Equation\,(\ref{ABS}) tells us that for an
uncountable set we have $\sin(2 \pi \sigma x) \to 0$ as $\sigma
\to \infty\ (IP)$.

Conversely, suppose we have for an uncountable set of $x$ such that
$\sin(2 \pi
\sigma x) \to 0$ as $\sigma \to \infty\,(IP)$. By doubling the $x$
values, we see that this means that for an uncountable set $K$, if
$x \in K$, and $\epsilon > 0$, we can choose $M_{\epsilon} \ge 1$
such that for all finite sets $F$ of whole numbers, all no smaller
than $M_{\epsilon}$, we have $ \|\sum\limits_{m \in F} n_mx \|
\le \epsilon$. In particular, for all $x\in K$, we have $\|n_mx\|
\to 0$ as $m \to \infty$. It follows that there is an uncountable
set $K_0 \subset K$ and some $M_0$ such that for all $x \in K_0$
and all finite sets $F$ of whole numbers all no smaller than
$M_0$, we have $ \| \sum\limits_{m\in F} n_mx \| \le \frac
1{100}$. In particular, if $m \ge M_0$, $\| n_mx_0\| \le \frac
1{100}$. Now suppose $\sum\limits_{m=1}^\infty \| n_mx_0\| =
\infty$ for some $x_0 \in K_0$. Then also
$\sum\limits_{m=M_0}^\infty \| n_mx_0\| = \infty$. Each
$\|n_mx_0\|$ is either $\{n_mx_0\}$ or it is $1 - \{n_mx_0\}$.
Say $I_1$ is the set of $m \ge M_0$ where the first formula holds,
and $I_2$ is the set of $m \ge M_0$ where the second formula
holds. Then either $\sum\limits_{m\in I_1} \|n_mx_0\| = \infty$
or $\sum\limits_{m\in I_2} \|n_mx_0\| = \infty$. Assume it is the
first case. Then by an upcrossing argument, we can choose a
finite set $I \subset I_1$ such that $\sum\limits_{m\in I}\|
n_mx_0 \|\in (\frac 18,\frac 38)$. But then $\sin(2\pi
\sum\limits_{m \in I} n_mx_0) = \sin(2\pi \sum\limits_{m \in I} \|
n_mx_0 \|) \ge \frac 1{\sqrt 2}$. This is not possible for $x\in
K_0$ because we know that $ \|\sum\limits_{m\in I} n_m x_0 \|$ is
small and so $|\sin(2\pi \sum\limits_{m\in I} n_mx_0)| =
|\sin(2\pi \| \sum\limits_{m\in I}n_mx_0 \|)| \le 2\pi
\|\sum\limits_{m\in I} n_mx_0 \| \le 2\pi \frac 1{100}$. In the
second case, again by an upcrossing argument, we can choose a
finite set $I \subset I_2$ such that $\sum\limits_{m\in I}\|
n_mx_0 \|\in (\frac 18,\frac 38)$. So then $\sin(2\pi
\sum\limits_{m \in I} n_mx_0) = -\sin(2\pi \sum\limits_{m \in I}
\| n_mx_0 \|) \le -\frac 1{\sqrt 2}$. Again this is not possible
for $x\in K_0$ because we know that $ \|\sum\limits_{m\in I} n_m
x_0 \|$ is small and so $|\sin(2\pi \sum\limits_{m\in I} n_mx_0)|
= |\sin(2\pi \| \sum\limits_{m\in I}n_mx_0 \|)| \le 2\pi
\|\sum\limits_{m\in I} n_mx_0 \| \le 2\pi \frac 1{100}$.
\end{proof}

This result, and Remark~\ref{moreIP} a) at the end of Section~\ref{rigidity}, give
this basic result.

\begin{prop}\label{OKforIP} If there exists an uncountable set of
$x$ values such that $\sum\limits_{m=1}^\infty \|n_mx\| < \infty$,
then there is a weakly mixing transformation that is rigid along the
$IP$ set generated by $(n_m)$.
\end{prop}

\begin{rem}\label{IPinfo}
It is not clear what growth property for $(n_m)$ corresponds to
Equation\,(\ref{ABS}) holding for an uncountable set of points. At
least this analysis shows that we can use results from
Erd\H{o}s and Taylor~\cite{ET}. In particular, they show that a
sufficient condition for the absolute convergence on an
uncountable set that we need is that $\sum\limits_{m=1}^\infty
n_m/n_{m+1}$ converges. See the beginning of the proof of
Proposition~\ref{fastworks}. Sometimes much less of a growth
condition is needed. For example, if one knows $n_{m+1}/n_m$
tends to infinity and $n_{m+1}/n_m$ is eventually a whole number,
then we have $\sum\limits_{m=1}^\infty \|n_mx\|$
converging for an uncountable set of $x$ values.
However, Erd\H{o}s and Taylor~\cite{ET} also give an example where $\lim\limits_{m\to \infty}
n_{m+1}/n_m = \infty$, but the ratio is infinitely often not a whole number, and yet
$\sum\limits_{m=1}^\infty \|n_mx\|$ converges only
for a countable set of values.
\end{rem}

Erd\H{o}s and Taylor~\cite{ET} observe that an earlier result
of Eggleston~\cite{Egg} is relevant here. Eggleston~\cite{Egg}
showed that it is necessary to have some hypothesis on $(n_m)$
that prevents $n_{m+1}/n_m$ from being bounded because
if these ratios are bounded then one
has $\exp(2\pi i n_mx) \to 1$ as $m \to \infty$ for at most a
countable set of values $x$. This fact is related to the weaker
version of the result above that is worth observing here, the
pointwise criterion needed for rigidity along the sequence itself.
The question is: what growth condition on a strictly increasing
sequence $(n_m)$ is needed for the sequence to admit a weakly
mixing dynamical system for which there is a non-trivial rigid
function along $(n_m)$? The same method as above, using the GMC,
shows that it is sufficient to know when there is an uncountable
set of points $K \subset [0,1]$ such that for all $x \in K$, we
have $\exp(2\pi in_mx) \to 1$ as $m \to \infty$. Hence, this
property can be characterized by having
\begin{equation}\label{termwise}
\lim\limits_{m \to \infty} \|n_mx\| = 0
\end{equation}
for an uncountable set of $x$ values. As with convergence along
an IP set, this property is stronger than what is needed to
produce just one weakly mixing dynamical system with rigidity
along $(n_m)$. This property guarantees that any dynamical
system, weakly mixing or not, whose non-trivial spectral measures
are supported on a subset of $K$, would have $(n_m)$ as a rigidity sequence.

These characterizations,
Equation\,(\ref{termwise}) and Equation\,(\ref{ABS}), show
the difference between having rigidity along a sequence
versus having rigidity along the IP set that the sequence
generates for all spectral measures supported in the set.
Eggleston shows the following in ~\cite{Egg}.

\begin{prop}\label{ratiogrows}
If $\lim\limits_{m\to \infty} n_{m+1}/n_m = \infty$, then
Equation\,(\ref{termwise}) holds for an uncountable set. Hence,
if $\lim\limits_{m\to \infty} n_{m+1}/n_m = \infty$, then $(n_m)$
is a rigidity sequence for some weakly mixing transformation.
\end{prop}

\begin{rem}\label{boundedornot}
It is not necessary to have $\lim\limits_{m\to \infty} n_{m+1}/n_m = \infty$
for Equation\,(\ref{termwise}) to hold on an uncountable set.
Indeed, it is not hard to construct examples of
strictly increasing sequences $(n_m)$ with Equation\,(\ref{termwise})
holding on an uncountable set, and yet there are
arbitrarily long pairwise disjoint blocks $B_k \subset \mathbb N$
such that $n_{m+1}/n_m = 2$
for all $m \in B_k$ and all $k$.  However, as commented above,
Eggleston also shows in ~\cite{Egg}
that for Equation\,(\ref{termwise}) to hold on an uncountable set,
it is certainly necessary to know that the ratios $n_{m+1}/n_m$ are
not uniformly bounded.  Also, Eisner and Grivaux~\cite{EG} show in Proposition 3.8
that one can weaken the hypothesis of Proposition~\ref{ratiogrows}
to just $\limsup\limits_{m\to\infty} \frac {n_{m+1}}{n_m}= \infty$
if $\frac {n_{m+1}}{n_m}$ is always a whole number.  This also follows
from using Proposition~\ref{integerratios}.
\end{rem}

\begin{rem}\label{egdisjoint}  a) Proposition~\ref{ratiogrows} allows us
to give interesting examples of disjoint weakly mixing dynamical systems with common
rigidity sequences.  Here disjointness is the standard disjointness from Furstenberg~\cite{Furstdisjoint},
their product is their only non-trivial joining.  This property means they also do not have
any common factors, and so of course are not isomorphic.  We construct weakly mixing $T$ and $S$ as follows.
Let $(a_m) = (2^{m^2})$, and let $(b_m)$ be the sequence which is $2^{m^2}$
for even $m$ and $2^{m^2}+1$ for odd $m$.  Since $a_{m+1}/a_m \to \infty$ and
$b_{m+1}/b_m \to \infty$ as $m \to \infty$, by Proposition~\ref{ratiogrows}
there exists $T$ which is weakly mixing and rigid along $(a_m)$ and $S$ which is weakly
mixing and rigid along $(b_m)$.  We see that $T$ and $S$ are rigid along the sequence
$(2^{(2m)^2})$.  Now take  $\nu^T$ and $\nu^S$ to be the maximal spectral
types of $T$ and $S$ on $L_{2,0}(X,p)$.  If we show that $\nu^T$ and $\nu^S$ are
mutually singular, then by Hahn and Parry~\cite{HP}, $T$ and
$S$ are disjoint.  But if $\omega \ll \nu^T$ and $\omega \ll \nu^S$,
we have $\omega$ rigid along both $(2^{(2m+1)^2})$ and $(2^{(2m+1)^2}+1)$.  It follows
that $\omega$ would have to be concentrated at $\{1\}$, which means $\omega = 0$
because $\nu^T(\{1\}) = \nu^S(\{1\}) = 0$.
\medskip

\noindent b)  Given ergodic transformations $T$ and $S$, with $T$ rigid along $(a_m)$ and
$S$ rigid along $(b_m)$, such that $b_m = a_m + p(m)$ for a non-zero polynomial, one can
argue in the style above, by taking successive differences, that the only spectral overlap of $T$ and $S$ can be with
eigenvalues that are $d$-th roots of unity where $d = p(0)$.  So if either $T$ or $S$
is totally ergodic (or even say weakly mixing), then $T$ and $S$ are disjoint.
\end{rem}

We want to make some general observations about the values of $x$ such
that $\|n_mx\| \to 0$ as $m\to \infty$. Alternatively, consider this set
in its
representation in $\mathbb T$; we are then looking for all $\gamma \in
\mathbb T$
such that $\gamma^{n_m} \to 1$ as $m \to \infty$. The first important point
is that this is a subgroup of $\mathbb T$, which we have denoted by
$\mathcal R(n_m)$. It
is easy to see that it is a Borel set, indeed it is clearly an
$\mathcal F_{\delta\sigma}$ because
\[\mathcal R(n_m) = \bigcap\limits_{k=1}^\infty \bigcup\limits_{M=1}^\infty
\bigcap\limits_{m=M}^\infty \{\gamma: |\gamma^{n_m} -1| \le 1/k\}.\]

There is quite a bit of literature
about such subgroups, and there is some interesting descriptive set
theory involved
in the study of the structure of this set too. First, consider this set
in the situation
that the $n_m = q_m$ are the denominators $q_m$ of the convergents
$\frac {p_m}{q_m}$
of the continued fraction
expansion of some fixed $\alpha \in [0,1]$, $\alpha$ irrational. Sometimes the ratios $\frac
{q_{m+1}}{q_m}$ are bounded.
In this case Larcher~\cite{Larcher} showed that $\mathcal R(n_m)$ is
just $\mathbb Z\alpha +
\mathbb Z$ i.e. in $\mathbb T$, we have $\mathcal R(n_m)$ just the
circle group generated
by $\exp(2\pi i\alpha)$. See also Kraaikamp and Liardet~\cite{Kr-Li} who
discuss issues of
speed of approach of $\|q_m\alpha\|$ to $0$. See also Host, M\'ela, and
Parreau~\cite{Ho-Me-Pa}
where this subgroup is considered extensively in the context of spectral
analysis of
dynamical systems.

Also, one can reverse the question of the structure of subgroups
$\mathcal R(n_m)$, by asking
which subgroups of $\mathbb T$ can be realized as such subgroups. It is
generally known that
there are subgroups of the circle that are not even Lebesgue measurable,
let alone Borel measurable.
In communication with S. Solecki, we learned of references that are very
thorough in evaluating
the descriptive set theoretic structure of subgroups of the circle. For
example, see the
articles by Klee~\cite{Klee}, Mauldin~\cite{Mauldin},
Solecki~\cite{Solecki1}, and Farah and Solecki~\cite{Farah-Solecki}.
In addition, Solecki~\cite{Solecki2}
has pointed out that there are even subgroups that are $\mathcal
F_{\delta\sigma}$ sets
which are not of the form $\mathcal R(n_m)$ for some $(n_m)$. Moreover,
subgroups like
$\mathcal R(n_m)$ are {\em Polishable} (see Farah and
Solecki~\cite{Farah-Solecki}), and not all
subgroups that are $\mathcal F_{\delta\sigma}$ sets are Polishable. In
addition, he points out
that there are Polishable $\mathcal F_{\delta\sigma}$ subgroups which
are not of the
form $\mathcal R(n_m)$ for some $(n_m)$. These results suggest that
there is unlikely to
be a descriptive set theoretic characterization of the class of
subgroups of the form
$\mathcal R(n_m)$.

\subsubsection{\bf Rank One Constructions} \label{rankone}

In this section, given an increasing sequence $(n_m)$
such that either $\frac {n_{m+1}} {n_m} \to \infty$ or
$\frac {n_{m+1}} {n_m}$ is a whole number for each $m$,
we will explicitly construct an infinite
measure-preserving rank one map $T$ such that $T^{n_m} \to Id$
in the strong operator topology.  As observed
in Remark~\ref{Poisson}, the Poisson suspension gives an example
of a weakly mixing finite measure-preserving transformation $S$ such
that $S^{n_m} \to Id$ in the strong operator topology.
  The Poisson suspension is an appealingly natural
construction  in that
   no spectral measure intervenes. However, it is also
  worth noting that  by the following lemma
  $U_{T}$ automatically has continuous spectrum
  so that one may  apply the GMC to its maximal spectral type to
  obtain the desired  $S$. We note that any rank one $T$ is
  necessarily ergodic.

  \begin{lem}\label{wmnoeigen}
      If $T$ is an ergodic measure-preserving automorphism of an
      infinite measure space $(X,\mathcal B,\mu)$ then $U_{T}$ has
      continuous spectrum.
  \end{lem}
  \begin{proof}
      Suppose $f \in L_{2}(X,\mu)$ and $f \circ T = \lambda f$. Then $|\lambda|
      = 1$ so $|f|$ is $T$-invariant and it follows that $|f|$ is constant. Since $\mu$ is
      infinite it follows that $f=0$.
  \end{proof}

We use below the notation $a:=b$ or $b=:a$ to mean that $a$ is defined to be $b$.
 We assume that the reader has some familiarity with
 rank one constructions but
 the following is  a quick refresher. For more details see ~\cite{Nadkarni}
 or ~\cite{Ferenczi}.
 Suppose $T$ is a rank one map preserving a finite or infinite
 measure $\mu$ and $\{\tau_N\}$ is a refining sequence of rank one
 towers for $T$. This means that $\tau_{N+1}$ may be viewed as
 constructed from $\tau_N$ by
 cutting $\tau_N$ into columns of equal width and
 stacking them above each other, with the possible addition of
  spacer levels between the columns. We will refer to these
 columns of $\tau_N$ as copies of $\tau_N$. The crucial condition
 that makes $T$ rank one is that  the towers
 $\{\tau_N\}$  are required to
 converge to the full sigma-algebra of the space in
 the sense that for any measurable set $E$  of finite measure
 and $\epsilon > 0$ there is
 an $E'$ which is a union of levels of some $\tau_N$ (and hence of
 all $\tau_N$ for $N$ sufficiently large) such that $\mu(E \triangle
 E') < \epsilon$. We let $X_N$ denote  the union of the levels of $\tau_N$.

 Any such $T$ may be realized concretely as a map of
 an interval $I \subset \R$ as follows. We take
 $X_{0}$ to be a finite sub-interval of $\R$ and let $\tau_{0}$
 be the tower of height $1$ consisting of the single level $X_{0}$.
 Now suppose that $\tau_{1}, \ldots , \tau_N$ have been constructed,
 each $\tau_{i}$ a tower whose levels are intervals and the union
 of the levels of each $\tau_{i}$ is an interval $X_{i}$. At this point $T$ is
 partially defined on $X_N$ by mapping each level of $\tau_N$ to
 the level directly above it by the appropriate translation, except
 for the top level, where $T$ remains undefined as yet.
Divide the base of $\tau_N$ into  $q$ subintervals of equal width $w$
and denote the columns of $\tau_N$ over these by $C_{1}, \ldots,
C_{q}$.
Let $r \geq 0$, take $S$ to be an
  interval of width $rw$ adjacent to the interval $X_N$,
   divide $S$ into  $r$ spacer
 intervals of width $w$,
 stack $C_{1}, \ldots, C_{q}$ in order above each other and
 interleave the $r$ spacer intervals in any way between, below and
 above the columns $C_{1}, \ldots, C_{q}$.
 We  then define $T$ partially on $X_{N+1}$ using $\tau_{N+1}$ in the
 same way it was defined on $X_N$ and this is evidently consistent
 with the definition of $T$ on $X_N$. Thus, in the limit $T$ is
 almost everywhere defined on $I = \bigcup\limits_{N=1}^\infty X_N$ and is evidently
 rank one.  Note that in the concrete model the convergence of
 $\tau_N$ to the Borel $\sigma$-algebra of $I$ is automatic.

 We let  $S_N = X_{N+1} \backslash X_N$    and
 $\epsilon_N = \frac {\mu(S_N)} {\mu(X_{N+1})}$, the fraction of the
 levels of $\tau_{N+1}$ which are not contained in a level of
 $\tau_N$; that is, they are spacers added at stage $n$ of the
 construction. We observe that $\mu$ is infinite precisely when
 $\sum\limits_{N=1}^\infty \epsilon_N = \infty$.

 For a fixed $\tau_N$ we will say a time $N > 0$ is
 $\epsilon$-rigid for $\tau_N$ if for each level $E$ of $\tau_N$ we have
 $\mu(T^{N}E \triangle E) < \epsilon \mu(E)$. Note that one then has the
 same inequality for any $E$ which is a union of levels of $\tau_N$.
 Consequently if $N$ is  $\epsilon$-rigid for $\tau_N$
 then it is also $\epsilon$-rigid for any
 $\tau_{M}, \ M \leq N$.
  We will say that the sequence $(n_m)$ is {\em rigid for a set $E$
  of finite measure} if
 $\mu(T^{n_{m}}E \triangle E) \to 0$; and that {\em $(n_{m})$ is rigid
 for $\tau_N$, $N$ fixed}  if $(n_{m})$ is rigid for each level
 of $\tau_N$ (equivalently, for the base of $\tau_N$).
 Finally note that if $(n_{m})$ is rigid for every
 $\tau_N$ then $T$ is rigid along $(n_m)$.

\begin{prop} \label{infrankone}
    Suppose $\frac {n_{m+1}} {n_m} \to \infty $ as $m \to \infty$ or $\frac {n_{m+1}}
    {n_m} $ is a whole number, $\frac {n_{m+1}}
    {n_m} \ge 2$ for all $m$. Then there exists an
    infinite measure-preserving, weakly mixing, rank one transformation  $T$ such that $T$ is rigid along $(n_m)$.
\end{prop}
\begin{proof}
    Suppose first that $\frac {n_{m+1}} {n_m} \to \infty $ as $m \to \infty$.
   We introduce the notation $h_{m}:=n_m$ as, in this case, these will be the
     heights of the rank one towers we construct.
    Write $h_{m+1} = q_m h_m + r_m$, $0 \leq r_m < h_m$.
     Define  $p_m < q_m$  to be the least integer $l \geq 0$
     such that,  $   \frac { r_m + lh_m} {h_{m+1}} > \frac 1 m$ ($p_m$ may be
   zero) and let $\epsilon_{m} = \frac { r_m + p_mh_m} {h_{m+1}}$.
   Thus we have $\sum\limits_{m=1}^\infty  \epsilon_m = \infty$. Moreover $\epsilon_m -
   \frac 1m < \frac {h_m} {h_{m+1 }} \to 0$ so $\epsilon_m \to 0$.

   Construct $T$ as follows.
     Start with  a tower $\tau_{1}$ of height
      $h_{1}$ and suppose the towers $\tau_{1}, \ldots, \tau_m$
      have been constructed.   Form $\tau_{m+1}$ by slicing $\tau_m$ into
   $s_m: = q_m - p_m$ columns, stacking these directly above
   each other and then following them
   by ${ h_{m+1}- s_mh_m} =  r_m + p_m h_m$ spacers
   to create a tower $\tau_{m+1}$ of height
   $h_{m+1}$. Thus $\frac {\mu(S_m)} {\mu(X_{m+1})} = \epsilon_m$,
   and since  $\sum\limits_{m=1}^\infty   \epsilon_m = \infty$ we see that the measure
   of the space we have   constructed is infinite.  By Lemma~\ref{wmnoeigen}, $T$
   is weakly mixing.

   We now check that $T$ is rigid along $(h_m)$. If $E$ is a
   level of $\tau_m$ and $E_{s_1}, \ldots, E_{s_m}$ are its pieces
   in $\tau_{m+1}$ then $T^{h_m}E_{i} = E_{i+1}$, except for $i =
   s_m$.
   It follows that $\mu(T^{h_m}E \backslash E) < \frac 1 {s_m} \mu(E) =:
   \delta_m\mu(E)/2$ so
   $h_m$ is $\delta_m$-rigid for $\tau_m$. Since this holds for
   every $m$ it follows that, for any fixed  $n$, $h_m$ is $\delta_m$
   rigid for $\tau_{n}$, for each $m > n$. Since $\delta_m \to 0$, it
   follows that  for each fixed $m$,
   $(h_m)$ is rigid for $\tau_{m}$.  So $T$ is rigid along $(h_m)$.
   This concludes the argument in case $\frac {n_{m+1}} {n_m}
   \to \infty$.

   Now suppose that $\frac {n_{m+1}} {n_m} $ is a whole number as large as $2$ for all
   $m$. For simplicity we will consider only the case   $n_m = 2^{m}$.
   The general case is no more difficult.   Let $h_m = 2^{m^{2}}$
   so $q_m:= h_{m+1} / h_m= 2^{2m+1} \to \infty$. Let
   $p_{m} \geq 0$ be the least
   integer $r$ so that
   $\frac {r h_m} {h_{m+1}} \geq \frac 1 m$
  and let $\epsilon_m = \frac {p_m h_m} {h_{m+1}}$.
   As before we have $\sum\limits_{m=1}^\infty  \epsilon_m = \infty$ and $\epsilon_m \to 0$.

   We construct the towers $\tau_m$ for  $T$ as before,
   by concatenating $s_m: = q_m - p_m$ copies
   of $\tau_m$ and adding $p_mh_m$ spacers to get the tower
   $\tau_{m+1}$ of height $h_{m+1}$. As before the space on which $T$
   acts has infinite measure and we need only check the rigidity of
   the sequence
   $(n_{m}) = (2^{m})$.

  Now suppose $E$ is a level of
  $\tau_m$ and $E_{1}, \ldots, E_{l}$, $l=s_ms_{m+1}$, are its pieces in
  $\tau_{m+2}$.  These occur with period $h_m$ in
  $\tau_{m+2}$, except for gaps corresponding to the spacers  in $S_m$
  and $S_{m+1}$. More precisely,
  let us divide $\tau_{m+2}$ into $q_mq_{m+1}$
    blocks of length $h_m$ and also into $q_{m+1}$
    blocks of length $h_{m+1}$ and refer to these as
    $m$-blocks and     $(m+1)$-blocks respectively. Each $m$-block is
    contained in either $X_m$,  $S_m$ or $  S_{m+1}$.
    Call these three types $X_m$-blocks, $S_m$-blocks and
    $S_{m+1}$-blocks
     and let
    the numbers of the three types  be
    $a= s_ms_{m+1}, b = p_{m}s_{m+1}$ and $c=q_mp_{m+1}$.

  Now  suppose that $M >0$ and let $m = m_M > 0$ be the integer such
  that $m^{2} \leq M < (m+1)^{2} $. Since there is at
  least one $(m+1)$-block at the top of
  $\tau_{m+2}$  which  is contained in $S_{m+1}$ we see that for each
  $i$,
  $1 \leq i \leq l$,
    $T^{2^{M}} E_{i}$ is
  still a level of $\tau_{m+2}$. Thus, if it is not contained in $E$ it
  must lie  in an $S_m$-block  or an $S_{m+1}$-block. It follows that
    \begin{align*}
    \frac   {\mu(T^{2^{M}} E \backslash E)} {\mu(E)} \leq \frac {b+c} {a}
      = & \frac {p_ms_{m+1} + q_mp_{m+1}} {s_ms_{m+1}} \\
       = & \frac{p_m} {s_m} + (\frac {q_m} {s_m})
      (\frac       {p_{m+1}} {s_{m+1}})
	= \epsilon_m + \frac 1 {1-\epsilon_m} \epsilon_{m+1} =: \delta_m/2.
       \end{align*}
	       This shows $2^{M}$ is $\delta_{m_{M}}$-rigid for
       $\tau_{m_{M}}$.  Fixing any $k \leq n_{M}$, it follows that
       $2^{M}$ is $\delta_{m_{M}}$-rigid for $\tau_{k}$. Letting $M \to
       \infty$ we have $m_{M} \to \infty$ and $\delta_{m_{M}} \to 0$.
       So we see
       that the sequence $\{2^{M}\}$ is rigid for $\tau_{k}$, and since $k$
       is arbitrary it follows that $T$ is rigid along  $(2^{M})$, as desired.

\end{proof}

We also note that there is a special case of Proposition~\ref{infrankone}
where we can
provide a direct construction of a finite measure-preserving rank one
$T$ such that $T^{n_m} \to Id$.
  \begin{prop}\label{specialinfrankone}
      Suppose that $ {n_{m+1}} = q_{m}{n_m} + r_{m}$, $0 \leq
      r_{m} < n_m$, $q_{m} \to \infty$, $\sum\limits_{n}^{\infty} \frac {r_{m}}
      {n_{m+1}} < \infty$ and $r_{m} \neq 0$ infinitely often.
      Then there is a finite measure-preserving, weakly mixing, rank one transformation
$T$ that is rigid along $(n_m)$.
  \end{prop}
  \begin{proof}
	 We construct the rank
      one towers $\tau_{m}$ of height $h_{m}: = n_m$ for $T$ as follows.
      Start with a tower of height $h_{1}$. When $r_{m} = 0$,
      we construct $\tau_{m+1}$ from $\tau_{m}$ by  simply
      concatenating $q_{m}$ copies of $\tau_{m}$  to obtain the tower
      $\tau_{m+1}$ of height
      $h_{m+1}$. When  $r_{m} \neq 0$, we
      place $a_{m}:= [q_m/3]$ consecutive copies of $\tau_{m}$ followed
      by one spacer, followed by $q_m - a_{m}$ consecutive copies
      of $\tau_{m}$, followed by $r_{m} - 1$ spacers, again giving
      $\tau_{m+1}$  of height $h_{m+1}$.
	    The resulting $T$ is finite measure-preserving because we have
      assumed $\sum\limits_{m=1}^{\infty} \frac {r_{m}}
      {h _{m+1}} < \infty$ and  it is very easy see that
       $T$ is rigid along $(h_m)$.

       We now check that $T$ is weakly mixing.   Suppose that $f \in
       L_{2}(X,\mu)$ and $f \circ T = \lambda f$. Without loss of generality
       $|f| = 1$. Given $\epsilon > 0$, find $n = n_{m}$ such that
       $r_m \neq 0$, and $f'$ which is a
       linear combination of the characteristic functions of the
       levels of $\tau_{m}$ such that $\|f -f'\|_{2} < \epsilon$.
       In addition we may assume that $|f(x)| = 1$ for all $x
       \in X_{m}$.
       \def\close#1{\ {\mathop{\sim}\limits^{#1}}\ }
       We agree to write $g \close \delta h$ whenever $g,h \in L_{2}(X,\mu)$ and
       $\| g- h\|_{2} < \delta $.  Note that  for any $k$ we have
       \begin{equation*}
	   f' \circ T^{k} \close \epsilon f \circ T^{k} = \lambda^{k} f \close
	   \epsilon \lambda^{k} f'.
       \end{equation*}

       Let $E_{1}$ denote the union of the first $s_{m}:= a_{m}h_{m}$ levels of
       $\tau_{m+1}$ and $E_{2}$ the union of the $s_{m}$ levels
       after the first spacer in $\tau_{m+1}$. We observe that
       $$
       f'|_{E_{1}} = (f' \circ T^{s_{m}+1})|_{E_{1}} \close {2\epsilon}
       \lambda^{s_{m}+1} f'|_{E_{1}}.
       $$
       It follows that $\lambda^{s_{m} + 1}
       \close {2\epsilon/ \|f'|_{E_{1}}\|_{2}} 1$. By taking $m$
       sufficiently large we may assume that $\mu(E_{1}) \geq \frac 1
       4$ and so
       $$\|f'|_{E_{1}}\|_{2} = \sqrt {\mu(E_{1})} > \frac 1 2.$$
       Thus, $\lambda^{s_{m} + 1} \close {4\epsilon} 1$. A similar argument with
       $E_{2}$ replacing $E_{1}$ shows that $\lambda^{s_{m}} \close {4\epsilon}
       1$ so we get $\lambda^{s_{m} + 1} \close {8\epsilon} \lambda^{s_{m}}$.
       Since $|\lambda| = 1$ it follows that $\lambda \close {8\epsilon} 1$ and
      since $\epsilon > 0$ is arbitrary we conclude that $\lambda = 1$.
  \end{proof}

\subsubsection{\bf Rates of Growth}\label{ratesofgrowth}

Here is some information on the question of rates of growth
of the gaps in a rigidity sequence. These results show in various ways that
although rigidity sequences $(n_m)$ have the gaps $n_{m+1} - n_m$
tending to infinity, they do not need to have these gaps growing quickly.
Indeed, there is no rate, no matter how slow, that these gaps must
grow for either ergodic rotations of the circle or weakly mixing
transformations.

Suppose we have an increasing sequence $\mathbf{n} = (n_m)$. We let
\[D(N,\mathbf{n}) =\frac {\#(\{n_m:m \ge 1\}\cap \{1,\ldots,N\})}N.\]
We say that $\mathbf{n}$ has {\em density zero} if
$D(N,\mathbf{n}) \to 0$ as $N \to \infty$. The following result
can be improved, see Corollary~\ref{bestdecay} below.  We prove
this here because it gives insight into the ideas in
Proposition~\ref{ergodicratewm}.

\begin{prop} \label{ergodicrate}
Given any sequence $(d_N:N\ge 1)$ such that $d_N \to 0$
as $N \to \infty$, and any ergodic rotation $T$ of $\mathbb T$, there
exists a rigidity sequence $\mathbf{n} = (n_m)$ for $T$ such that
$D(N,\mathbf{n}) > d_N$ for infinitely many $N \ge 1$.
\end{prop}
\begin{proof}
We have some
$\gamma \in \mathbb T$ of infinite order such that $T(\alpha) =
\gamma \alpha$ for all $\alpha \in \mathbb T$.
Because $d_N\to 0$
as $N\to \infty$, we can choose an increasing sequence $(N_k)$ such
that for all $N \ge N_k$, we have
$d_N \le \frac 1{4^k}$. Now we construct a
suitable $(n_m)$ that is rigid for $T$. First, we can inductively
choose an increasing sequence $(M_k)$ so that we have $\#\{n \in [1,M_k]:
|\gamma^n-1| \le \frac 1{2^k}\} \ge \frac {M_k}{2^k}$. This is possible
because
the Lebesgue measure of the arc $\{\alpha: |\alpha - 1|\le \frac 1{2^k}\}$
is $\frac 2{2^k}$ and $(\gamma^n:n\ge 1)$ is uniformly distributed in
$\mathbb T$. In the process
of choosing $(M_k)$, there is no obstruction to taking each $M_k \ge N_k$.
Now let $(n_m)$ be the increasing sequence whose terms are
$\bigcup\limits_{k=1}^\infty \{n \in [1,M_k]: |\gamma^n-1| \le \frac
1{2^k}\}$.
By the construction, $(n_m)$ is rigid for $T$. Now
we claim that $D(N,\mathbf{n}) \ge d_N$ for infinitely many $N$.
Indeed, $D(M_k,\mathbf{n}) \ge \frac 1{2^k}$ by the choice
of $M_k$ and the definition of $\mathbf {n}$.
However, since $M_k \ge N_k$, we have $d_{M_k} \le \frac 1{4^k}$.
\end{proof}

\begin{cor} \label{rotslow} Given any sequence $G(m)$ tending to infinity
and any ergodic rotation $T$ of $\mathbb T$,
there exists a rigidity sequence $(n_m)$ for $T$ such that
$\limsup\limits_{m\to \infty} \frac {G(m)}{n_{m+1} - n_m} = \infty$.
\end{cor}
\begin{proof}
We have some
$\gamma \in \mathbb T$ of infinite order such that $T(\alpha) =
\gamma \alpha$ for all $\alpha \in \mathbb T$.
Take $g(m)$ tending to infinity.
\medskip

\noindent {\bf Claim}: There is a sequence $(d_N)$ tending to zero,
determined
by $g$ alone, such that for any
sequence $\mathbf {n} =(n_m)$ that has $n_{m+1} - n_m \ge g(m)$ for all
$m \ge 1$,
we would have $D(N,\mathbf{n}) \le d_N$ for all $N \ge 1$.
\smallskip

\noindent{\bf Proof of Claim}: Observe that among sequences with
$n_{m+1} - n_m \ge g(m)$ for all $m \ge 1$,
$\#\{n_m \le N\}$ is largest in the case that we take the explicit sequence
$n_1 =1$ and $n_{m+1} = n_m + \lceil g(m)\rceil$ for all $m\ge 1$. So,
take this
as our sequence. Let $g(0) = 0$. Then
\[\#\{n_m \le N\} \le \sup\{m \ge 1: 1+\sum\limits_{k=0}^{m-1} \lceil
g(k)\rceil \le N\}.\]
Thus, let $d_N = \frac {\sup\{m: 1+\sum\limits_{k=0}^{m-1} \lceil
g(k)\rceil \le N\}}N$,
which tends to zero as $N \to \infty$ because $g(m) \to \infty$ as $m
\to \infty$.
We have $D(N,\mathbf n) \le d_N$ for all $N \ge 1$.
\medskip

\noindent Continuing now with our proof, for any sequence $\mathbf {n}
=(n_m)$ that has $n_{m+1} - n_m \ge g(m)$ for all $m \ge M$,
then there exists $N_M$ such that $D(N,\mathbf{n}) \le 2d_N$ for all $N
\ge N_M$.
Using $(2d_N)$ in place of $(d_N)$, Proposition~\ref{ergodicrate} and
Claim~\ref{rotslow} above
how that we can construct a rigidity
sequence $(n_m)$ for $T$ such that $n_{m+1} - n_m < g(m)$ for infinitely
many $m$.
Now, for any $G(m)$ increasing to $\infty$, we
can construct $g(m)$ tending to $\infty$ so that
$\lim\limits_{m\to\infty} \frac {G(m)}{g(m)} =\infty$.
Using this $g$ above, we have shown that
there is a rigid sequence $(n_m)$ for $T$ such that
$\limsup\limits_{m\to\infty}
\frac {G(m)}{n_{m+1} - n_m} = \infty$.
\end{proof}

\begin{cor} \label{notexp} For any ergodic rotation $T$ of $\mathbb T$,
there is a rigidity sequence
$(n_m)$ such that $\liminf\limits_{m\to \infty} \frac {n_{m+1}}{n_m} = 1$.
\end{cor}
\begin{proof} Let $G(m) = \sqrt m$. Then use Corollary~\ref{rotslow} to
construct
a rigidity sequence for $T$ such that $n_{m+1} - n_m \le \sqrt m$
infinitely often. Since
$n_m \ge m$, we have $\liminf\limits_{m\to \infty} \frac {n_{m+1}}{n_m}
= 1$.
\end{proof}

Our next result, and some that follow, show that although rigidity
sequences are sparse sets,
they are not always {\em thin sets}
in certain senses that are commonly used in harmonic analysis. See Lopez
and Ross~\cite{LR}
for background information on thin sets in harmonic analysis. In
particular, we will see that
rigidity sequences are not
always {\em Sidon sets}. By a Sidon set here we mean a subset $\mathcal
S$ of the integers
such that given any bounded complex-valued function $\psi$ on $\mathcal
S$, there exists a complex-valued Borel
measure $\nu$ on $\mathbb T$ such that $\widehat {\nu} = \psi$ on
$\mathcal S$. Originally, this
property was observed for lacunary sets, and finite unions of lacunary
sets, but the general notion of Sidon sets gives a larger class of sets
to work with that in a general sense will have similar harmonic analysis properties.

\begin{cor} \label{notlac}
Given any ergodic rotation $T$ of $\mathbb T$, there exists a rigidity
sequence for $T$ which is not a Sidon set, and so is not the
union of a finite number of lacunary sequences.
\end{cor}
\begin{proof}
It is a standard fact that finite unions of lacunary sequences are Sidon
sets. See
for example ~\cite{LR}. Also in ~\cite{LR}, Corollary 6.11, is the proof
that if $\mathbf n = (n_m)$ is a Sidon set then there is a constant $C$
such that
$D(N,\mathbf n) \le \frac {C\log N}N$. By the argument above, there exists
some $d_n \to 0$ as $n \to \infty$ such that for all $C$, eventually
$d_N \ge C \frac {\log N}N$. Using this $(d_N)$, construct a rigidity
sequence $\mathbf {n}$
for $T$ as in Proposition~\ref{ergodicrate}. This choice of $(d_N)$
shows that $\mathbf {n}$ is not a Sidon set.
\end{proof}

\begin{rem} We did not need it here, but sometimes when dealing with
classes of
sequences, it is good to have a result as follows. Suppose we have a
sequence of sequences $(d_N(s):N\ge 1)$ where
$d_N(s) \to 0$ as $N \to \infty$ for every $s$. Then there exists a sequence
$(d_N)$ which also has $d_N \to 0$ as $N \to \infty$, but also for all
$s$, $d_N \ge d_N(s)$
for large enough $N$. This is a standard result. First, let $d_N^*(k) =
\max(d_N(1),\ldots,d_N(k))$. Then for all $k$, again $d_N^* \to 0$ as
$N\to \infty$.
Choose an increasing sequence $(N_k)$ such that $d_N^*(k) \le \frac 1{2^k}$
all $N \ge N_k$. Let $d_N = 1$ for all $1 \le N < N_1$, and for $k \ge
1$, let $d_N =
d_N^*(k)$ for $N_k \le N < N_{k+1}$. Then $d_N \le \frac 1{2^k}$ for $N
\ge N_k$. Also,
for all $j$, $d_N \ge d_N^*(k) \ge d_N(j)$ for any $N \ge N_k$ with $k
\ge j$.
\end{rem}

Now we extend the construction above to give weakly mixing
transformations with
rigidity sequences that satisfy similar slow decay properties.

\begin{prop}\label{ergodicratewm}
Given any sequence $(d_N:N\ge 1)$ such that $d_N \to 0$
as $N \to \infty$, there exists a weakly mixing transformation
and a rigidity sequence $\mathbf {n} =(n_m)$ for $T$ such that
$D(N,\mathbf{n}) > d_N$ for infinitely many $N$.
\end{prop}
\begin{proof}
To carry out this construction, we
start with a closed perfect set $\mathcal K$ in $\mathbb T$ such that
every finite set $F\subset \mathcal K$
generates a free abelian group of order $\#F$. Except that this is in
the circle, it is the
same as constructing a closed perfect set $\mathcal K$ in $[0,1]$ of
rationally independent
real numbers modulo $1$ (i.e. $\mathcal K\cup \{1\}$ is rationally
independent in the real numbers).
See Rudin~\cite{rudin}.
The rational independence tells us that for any $\gamma_1,\ldots,\gamma_L \in \mathcal K$, the
sequence of powers $(\gamma_1^n,\ldots,\gamma_L^n), n \ge 1$
is uniformly distributed in $\mathbb T^L$.

Because $d_N\to 0$
as $N\to \infty$, we can choose an increasing sequence $(N_k)$ such
that for all $N \ge N_k$, we have
$d_N \le (\frac 1{4^k})^{4^k}$.

We now inductively construct $K$, a closed subset of $\mathcal K$, and
$(M_k)$ with certain properties. This
will be a Cantor set type of construction. First,
choose distinct $\gamma(i_1), i_1 = 1,2$, in $\mathcal K$. Then choose
$M_1 \ge N_1$
so that $\#\{n \in [1,M_1]: \text {for}\, i_1=1,2,\, |\gamma(i_1)^n - 1| \le
\frac 13\} \ge (\frac 23)^2 \frac {M_1}2$. This is possible because the
set in $\mathbb T^2$
consisting of $(\alpha_1(1),\alpha_1(2))$ with $|\alpha_1(i) - 1| \le
\frac 13$ for both $i=1,2$ has
Lebesgue measure $(\frac 23)^2$, and $((\gamma(1),\gamma(2))^n:n\ge 1)$
is uniformly distributed
in $\mathbb T^2$. We then choose two disjoint closed arcs $B(i_1),
i_1=1,2$ with $\gamma(i_1) \in \text{int}(B(i_1))$
for $i_1=1,2$ and such that for any $\omega(i_1) \in B(i_1)$, and
$n \in [1,M_1]$ such that $|\gamma(i_1)^n - 1| \le \frac 13$ for
$i_1=1,2$, we have the
somewhat weaker inequality $|\omega(i_1)^n - 1| \le \frac 23$. Let
$A(i_1) = B(i_1)\cap \mathcal K$ for $i_1=1,2$.
Let $A(i_0) = K_0 = \mathcal K$.

We have to continue this inductively. Suppose for fixed $k \ge 1$,
we have constructed $K_l$ and
$M_l \ge N_l$ for all $l=1,\ldots,k$ with the following properties. Each
$K_l \subset K_{l-1}$
and each $K_l$ is a union of a finite number
of pairwise disjoint, non-empty, closed perfect sets $A(i_1,\ldots,i_l)=
B(i_1,\ldots,i_l)\cap \mathcal K$,
given by $i_j=1,2$ for all $j=1,\ldots,l$. Here the sets
$B(i_1,\ldots,i_l)$ are closed arcs , with $\text
{int}(B(i_1,\ldots,i_l))\cap\mathcal K$
not empty, for all $i_j=1,2, j=1,\ldots,l$.
For each $(i_1,\ldots,i_l)$, we have $A(i_1,\ldots,i_l) \subset
A(i_1,\ldots,i_{l-1})$.
In addition, consider the set $E_l$ of $n \in [1,M_l]$ such that for all
$\omega(i_1,\ldots,i_l)
\in B(i_1,\ldots,i_l)$, we have $|\omega(i_1,\ldots,i_l)^n - 1| \le \frac
2{3^l}$. We assume
inductively that $\#E_l \ge (\frac 2{3^l})^{2^l} \frac {M_l}2$.

Now, for each $(i_1,\ldots,i_k)$, choose two distinct
$\gamma(i_1,\ldots,i_k,i_{k+1}) \in A(i_1,\ldots,i_k)$
where $i_{k+1} =1,2$. We can choose $\gamma(i_1,\ldots,i_k,i_{k+1})
\in \text {int}(B(i_1,\ldots,i_k))$. Then the
point $p$ in $\mathbb T^{2^{k+1}}$ with coordinates
$\gamma(i_1,\ldots,i_k,i_{k+1})$, listed in any order,
has $(p^n: n\ge 1)$ uniformly distributed in $\mathbb T^{2^{k+1}}$. So,
there exists $M_{k+1} \ge
N_{k+1}$ such that $\#E_{k+1} \ge (\frac 2{3^{k+1}})^{2^{k+1}} \frac
{M_{k+1}}2$ where $E_{k+1}$ is the set of
$n \in [1,M_{k+1}]$ such that for all $(i_1,\ldots,i_{k+1})$ we have
$|\gamma(i_1,\ldots,i_{k+1})^n-1|
\le \frac 1{3^{k+1}}$. Choose pairwise disjoint closed arcs
$B(i_1,\ldots,i_{k+1})$
with $\gamma(i_1,\ldots,i_{k+1})\in \text{int}(B(i_1,\ldots,i_{k+1}))$
such that we have
the following. Consider any $\omega(i_1,\ldots,i_{k+1})\in
B(i_1,\ldots,i_{k+1})$,
and any $n \in [1,M_{k+1}]$ such that $|\gamma(i_1,\ldots,i_{k+1})^n-1|
\le \frac 1{3^{k+1}}$. Then we have the somewhat weaker inequality
$|\omega(i_1,\ldots,i_{k +1})^n-1|
\le \frac 2{3^{k+1}}$. There is no difficulty in also having
$B(i_1,\ldots,i_{k+1}) \subset B(i_1,\ldots,i_k)$ because we chose
$\gamma(i_1,\ldots,i_k,i_{k+1})
\in \text {int}(B(i_1,\ldots,i_k))$.
We now let
\[A(i_1,\ldots,i_{k+1}) = B(i_1,\ldots,i_{k+1})\cap \mathcal K.\]
Since $B(i_1,\ldots,i_{k+1}) \subset B(i_1,\ldots,i_k)$, we have
$A(i_1,\ldots,i_{k+1}) \subset A(i_1,\ldots,i_k)$.
Let $K_{k+1} = \bigcup\limits_{(i_1,\ldots,i_{k+1})}A(i_1,\ldots,i_{k+1})$.
This completes the inductive step.

To finish this construction, let $K = \bigcap\limits_{k=1}^\infty K_k$.
Then $K$
is a closed perfect subset of $\mathcal K$. Let $\mathbf {n} = (n_m)$ be the
increasing sequence whose terms are all values of $n$
such that $n \in [1,M_k]$ for some $k$, and for all $\omega(i_1,\ldots,i_k)
\in A(i_1,\ldots,i_k)$, we have $|\omega(i_1,\ldots,i_k)^n - 1| \le \frac
2{3^k}$. By the construction,
for all $\omega \in K$, we have $\omega^{n_m} \to 1$ as $m \to \infty$.
Indeed, given $m\ge 1$ choose $k_m$ to be the largest $k$ so that $n_m >
M_k$.
Notice that $k_m\to\infty$ when $m\to \infty$, and the term $n_m$ was
chosen from $[1,M_{s_m}]$ with $s_m > k_m$ and satisfying
$|\omega(i_1,\ldots,i_{s_m})^{n_m}-1|< \frac 2{3^{s_m}}$
for all $\omega(i_1,\ldots,i_{s_m}) \in B(i_1,\ldots,i_{s_m})$. In
particular, if $\omega \in \mathcal K$,
then $|\omega^{n_m}-1| \le \frac 2{3^{s_m}} \le \frac 2{3^{k_m}}$
because $\omega \in B(i_1,\ldots,i_{s_m})$ for some
choice of $(i_1,\ldots,i_{s_m})$.

Hence, for any Borel probability measure $\nu$ supported in $K$, we have
$\widehat {\nu}(n_m) \to 1$
as $m \to \infty$. Since $K$ is a closed perfect set, there are continuous
Borel probability measures $\nu$ supported in $K$. Take the
symmetrization of any such measure and use
the GMC to construct the corresponding
weakly mixing transformation $T$. Then $\mathbf {n}$ is a rigidity
sequence for $T$.
But also we claim that $D(N,\mathbf{n}) > d_N$ for infinitely many $N$.
Indeed, our construction guarantees that
$D(M_k,\mathbf{n}) = \#\{m: n_m \in [1,M_k]\} \ge \frac 12(\frac
2{3^k})^{2^k}$ for
all $k$. But also, because we have chosen $M_k \ge N_k$ for all $k$,
we have $d_{M_k} \le (\frac 1{4^k})^{4^k}$. Therefore,
$d_{M_k}<D(M_k,{\bf n})$.
\end{proof}

\begin{cor} \label{wmslow} Given any sequence $G(m)$ tending to
infinity, there exists a weakly mixing transformation
and a rigidity sequence $(n_m)$ for $T$ such that
$\limsup\limits_{m\to \infty} \frac {G(m)}{n_{m+1} - n_m} = \infty$.
\end{cor}
\begin{proof}
The proof proceeds in the same manner as Corollary~\ref{rotslow}.
\end{proof}

Using the same argument as given in Corollary~\ref{notexp}, one can see from
Corollary~\ref{wmslow} that there is a weakly mixing transformation $T$
and a rigidity sequence
$(n_m)$ for $T$ such that $\liminf\limits_{m\to \infty} \frac
{n_{m+1}}{n_m} = 1$.
See also Remark~\ref{wazna} for another construction of this type.
However, this is actually
a pervasive principle.

\begin{prop} \label{rationear1} Given any rigid weakly mixing transformation
$T$, there is a rigidity sequence $(n_m)$ for $T$ such that
$\liminf\limits_{m\to \infty} \frac {n_{m+1}}{n_m} = 1$.
\end{prop}
\begin{proof} Take a rigidity sequence $(N_m)$ for $T$. We can replace
this by a subsequence so that the
IP set it generates when written in increasing order is also a rigidity
sequence $(n_m)$ for $T$. We can also
arrange that this IP set is sufficiently rarified (by excluding more
terms from $(N_m)$ if necessary)
so that $(n_m)$ has $\liminf\limits_{m\to \infty} \frac {n_{m+1}}{n_m} = 1$.
\end{proof}

\begin{rem} The syndetic nature of recurrence noted at the beginning of
the proof of Proposition~\ref{bestgrowth} below
shows that the above can be modified to give a rigidity sequence $(n_m)$
for $T$, either in the
ergodic rotation case, or the weakly mixing case, for which the ratios
$\frac {n_{m+1}}{n_m}$
are near one for arbitrarily long blocks of values $m$, infinitely often.
\end{rem}

\begin{rem}\label{EisnerGrivaux} In Example 3.18, Eisner and Grivaux~\cite{EG} 
construct a weakly mixing transformation $T$ and a rigidity sequence $(n_m)$ for $T$
such that $\lim\limits_{m\to \infty} \frac {n_{m+1}}{n_m} = 1$.  However,
their example does not necessarily give the type of upper density rate
results of Proposition~\ref{ergodicratewm} and Proposition~\ref{wmslow}.
\end{rem}

The techniques used in the above constructions give the following
important consequence in Corollary~\ref{notSidon}. If one looks at
all of the other constructions of rigidity sequences given in this
article, one might expect
the opposite of what Corollary~\ref{notSidon} gives us. Also, given this type of example of a
rigidity sequence
for a weakly mixing transformation, and the others in this
article, it seems that it may be very difficult to characterize these
sequences
in any simple structural fashion. See also Remark~\ref{linformeg}, c) for a
different viewpoint
on the issue of characterizing rigidity sequences.

\begin{cor} \label{notSidon} There is a weakly mixing transformation $T$
and a rigidity sequence
$(n_m)$ for $T$ such that $(n_m)$ is not a Sidon set, and so is not
the union of a finite number of lacunary sequences.
\end{cor}
\begin{proof} The proof is just like the proof of
Corollary~\ref{notlac}, only here we use Proposition~\ref{ergodicratewm} instead of
Proposition~\ref{ergodicrate}.
\end{proof}

\begin{rem} Again, Example 3.18 in Eisner and Grivaux~\cite{EG} gives a rigidity sequence
for a weakly mixing transformation such that $\lim\limits_{m\to \infty} \frac {n_{m+1}}{n_m} = 1$.  Such a sequence cannot
be a Sidon set because its density $D(N,\mathbf {n})$ is not bounded by $C\log N/N$
for any constant $C$.
\medskip

\noindent There are other types of sequences that our technique here could apply
to, and show that they
cannot characterize rigidity sequences. For example, consider the
sequences studied by Erd\H{o}s
and Tur\'an~\cite{ETuran}, which they call Sidon sets but are now given a
different name (they are
called $\mathcal B_2$ sequences) because of
the current use of the term Sidon sets mentioned above. They show
there sets have density at most $C\frac {N^{1/4}}N$.
So again, we can construct rigidity sequences that cannot be a finite
union of such $\mathcal B_2$
sequences.
\end{rem}

Another direction we can seek for constructing special rigidity sequences
$(n_m)$ is to try and construct them with $n_m \le \Psi(m)$ where
$\Psi(m)$ is growing slowly, or to prove instead that this would force
$\Psi(m)$ to grow quickly. First, we have this result.

\begin{prop} \label{bestgrowth}
Suppose $\Psi(m)\ge m$ for all $m \ge 1$,
and $\lim\limits_{m\to \infty} \frac {\Psi(m)}m = \infty$. Then for any
ergodic rotation $T$ of $\mathbb T$, there
is a rigidity sequence $(n_m)$ and a constant $C$ such that $n_m \le
C\Psi(m)$
for all $m \ge 1$.
\end{prop}
\begin{proof}  First choose $\gamma \in \mathbb T$ of infinity order so that
$T(\alpha) = \gamma \alpha$ for all $\alpha \in \mathbb T$.
Fix open arcs $A_s$ centered on $1$ with
$\lambda_{\mathbb T}(A_s) = \epsilon_s$, where $(\epsilon_s)$
is a sequence decreasing to $0$ as $s \to \infty$. For each
$s$, the sequence $(n\ge 1: \gamma^n \in A_s)$ is syndetic. That
is, there is some $N_s \ge 1$ such that for all
$M$, there exists $n\in [M+1,\ldots,M+N_s]$ such that $\gamma^n \in A_s$.
The syndetic property here is automatic from minimality.  But it is easy to see
this explicitly in this case.
Indeed, fix an open arc $B_s$ centered at $1$ with $\lambda_{\mathbb
T}(B_s) = \epsilon_s/4$.
Then $\{\gamma^nB_s: n\ge 1\}$ covers $\mathbb T$ and so by compactness
there exists
$(\gamma^{n_1}B_s,\ldots,\gamma^{n_K}B_s\}$ which also covers $\mathbb
T$. But
then, for all $M$, there exists some $n_k$ such that $\gamma^{n_k}B_s
\cap \gamma^{-M}B_s
\not= \emptyset$. Hence, $\gamma^{n_k+M} \in B_sB_s^{-1} \subset A_s$. So
if we let $N_s = \max\{n_1,\ldots,n_K\}$, then there exists
$n\in [M+1,\ldots,M+N_s]$ such that $\gamma^n \in A_s$, i.e. $|\gamma^n
-1|< \epsilon_s/2$.

By replacing $\Psi$ by a more slowly growing function, we can arrange
without loss of generality for the additional property that $\Psi(m)/m$
is increasing.
We write $\Psi(m) = m\theta(m)$. Let $C \ge N_1$ and take an
increasing sequence of whole numbers $(K_L: L \ge 1)$ such that
$\theta(K_L) \ge
N_{L+1}$ for all $L \ge 1$. Let $K_0 = N_0 = 0$.
Consider the blocks in the integers, for $L \ge 1$,
of the form
\[B(L,j) =
[K_0N_0+\ldots+K_{L-1}N_{L-1}+jN_L+1,K_0N_0+\ldots+K_{L-1}N_{L-1}+(j+1)N_L]\]
where $j=0,\ldots,K_L-1$. Then we have a sequence of $K_L$ blocks of length
$N_L$. So we can choose $n(L,j) \in B(L,j)$ such that
$|\gamma^{n(L,j)} -1| < \epsilon_L/2$. Let $(n_m)$ be the increasing
sequence
consisting of all such choices $n(L,j)$ where $L\ge 1$ and
$j=0,\ldots,K_L-1$.
By construction, $(n_m)$ is a rigidity sequence for $T$.

We want to show that $n_m \le Cm\theta(m)$ for all $m$. But the values of
$m$ here are of the form $K_0+\ldots+K_{L-1}+j+1$.
For each such $m$, the corresponding $n_m$ is being chosen in
the interval $B(L,j)$. So, it is sufficient for us to prove that
\[K_0N_0+\ldots+K_{L-1}N_{L-1}+(j+1)N_L \le C\Psi(K_0+\ldots+K_{L-1}+j+1).\]
For $L=1$, we need to have $(j+1)N_1 \le C\Psi(j+1)$.
But this follows since $\Psi(m) \ge m$ for all $m\ge 1$, and $C \ge N_1$.
For $L \ge 2$, we see that the inequality is
\[K_1N_1+\ldots+K_{L-1}N_{L-1}+(j+1)N_L \le
C(K_1+\ldots+K_{L-1}+j+1)\theta(K_1+\ldots+K_{L-1}+j+1).\]
But we have
\[C(K_1+\ldots+K_{L-1}+j+1)\theta(K_1+\ldots+K_{L-1}+j+1)\ge\]
\[(K_1+\ldots+K_{L-1}+j+1)\theta(K_{L-1}) \ge (K_1+\ldots+K_{L-1}+j+1)N_L\]
\[\ge K_1N_1+\ldots+K_LN_{L-1}+(j+1)N_L.\]
\end{proof}

\begin{cor}\label{bestdecay}
Given any sequence $(d_N:N\ge 1)$ such that $d_N \to 0$
as $N \to \infty$, and any ergodic rotation $T$ of $\mathbb T$, there
exists a rigidity sequence $\mathbf{n} = (n_m)$ for $T$ and
a constant $c > 0$ such that
$D(N,\mathbf{n}) > cd_N$ all $N \ge 1$.
\end{cor}
\begin{proof} We assume without loss of generality that $(d_N)$
is decreasing. Take the sequence $(n_m)$ constructed in
Proposition~\ref{bestgrowth}
with $\Psi(m) = \frac m{d_{m-1}}$ for $m \ge 2$.
Fix $N$ with $n_m \le N < n_{m+1}$. Then
\begin{eqnarray*} \frac {\#\{n_k:k\ge 1\}\cap\{1,\ldots,N\}}N &=& \frac mN
\ge \frac m{n_{m+1}} \\
&\ge& \frac m{C\Psi(m+1)}
\ge \frac {d_m}{2C}\\
&\ge& \frac {d_{n_m}}{2C}
\ge \frac {d_N}{2C}.
\end{eqnarray*}
\end{proof}

\begin{rem} The arguments above for an ergodic rotation of the circle
can easily be
generalized to an ergodic rotation of any compact metric abelian group
$G$. The details of this
are a straightforward generalization of the arguments given here.
Indeed, an ergodic rotation $T$
of $G$ is given by $T(g)=g_0g$ for some $g_0 \in G$ that generates a
dense subgroup of $G$. But then also for
any $U$, an open neighborhood of the identity in $G$, we have $\{n \ge
1: g_0^n \in U\}$ is a
syndetic sequence. Again, the syndetic property here
is automatic from minimality.  But it is easy to see
this explicitly in this general case by a proof similar to the one given at the
beginning of
Proposition~\ref{bestgrowth}. Another approach to
Proposition~\ref{bestgrowth} and Corollary~\ref{bestdecay},
and their generalizations to compact metric abelian groups, would be to
use discrepancy estimates and
the fact the $\{g_0^n: n \ge 1\}$ is uniformly distributed in $G$.
See Kuipers and Niederreiter~\cite{KN} for background information about
uniform distribution in
compact abelian groups.
\end{rem}

\begin{rem} It has not yet been possible (and may not be) to carry out a construction
as in Proposition~\ref{bestgrowth} or Corollary~\ref{bestdecay} for some
weakly mixing transformation.
\end{rem}

\subsection{\bf Symbolic approach}
\label{symbolic}

In this section, we use shifts on products of finite spaces,
a standard model that appears in symbolic dynamics.  We look at a construction
of rigid sequences for weakly mixing transformations that is given by taking the usual
coordinate shift on the product of certain finite sets, and giving a careful construction of a
measure on this product space.  This will allow us to show that
certain types of sequences $(n_m)$ can be rigidity sequences for weakly mixing
transformations, even though the sequences do not have the pointwise
behavior of the sequences in Section~\ref{diophantine} i.e. the set $\mathcal R(n_m)$
contains no elements at all in $\mathbb T$ of infinite order, let alone an infinite
perfect set of points.
After this, in Section~\ref{products}, we take an approach to
the basic construction in this section, but we use Riesz products to get the results.
We consider these Riesz product constructions here partly because the issues that one needs to handle there anticipate the
results in Section~\ref{secdisjoint} where we show there are no universal rigid sequences.

We have seen that if $n_{m+1}/n_m$ is bounded then a pointwise approach
to the rigidity question will not work. So when this happens, the next
result
is giving us weakly mixing transformations that are rigid along $(n_m)$ even
though $\mathcal R(n_m)=\{\gamma \in \mathbb T: \lim\limits_{m\to \infty}
\gamma^{n_m} =1\}$ is countable.

\begin{prop} \label{integerratios} Given an increasing sequence
$(n_m)$ such that $n_{m+1}/n_m$ is always a whole number, there is a weakly
mixing dynamical system for which there is rigidity along $(n_m)$.
\end{prop}
\begin{proof}
let $a_1 = n_1$ and $a_{m+1} = n_{m+1}/n_m$ for all $m \ge 1$.
So $n_m = \prod\limits_{k=1}^m a_k$ and
$\frac {n_m}{n_M} = \prod\limits_{k=M+1}^m a_k$ for all $m \ge M+1$.
Consider the series representations for $x \in [0,1)$ of the form
$x =\sum\limits_{m=1}^\infty \frac {b_m}{n_m}$ where $b_m$ is
a whole number such that $0 \le b_m < a_m$. Except for a set
of Lebesgue measure zero, such series are uniquely determined
by $x$, and vice versa. We will write $b_m(x)$ to indicate
the dependence of $(b_m)$ on $x$. Let $\Pi =
\prod\limits_{m=1}^\infty \{0,\ldots,a_m-1\}$ and let
$\pi = \prod\limits_{m=1}^\infty \pi_m$ where $\pi_m$
is the uniform counting
measure on $\{0,\ldots,a_m-1\}$. It is easy to see that
there is a one-to-one, onto Borel mapping $\Phi$ of
$[0,1)$ to $\Pi$ such that $\pi\circ \Phi^{-1}$ is Lebesgue
measure.

Now we use a block construction to build a positive
continuous Borel measure $\nu$ such that $\|n_mx\| \to 0$ in
measure with respect to $\nu$ as $m \to \infty$.
Consider disjoint intervals $I_k = [N_k+1,\ldots,N_{k+1}]$, with
$0 = N_0 < N_1 < N_2 < \ldots$, and $N_{k+1} - N_k = |I_k|$, the length
of $I_k$, increasing to $\infty$. Fix $(\epsilon_k)$ with
$0 < \epsilon_k \le \frac 12$ such that $\sum\limits_{k=1}^\infty \epsilon_k
= \infty$ and $\lim\limits_{k \to \infty} \epsilon_k = 0$.
Define $\nu_k$ on $\Pi$ as follows. Let $\overline {0_k}$ be
the element in $\prod_{I_k} \{0,\ldots,a_m-1\}$ whose entries are
all $0$. Let $\nu_k(\overline {0_k})
= 1 - \epsilon_k$ and $\nu_k$ uniformly distributed over the points
in $\prod_{I_k} \{0,\ldots,a_m-1\}\backslash \{\overline {0_k}\}$ with
the total mass $\nu_k(\prod_{I_k} \{0,\ldots,a_m-1\}
\backslash \{\overline {0_k}\}) =
\epsilon_k$. Then let
$\nu = \prod\limits_{k=1}^\infty \nu_k$ and consider
the measure $\nu\circ \Phi$ on $[0,1)$.

We know that $\nu$ and $\nu\circ \Phi$ are
regular Borel probability measures because all finite Borel measures
are regular in this situation. Moreover, $\nu$
and $\nu\circ \Phi$ are
continuous, i.e. they have no point masses. Indeed, $\nu(\{\overline 0\})
= \prod\limits_{k=1}^\infty (1 - \epsilon_k) = 0$ because
$\sum\limits_{k=1}^\infty \epsilon_k = \infty$, and by the definition
of $\nu$ for every other point $\overline x \in \Pi$, we have
$\nu(\{\overline x\}) \le \nu(\{\overline 0\})$.

We claim that $\|n_Mx\| \to 0$ in measure
with respect to $\nu\circ \Phi$ as $M \to \infty$.
We see that $\|n_Mx\| = \sum\limits_{m=M+1}^\infty \frac
{b_m}{n_m/n_M} = \sum\limits_{m=M+1}^\infty
\frac {b_m}{a_{M+1}\ldots a_m}$.
For any $M \ge 1$, we have
$M \in I_k$ for a unique $k = k(M)$.
Clearly, as $M \to \infty$, we have $k(M) \to \infty$.
Consider the set $D_k$ of vectors
$\overline b = (b_m) \in \Pi$ such that $b_m = 0$ for
$m \in I_k \cup I_{k+1}$. We have $\nu(D_k) =
(1 - \epsilon_k)(1 - \epsilon_{k+1}) \to 1$ as $k \to \infty$.
But also for $\overline b \in D_{k(M)}$, we have
$b_m = 0$ for all
$m, N_{k(M)}\le m \le N_{k(M)+2}$. Hence, if $\Phi(x) \in D_{k(M)}$,
then
$\|n_Mx\| = \sum\limits_{m=N_{k(M)+2}+1}^\infty
\frac {b_m}{a_{M+1}\ldots a_m} \le \sum\limits_{m=N_{k(M)+2}+1}^\infty
\frac 1{a_{M+1}\ldots a_{m-1}} \le \sum\limits_{m=N_{k(M) +2}+1}\frac
1{2^{m-M-1}} = 2^{M+1}\frac 1{2^{N_{k(M)+2}}}$.
But because we chose
$|I_k|$ increasing to $\infty$, we have
$2^{M+1}\frac 1{2^{N_{k(M)+2}}} \to 0$ as $M \to \infty$.
This means that, as $M \to \infty$,
we have $\nu\circ \Phi(\Phi^{-1}D_{k(M)}) \to 1$
and for $x \in \Phi^{-1} D_{k(M)}$, we have $\|n_Mx\| \to 0$.
\end{proof}

This result leads to the following important special case.

\begin{cor} \label{powers} Given any whole number $a \ge 2$, there is a
weakly
mixing dynamical system for which there is rigidity along
$(a^m: m \ge 1)$. However, there is never rigidity along $\sigma
\in FS(a^m: m \ge 1)$ as $\sigma \to \infty (IP)$.
\end{cor}
\begin{proof} The first statement follows
from Proposition~\ref{integerratios}.
For the second part, see Proposition~\ref{ipodometer}.
But here is at least the idea.  First
consider the sequence $(2^j: j \ge 1)$. When we take all finite
sums $2^{j_1}+\ldots+2^{j_k}$ with $m \le j_1 < \ldots < j_k$, then
we get all whole numbers $2^ms, s \ge 1$. Hence, to have rigidity
for $f \circ T^{\sigma}$ as $\sigma \to \infty (IP)$, with $\sigma
\in \Sigma = FS(2^j:j \ge 1)$ would imply that $\|f \circ T^{2^ms}
- f\|_2 \to 0$ as $m \to \infty$, independently of $s \ge 1$.
Assume that $f$ is mean-zero and not zero. Then we fix $m$ such
that $\|f\circ T^{2^ms} - f\|_2 \le \frac 12$ for all $s\ge 1$.
But if $T$ is weakly mixing, so is $T^{2^m}$ and hence
\[\lim\limits_{s \to \infty} \langle f\circ T^{2^ms},f\rangle = 0.\]
This is not possible. The same argument works with any sequence
$(a^j: j \ge 1)$ in place of $(2^j)$ with $a \in \mathbb Z^+, a
\ge 2$. The only difference is that we would need to use some
fixed number of finite sums of elements in $\Sigma$ when $a \ge
3$.
\end{proof}

\begin{rem} \label{notIP} a) See Proposition~\ref{ipodometer} for
a different approach to a version of the above result.
\medskip

\noindent b) The contrast of this result with the examples
in Remark~\ref{linformeg} c) is clear.  While $(2^n)$ is a
rigidity sequence for a weakly mixing transformation, certain
simple perturbations of it, that are still lacunary sequences,
like $(2^n+1)$ are not rigidity sequences for weakly mixing
transformations.  This leads to the obvious question:
if we assume that $(n_m)$ is lacunary, but not
necessarily as above a power of a fixed whole number $a \ge 2$,
when is there still a continuous Borel probability measure on
$[0,1)$ with $\widehat {\nu}(n_m) \to 1$ as $m \to \infty$? This
seems to be a difficult problem because it is not just about the
growth rate inherent in lacunarity, but also about the algebraic
nature of the sequence $(n_m)$.
\medskip

\noindent c) On the other hand, it is not clear what happens for
the lacunary sequence $(2^m+3^m)$. That is, can
there be a continuous Borel probability measure
$\nu$ on $\mathbb T$ such that $\widehat {\nu}(2^m+3^m) \to 1$
as $m \to \infty$? By Proposition~\ref{rigidfacts},
if this happens, then $z^{2^m+3^m} \to 1$
in measure with respect to $\nu$. So $z^{2(2^m+3^m)}\to 1$ and
$z^{2^{m+1}+3^{m+1}} \to 1$ in measure with respect to $\nu$.
Taking the ratio gives $z^{3^m}\to 1$ in measure with respect
to $\nu$, and so again taking the appropriate ratio both $z^{2^m} \to 1$
and $z^{3^m} \to 1$ in measure with respect to $\nu$. But the
converse is also true by taking the product of these. So Proposition
~\ref{rigidfacts} shows that $\widehat {\nu}(2^m+3^m) \to 1$
as $m \to \infty$ if and only if both $\widehat {\nu}(2^m) \to 1$
and $\widehat {\nu}(3^m) \to 1$ as $m \to \infty$. Therefore, using
the GMC, what we are seeking is a weakly mixing
transformation which has both $(2^m)$ and $(3^m)$ as rigidity sequences.
\medskip

\noindent d) The example of Eisner and Grivaux~\cite{EG} of an increasing sequence $(n_m)$ such
that $\lim\limits_{m \to \infty} \frac {n_{m+1}}{n_m} = 1$, for
which there is rigidity along $(n_m)$ for some weakly mixing
dynamical system, is one that is constructed inductively.
Concretely, is there rigidity along $(n_m)$ for
some weakly mixing dynamical system if $(n_m)$ is the sequence
obtained by writing $\{2^k3^l: k,l \ge 1\}$ in increasing order?
Note: results of Ajtai, Havas, and
Koml\'os~\cite{AHK} show that for any sequence $(\epsilon_m) $
decreasing to $0$, no matter how slowly, there exists an
increasing sequence $(n_m)$ such that $\frac {n_{m+1}}{n_m} \ge 1
+ \epsilon_m$ but there is not rigidity along $(n_m)$. Indeed,
they give examples where $\lim\limits_{M \to \infty} \frac 1M
\sum\limits_{m=1}^M \exp(2\pi i
n_m x) = 0$ for all $x \in (0,1)$, and hence for any positive
Borel measure $\nu \not= \delta_0$, $\lim\limits_{M \to \infty} \frac 1M
\sum\limits_{m=1}^M \widehat
{\nu}(n_m) = 0$.
\end{rem}

\subsection{\bf Riesz Products}
\label{products}

In this section, we will use products of discrete measures
to give an alternative approach to the results in Section~\ref{symbolic}.
There are several benefits to looking at this approach.  One
is that it may with more work end up being more flexible than
the method in Proposition~\ref{integerratios}.  Also, it gives
us interesting examples that relate to the question of disjointness: the construction of
two weakly mixing dynamical systems with rigid sequences for
which the product action has no rigid sequences.   Disjointness
turns out to be a very difficult fact to prove in  explicit cases, requiring analytical
methods that are not yet fully developed.  However, in Section~\ref{secdisjoint},
we will see how to use random methods to create this disjointness in a
very general way.

The approach we are taking here to defining continuous Borel probability measures $\nu$
with $\widehat {\nu}(n_m) \to 1$ as $m \to \infty$ currently
works just in cases like $(n_m) = (a^m)$ for any integer $a \ge 2$.  We have
some indication that this method can be made more flexible.   This
method is like the classical method of Riesz products. There
are many references for Riesz products, but one that is closely
related to the considerations in this article is Host, M\'ela,
and Parreau~\cite{Ho-Me-Pa}. But unlike the classical approach,
where one is taking Riesz products
in the frequency variable, we are taking Riesz products
in the space variable, using finitely supported
discrete measures. Of course, this too has been used
by others to give constructions of measures.
See Brown and Moran~\cite{BM1, BM2} and
Graham and McGehee~\cite{GMcG}. This method may possibly give a way
to handle cases that we have not been able to handle before, but
that is not clear. However, some other harmonic analysis issues
come into play here that lead to interesting conclusions in
ergodic theory.

The basic approach is as follows. Choose $(a_k)$ and $(b_k)$ positive,
with $a_k + b_k =1$. We assume that $\lim\limits_{k \to \infty}
b_k = 0$ and $\prod\limits_{k=1}^{\infty} a_k =
\prod\limits_{k=1}^\infty (1 - b_k) = 0$. So we assuming
$\sum\limits_{k=1}^\infty b_k = \infty$.
Take a sequence of points $(x_k)$ from $[0,1)$ and let $\omega_k =
a_k\delta_1 + b_k\delta_{\exp(2\pi ix_k)}$. Let $\nu_k = \omega_k
\ast \omega_k^* = (a_k^2 + b_k^2)\delta_1 +
a_kb_k\delta_{\exp(2\pi ix_k)} + a_kb_k\delta_{\exp(-2\pi ix_k)}$.
So we have $\widehat {\nu_k}(j) = |\widehat {\omega_k}(j)|^2 = 1 -
2a_kb_k(1 - \cos(2\pi jx_k))$. In particular, $0 \le \widehat
{\nu_k}(j) \le 1$ for all $k$ and $j$. Consider the infinite
product $\nu = \prod\limits_{k=1}^\infty \nu_k$. We take this as
defined in the weak$^*$ sense; that is, we know that we have the
Fourier transforms of the partial products $\Pi_K =
\prod\limits_{k=1}^K \nu_k$ converging and hence they give the
Fourier transform of some Borel probability $\nu$ in the limit.
Then we evaluate $\widehat {\nu}(n_m)$ by looking at what the
infinite product does.

Here is an example of this approach. Take $n_m = 2^m$ for all $m
\ge 1$ and let $x_k = \frac 1{2^k}$ for all $k \ge 1$. We
specifically choose $b_k = \frac 1{k+1}$, so $a_k = 1 -\frac
1{k+1}$. We defined $\nu$ by the spectral condition
\[\widehat {\nu}(j) = \prod\limits_{k=1}^\infty
\left (1 - 2\frac {(1 - \frac 1{k+1})}{k+1}(1 - \cos (2\pi \frac
j{2^k}))\right ).\]
Now consider the values $\widehat {\nu}(2^m)$. We always have
\[\widehat {\nu}(2^m) = \prod\limits_{k=m+1}^\infty
\left (1 - 2\frac {(1 - \frac 1{k+1})}{k+1}(1 - \cos (2\pi \frac
{2^m}{2^k}))\right ).\]
So we need to know what this does as $m \to \infty$.

But we claim that $\widehat {\nu}(2^m) \to 1$ as $m \to \infty$.
That is, by taking the natural logarithm of the expression, we
need to show that
\[\lim\limits_{m \to \infty} \sum\limits_{k=m+1}^\infty 2\frac {(1 -
\frac 1{k+1})}{k+1}(1 -\cos(2\pi \frac 1{2^{k-m}})) = 0.\]
However, this is the same as showing
\[\lim\limits_{m \to \infty} \sum\limits_{k=m+1}^\infty 2\frac {(1 -
\frac 1{k+1})}{k+1} \, \frac 1{2^{2(k-m)}} = 0.\]
But $\sum\limits_{k=m+1}^\infty 2\frac {(1 -
\frac 1{k+1})}{k+1} \, \frac 1{2^{2(k-m)}} \le \frac {2^{2m+1}}m
\sum\limits_{k=m+1}^\infty \frac 1{2^{2k}} \le \frac {2^{2m}}m
\frac 1{2^{2m}} = \frac 1m$. So we have the estimate that we
needed.

We also need to make an argument that $\nu$ has no atoms. First,
each of the partial products $\Pi_K = \prod\limits_{k=1}^K \nu_k$
is a purely atomic measure. The choice of $a_k$ and $b_k$
guarantees that for $K \ge 1$, these partial products have a
maximal mass at $1$ with the value $\prod\limits_{k=1}^K
(a_k^2+b_k^2) = \prod\limits_{k=1}^K (1 - 2a_kb_k)$. Suppose that
$\nu$ is suspected of having a non-trivial point mass at some
point $\gamma_0$. For a fixed value of $K$, choose an open dyadic
arc $A = \{\exp(2\pi i x): \frac {j-1}{2^K} < x < \frac
{j+1}{2^K}\}$ which contains $\gamma_0$. Then there is a
continuous function $h_K, 0 \le h_K \le 1$, with $h_K(\gamma_0)
=1$ and the support of $h_K$, that is the closure of $\{h_K >
0\}$, contained in $A$. We have for any $M$, $\int h_K \, d\Pi_M
\le \Pi_M(A)$. We can easily see that for $M \ge K$, $\Pi_M(A) \le
\epsilon_K$ with $\lim\limits_{K \to \infty} \epsilon_K = 0$. But
then $\int h_K\, d\Pi_M \le \epsilon_K$ too, so letting $M$ tend
to $\infty$ shows that $\int h_K \, d\nu \le \epsilon_K$. Then
letting $K \to \infty$, this proves that $\nu(\{\gamma_0\}) = 0$
because for all $K$, $\nu(\{\gamma_0\}) \le \int h_K \, d\nu$.

But the estimate we need is easy because, for
$M \ge K$, we can write the partial product $\Pi_M$ as a sum
$\sum\limits_{t=1}^{2^M} c_t(M) \delta_{\exp(2\pi i \frac
{t-1}{2^M})}$ with positive coefficients $c_t(M)$ such that
$\sum\limits_{t=1}^{2^M} c_t(M) = 1$. The value of $\Pi_M(A)$ can
be seen to be no larger than the value that the measure $\mu$
given by
\[ \left (c_{j-1}(K)\delta_{\exp(2\pi i\frac {j-1}{2^K})}+
c_j(K)\delta_{\exp(2\pi i \frac j{2^K})} +
c_{j+1}(K)\delta_{\exp(2\pi i \frac {j+1}{2^K})}\right )
\prod\limits_{k={K+1}}^M
\nu_k\]
gives to $A$ because for the other point masses $\delta_{\exp(2\pi
i \frac {t-1}{2^K})}$, the measure $\delta_{\exp(2\pi i \frac
{t-1}{2^K})}\prod\limits_{k=K+1}^M \nu_k$ has support disjoint
from $A$. Since $\prod\limits_{k={K+1}}^M \nu_k$ itself is a
probability measure, this shows that $\Pi_M(A) \le c_{j-1}(K) +
c_j(K) + c_{j+1}(K) \le 3\prod\limits_{k=1}^K (a_k^2+b_k^2)$. We
take $\epsilon_K= 3\prod\limits_{k=1}^K (a_k^2+b_k^2) =
3\prod\limits_{k=1}^K(1 - 2a_kb_k)$. The choice of $(a_k)$ shows
that $\epsilon_K \to 0$ as $K \to \infty$. Indeed,
$\sum\limits_{k=1}^\infty b_k = \infty$ and $\lim\limits_{k \to
\infty} a_k = 1$, so $\prod\limits_{k=1}^\infty (1 - 2a_kb_k) =
0$.

\begin{rem} a) The method above suggests that we might be able
to choose $(x_k)$ for more general sequences, including perhaps
ones that are not lacunary. But the method implicitly needs
inductively chosen good choices of $x$ where $\exp(2\pi i n_mx)
\to 1$ as $m \to \infty$. It is not clear for a fixed $(n_m)$ how
to characterize when the set of $x$ such that $\lim\limits_{m \to
\infty} \exp(2\pi i n_mx) = 1$ is empty, when it is finite, when
it is countably infinite, or when it is uncountable!
\smallskip

\noindent b) We believe that what is needed generally for the
construction as above of a continuous Borel probability measure
$\nu$ with $\lim\limits_{m \to \infty}
\widehat {\nu}(n_m) = 1$ as $m \to \infty$ is
\begin{itemize}
\item[(a)] $\lim\limits_{k \to \infty} b_k = 0$,
\item [(b)] $\sum\limits_{k=1}^\infty b_k = \infty$, and
\item[(c)] there is a sequence $(x_k)$ such that
$\lim\limits_{m \to \infty} \sum\limits_{k=1}^\infty b_k
\|n_mx_k\|^2=0$
\end{itemize}
For example, this method would work not only for the sequences
above, but also for sequences as in Proposition~\ref{integerratios}.
\smallskip

\noindent c) We have tried to prove directly that the examples
above have a Fourier transform converging to $0$ at $\infty$ along
a sequence of positive density. For example, one would want to
show that for some sequence of values $j$ with positive density,
we have
$$\lim\limits_{j \to \infty} \sum\limits_{k=1}^\infty b_k(1
-\cos(2\pi \frac j{2^k})) = \infty.$$
But currently we do not see how to prove this.
\end{rem}

\subsubsection{\bf Disjoint Rigidity}
\label{secdisjoint}

The method above can perhaps give us some good examples to look at
for which we have two systems that are rigid but not
simultaneously, indeed when the
convolution of their maximal spectral types is strongly mixing.
For example, take the measures $\nu(2)$ and $\nu(3)$ where

\[\widehat {\nu(2)}(j) = \prod\limits_{k=1}^\infty
\left (1 - 2a_kb_k(1 - \cos (2\pi \frac
j{2^k}))\right )\]
and
\[\widehat {\nu(3)}(j)= \prod\limits_{k=1}^\infty
\left (1 - 2a_kb_k(1 - \cos (2\pi \frac
j{3^k}))\right ).\]

We have $\widehat {\nu(2)}(2^m) \to 1$ and $\widehat {\nu(3)}(3^m)
\to 1$ as $m \to \infty$. But does
\beq\label{disjoint}
\widehat {\nu(2)}(j)\widehat {\nu(3)}(j) \to 0 \ \ \text {as} \ \ j
\to \infty?\eeq
This is the same as asking if the series
\[\sum\limits_{k=1}^\infty b_k\left (
1 - \cos (2\pi \frac j{2^k}) + 1 - \cos (2\pi \frac j{3^k})\right
) \] tends to $\infty$ as $j$ tends to $\infty$. This will not
work with $b_k = \frac 1k$; one can see this by taking $j = 6^k$
for large values of $k$. But it may work with $b_k = \frac
1{\sqrt k}$.

While we cannot answer the question in Equation~\ref{disjoint} at this time, we can
modify a construction of LaFontaine~\cite{lafontaine} to
obtain Proposition~\ref{AA0} below. We will need two procedural lemmas.
First, given a finite Borel measure $\omega$ on $\mathbb R$, we
let $FT_B(\omega)(n) = \int\limits_0^B \exp(-2\pi i n\frac xB)\ d\omega(x)$.
We use the notation $\widehat {\omega}$ for the
Fourier transform on $\mathbb R$ given by $\widehat {\omega}(t) =
\int \exp(-2\pi itx)\ d\omega(x)$.
Let $m=m_{\mathbb R}$ denote the Lebesgue measure on $\mathbb R$, and let
$L_2(m)$ denote the Lebesgue space $L_2(\mathbb R,\mathcal B_m,m)$.

\begin{lem}\label{transforms} Suppose we have a positive
Borel measure $\omega$ on $\mathbb R$ that has compact
support. Then $\frac {d\omega}{dm} \in L_2(m)$ if and only
if $\sum\limits_{n=-\infty}^\infty |\widehat {\omega}(n)|^2 < \infty$.
\end{lem}
\begin{proof}
By translating $\omega$, we may assume without loss of generality
that the support of $\omega$ is a subset $[0,B]$ for some whole number
$B$. Then
\begin{eqnarray*}
\widehat {\omega}(n) &=& \int\ \exp(-2\pi i n x)\ d\omega(x)
= \sum\limits_{k=0}^{B-1} \int\limits_{k}^{k+1} \exp(-2\pi i n x) \
d\omega(x)\\
&=& \sum\limits_{k=0}^{B-1} \int\limits_0^1\exp(-2\pi i n (x+k)) \
d\omega(x+k)\\
&=& \sum\limits_{k=0}^{B-1} \int\limits_0^1\exp(-2\pi i n x) \
d\omega(x+k)
=\int\limits_0^1 \exp(-2\pi i n x)\ d\Omega(x).
\end{eqnarray*}
where $d\Omega(x) = 1_{[0,1)}(x)\sum\limits_{k=0}^{B-1} \omega(x+k)$
is a positive Borel measure supported on $[0,1]$. That is,
$\widehat {\omega}(n) = FT_1(\Omega)(n)$ for all $n$. The
usual classical argument shows that $\frac {d\Omega}{dm} \in L_2(m)$
if and only if $\sum\limits_{n=-\infty}^\infty |FT_1(\Omega)(n)|^2 < \infty$
because $\Omega$ is supported on $[0,1)$.
Therefore, $\sum\limits_{n= -\infty}^\infty |\widehat {\omega}(n)|^2 <
\infty$
is equivalent to knowing that $\frac {d\Omega}{dm} \in L_2([0,1),m)$.

So we conclude that if $\frac {d\Omega}{dm} \in L_2([0,1),m)$, then
for each $k=0,\ldots,B-1$, the positive measure
$1_{[0,1]} d\omega(x+k)$ satisfies
$0 \le 1_{[0,1]} d\omega(x+k) \le d\Omega(x)$, and so it also has a
density in $L_2([0,1),m)$.
Hence, by translating the terms back again and adding them together, we
also know that
$\frac {d\omega}{dm}$ is in $L_2(\mathbb R,m)$. Of course also
conversely, if
$\frac {d\omega}{dm}$ is in $L_2(\mathbb R,m)$, then $\frac
{d\Omega}{dm} \in L_2([0,1),m)$.
This proves that $\frac {d\omega}{dm} \in L_2(m)$ if and only if
$\sum\limits_{n= -\infty}^\infty |\widehat {\omega}(n)|^2 < \infty$.
\end{proof}

\begin{rem}\label{othertransforms} This remark and
Lemma~\ref{transforms} are related to the ideas behind
Shannon sampling and the Nyquist frequency for band-limited signals.
In Lemma~\ref{transforms}, if we replace $\omega$ by a dilation of it,
then one can see that more generally for a positive Borel measure
$\omega$ on $\mathbb R$ with compact support, we have $\frac {d\omega}{dm}
\in L_2(m)$ if and only if for some $b > 0$ (or for all $b > 0$),
we have $\sum\limits_{n=-\infty}^\infty |\widehat {\omega}(bn)|^2 < \infty$.

The assumption that $\omega$ is positive is necessary here.
Indeed, suppose we have a compactly
supported complex-valued Borel measure $\omega$ on $\mathbb R$. Suppose
that the support of $\omega$ is a subset of $[0,B]$.
Then by the usual classical argument, we know that
$\frac {d\omega}{dm} \in L_2(m)$ if and only if
$\sum\limits_{n= -\infty}^\infty |FT_B(\omega)(n)|^2 < \infty$.
But $\widehat {\omega}(\frac nB) = FT_B(\omega)(n)$
because $\omega$ is supported on $[0,B]$.
So we have in this case, $\frac {d\omega}{dm} \in L_2(m)$ if and only if
$\sum\limits_{n= -\infty}^\infty |\widehat{\omega}(\frac nB)|^2 < \infty$.
Here, $B$ can be replaced by any larger value, but not necessarily
by a smaller value. For example, if we take $\omega_0$ supported
on $[0,1)$ and define $d\omega(x) = d\omega_0(x) - d\omega_0(x-1)$,
then our value of $B = 2$, but
$\widehat {\omega}(n)=0$ for all $n$. However, the measure $\omega_0$
could be singular to $m$ and hence $\omega$ might not
have an $L_2(\mathbb R,m)$-density.
\end{rem}

We also want to make a few observations about the differences between
convolving measures on $\mathbb T$ and convolving their associated
measures on $\mathbb R$. Given a Borel measure $\omega$
on $\mathbb T$, let $\omega_{\mathbb R}$ denote the Borel measure on
$\mathbb R$ given by $\omega_{\mathbb R}=\omega\circ E$
where $E:[0,1)\to\mathbb T$ by $E(x) = \exp(2\pi i x)$. Let
$m_{\mathbb T}$ denote the usual Lebesgue measure on $\mathbb T$
i.e. $m_{\mathbb T}\circ E = 1_{[0,1)}m_{\mathbb R}$.

\begin{lem}\label{linetocircle} If $\mu$ and $\nu$ are Borel
measures on $\mathbb T$, then $\mu\ast\nu$
is absolutely continuous with respect to $m_{\mathbb T}$ if
$\mu_{\mathbb R}\ast\nu_{\mathbb R}$ is absolutely continuous with
respect to $m_{\mathbb R}$.
\end{lem}
\begin{proof} For $f \in C(\mathbb T)$, we have
\begin{eqnarray*}
\int f(\gamma) d(\mu\ast\nu)(\gamma) &=& \int\int
f(\gamma_1\gamma_2)d\mu(\gamma_1)d\nu(\gamma_2)\\
&=& \int\int f(E(x)E(y)) d\mu_{\mathbb R}(x)d\nu_{\mathbb R}(y)\\
&=& \int f\circ E(x+y) d\mu_{\mathbb R}(x)d\nu_{\mathbb R}(y)\\
&=& \int f\circ E(z) d(\mu_{\mathbb R}\ast\nu_{\mathbb R})(z).
\end{eqnarray*}
Hence, it is clear that if $\mu_{\mathbb R}\ast \nu_{\mathbb R}$
is absolutely continuous with respect to $m_{\mathbb R}$ with
density $F$, then $\mu\ast\nu$ is absolutely continuous with respect
to $m_{\mathbb T}$ with density $F\circ E^{-1}$.
\end{proof}

\begin{rem} The converse statement to Lemma~\ref{linetocircle} is
not true without additional assumption, for example the assumption
that both measures are positive. For example, take a non-zero measure
$\omega_0$ on $\mathbb T$ supported in the arc $E([0,1/2))$. Let
$\omega = \omega_0 - \omega_0\ast \delta_{-1}$. Then let $\nu = \delta_1+
\delta_{-1}$, a positive discrete measure on $\mathbb T$. We have
\begin{eqnarray*}
\omega\ast\nu &=& \omega_0 -
\omega_0\ast\delta_{-1}+\omega_0\ast\delta_{-1} -
\omega_0\ast\delta_1\\
&=&\omega_0-\omega\ast\delta_1\\
&=&0.
\end{eqnarray*}
However, $\omega_{\mathbb R} = (\omega_0)_{\mathbb R}
-(\omega_0)_{\mathbb R}\ast\delta_{1/2}$ and $\nu_{\mathbb R} =
\delta_0+\delta_{1/2}$. So
\begin{eqnarray*}
\omega_{\mathbb R} \ast \nu_{\mathbb R}
&=&(\omega_0)_{\mathbb R} -(\omega_0)_{\mathbb R}\ast\delta_{1/2}+
(\omega_0)_{\mathbb R}\ast\delta_{1/2}-(\omega_0)_{\mathbb R}\ast\delta_1\\
&=& (\omega_0)_{\mathbb R} -(\omega_0)_{\mathbb R}\ast\delta_1.
\end{eqnarray*}
This measure is not zero as a measure on $\mathbb R$.
If also $\omega_0$ is singular to $m_{\mathbb T}$, then we
have $\omega\ast\nu$ absolutely continuous with respect to
$m_{\mathbb T}$, since it is $0$, while
$\omega_{\mathbb R} \ast \nu_{\mathbb R}$ is not absolutely continuous
with respect to $m_{\mathbb R}$.
\end{rem}

These two lemmas will help in proving the following.

\begin{prop}\label{AA0} Assume that
$(X,\mathcal B_X,p_X,T)$ is a rigid, weakly mixing dynamical system.
Then there is a weakly mixing, rigid dynamical system
$(Y,\mathcal B_Y,p_Y,S)$ such that the maximal spectral type of
$U_{T\times S}$ in the orthocomplement of
$F:=L_2(X,p_X)\otimes1_Y\oplus1_X\otimes L_2(Y,p_Y)$ is Rajchman.
In other words, for all $f \in L_2(X\times Y,p_X \otimes p_Y)$
that are orthogonal to both the $X$-measurable functions and the
$Y$-measurable functions, we have $\langle f \circ (T \times
S)^n,f\rangle \to 0$ as $n \to \infty$. In fact, $U_{T\times S}$
has absolutely continuous spectrum on $F^\perp$.
\end{prop}
\begin{proof} From the spectral point of view we want
to show that given a continuous Dirichlet measure $\mu$,
there is a continuous Dirichlet probability measure $\nu$ on
$\mathbb T$ such that $\mu\ast \nu$ is an absolutely continuous
(hence a Rajchman measure). Indeed, we then can take $\mu=\nu^T$ and we
let $(Y,\mathcal B_Y,p_Y,S)$ be given by $S=G_\nu$. As $\mu\ast\nu^{\ast
k}=(\mu\ast\nu)\ast\nu^{\ast(k-1)}$ we easily check that
$\mu\ast\sum\limits_{k=1}^\infty\frac1{2^k}\nu^{\ast k}$ is still
absolutely continuous.
Given the references we are using,
it is better to carry out our construction in $\mathbb R$.
So consider first the measure $\mu\circ E$ on $\mathbb R$.
Lemma~\ref{linetocircle} shows that, to get our result, it will
be enough to construct a suitable positive Borel measure $\nu\circ E$
on $\mathbb R$, with support in $[0,1)$,
such that $(\mu\circ E)\ast(\nu\circ E)$ is absolutely continuous.
For notational convenience we will denote $\mu\circ E$
and $\nu\circ E$ by $\mu$ and $\nu$ in the rest of this proof.

Hence, suppose we have a continuous positive Borel measure $\mu$ on
$\R$, which is
compactly supported. Consider the dilation $\mu_{\lambda}$,
$\lambda > 0$, given by $\mu_{\lambda}(E) = \mu(\frac 1{\lambda}
E)$ for all Borel sets $E\subset\R$. Then $\widehat
{\mu_{\lambda}}(t) = \int \exp(-2\pi itx) \, d\mu_{\lambda}(x) = \int
\exp(-2\pi it\lambda x) \, d\mu(x) = \widehat {\mu}(\lambda t)$.
We would like to construct a suitable continuous positive Borel
measure $\nu$ supported in $[0,1]$ such that $J(\lambda) =
\sum\limits_{n=-\infty}^\infty |\widehat {\mu_{\lambda}}(n)|^2
|\widehat {\nu}(n)|^2$ is finite. But
\[\int\limits_0^1 J(\lambda) \, dm(\lambda)
= \sum\limits_{n=-\infty}^\infty |\widehat {\nu}(n)|^2 \left
(\int\limits_0^1 |\widehat {\mu}(\lambda n)|^2 \,
dm(\lambda)\right ).\] Using Wiener's lemma for measures on $\R$
(e.g.\ \cite{Ka}, Chapter VI.2) and the fact that $\mu$ is
continuous
$$a_n := \int\limits_0^1 |\widehat {\mu}(\lambda n)|^2
\, dm(\lambda)=\frac1n\int_0^n|\widehat{\mu}(t)|^2\,dt \to 0$$ as
$|n| \to \infty$. So, we are seeking a suitable $\nu$ such that
\beq\label{ww10}\sum\limits_{n = -\infty}^\infty |\widehat
{\nu}(n)|^2a_n<+\infty.\eeq
Clearly, if $\nu$ were actually
absolutely continuous with respect to $m$ with a square integrable
density, then we would have this condition. But $\nu$ could not
be rigid in this situation. However, as LaFontaine points out,
the two articles of Salem~\cite{salem1,salem2} give a construction
of a Borel probability measure $\nu$ with support in $[0,1)$ with
this property, and which is also continuous and rigid. See
LaFontaine~\cite{lafontaine} and Salem~\cite{salem1,salem2} for
the details. It follows that there exists a continuous probability
measure $\nu$ supported on $[0,1)$ such that for some increasing sequence
$(n_m)$ of integers \beq\label{ww11} \widehat{\nu}(n_m)\to 1\eeq
and (\ref{ww10}) holds.

It follows from (\ref{ww10}) that for $m$-a.e.\ $\lambda\in[0,1]$
\beq\label{ww12} \sum\limits_{n=-\infty}^\infty
|\widehat{\mu_\lambda\ast\nu}(n)|^2=
\sum\limits_{n=-\infty}^\infty |\widehat
{\mu_\lambda}(n)|^2|\widehat {\nu}(n)|^2 =J(\lambda) < \infty.\eeq
The measure $\mu_\lambda\ast\nu$
is supported in $[0,1+\frac1\lambda]$ and in view
of Equation~(\ref{ww12}) and Lemma~\ref{transforms}, it must be
absolutely continuous with respect to $m$, with
$\frac {d(\mu_\lambda\ast\nu)}{dm} \in L_2(m)$.
Thus, if we choose any value of $\lambda$
satisfying~(\ref{ww12}), we have
\[\int\limits_{-\infty}^\infty
|\widehat {\mu}(\lambda t)|^2 |\widehat {\nu}(t)|^2 \,dm(t) =
\frac 1{\lambda}\int\limits_{-\infty}^\infty |\widehat
{\mu}(t)|^2|\widehat {\nu}(\frac t{\lambda})|^2\, dm(t)\]
is finite. It follows that the measure $\mu\ast \nu_{1/\lambda}$ is
compactly supported and absolutely continuous. In view of~(\ref{ww11}),
$\widehat
{\nu_{1/\lambda}}(\lambda n_m)$ converges to $1$ as $m\to \infty$.
Moreover, with respect to $m$, for almost every
$\lambda$, we would know that $(n_k\lambda: k \ge 1)$ is uniformly
distributed modulo $1$. Hence, for some subsequence $(n_{k_j})$
and some sequence of integers $(m_j)$, we have $\lim\limits_{j
\to\infty}\left( n_{k_j}\lambda - m_j\right) = 0$. Then by the
uniform continuity of $\widehat {\nu_{1/\lambda}}$, we would also
have $\widehat {\nu_{1/\lambda}}(m_j) \to 1$ as $j \to \infty$.
Hence, with respect to $m$, almost every choice of $0<\lambda < 1$
gives $\nu_{1/\lambda}$ supported on $[0,\lambda]\subset [0,1]$, which is a
continuous rigid Borel probability measure, such that $\mu\ast
\nu_{1/\lambda}$ is absolutely continuous.
\end{proof}

\begin{rem}\label{notjtrigid} The transformations $T$ and $S$ in Proposition~\ref{AA0}
must be disjoint in the sense that their only joining is their product.
See Example~\ref{egdisjoint} a).  As a first step,  assume that $T^{n_m}\to Id$
in the strong operator topology.  By passing to a subsequence if
necessary we can assume that $S^{n_m}\to \Phi$ in the
weak operator topology, where $\Phi:L_2(Y,p_Y)\to L_2(Y,p_Y)$
is a Markov operator.  We claim that
$\Phi(g)\neq g$ for each non-zero $g\in L_{2,0}(Y,p_Y)$.
Indeed suppose for some non-zero $g$, $\Phi(g)=g$.  Take any non-zero $f\in L_{2,0}(X,p_X)$. Then
$\langle (T\times S)^{n_m}(f\otimes g),f\otimes g\rangle \to \langle f\otimes g,f\otimes g\rangle$.
So the spectral
measure of $f\otimes g$ is a Dirichlet measure, contrary to construction
in Proposition~\ref{AA0}.  Now as a second step,
take a joining $J$ of $T$ and $S$.  On the operator level, this means that
we have a Markov operator $W=W_J$ corresponding to $J$ such that $W: L_2(X,p_X)\to L_2(Y,p_Y)$ and $WT=SW$.
Then $WT^{n_m}=S^{n_m}W$,
and by passing to limits we obtain $W=\Phi W$.  So by our first step,
we have $W(L_{2,0}(Y,p_Y))= \{0\}$.  But this means
that $W$ is the Markov operator for the product joining, and so $J$
is the product measure $p_T\otimes p_S$.
Thus, $T$ and $S$ are disjoint. For background,
see Glasner~\cite{Glasner}, Chapter 6, Section 2.
\end{rem}

\begin{rem} The abstract argument used in Proposition~\ref{AA0}
gives some motivation for finding concrete examples of 
continuous rigid measures whose convolution has a
Fourier transform vanishing at infinity.  For example, we were considering
this in Equation~\ref{disjoint} with the Riesz product constructions for powers of $2$ and
powers of $3$.
\end{rem}

\begin{rem} We can argue differently that if a continuous
probability measure $\rho$ supported on $[0,1]\subset\R$ satisfies
$\widehat{\rho}(r_k)\to1$ for some sequence of reals
$r_k\to\infty$, then, as a circle measure, $\rho$ is Dirichlet.
Indeed, consider the flow $V_t(f)(x)=e^{itx}f(x)$ on
$L_2(\R,\rho)$. Our assumption says that $V_{r_k}\to Id$. Consider
then $V_{\{r_k\}}$, $k\geq1$, which replaces $r_k$ by its
fractional part $\{r_k\}$. By passing to a subsequence if
necessary and using the continuity of the unitary representation
$\R\ni t\mapsto V_t$, we have $V_{\{r_k\}}\to V_s$ for some
$s\in[0,1]$. It follows that $V_{[r_k]}\to Id\circ
(V_s)^{-1}=V_{-s}$. Then, $([r_{k+1}]-[r_k])$ is a rigidity
sequence for $\rho$.\end{rem}

\begin{rem} Clearly, there is an IP version of Proposition~\ref{AA0}
that follows by passing to a subsequence of the rigidity sequence
for $S$.
\end{rem}

This result shows that whenever we have a weakly mixing rigid
transformation $T$, then there is a weakly mixing rigid transformation
$S$ such that $T\times S$ is not rigid for any sequence.  More
generally, we can prove the following result.  We again use
the closure of $\{T^n: n \in \mathbb Z\}$ in the strong operator
topology, which can be identified with the centralizer of $T$ in
case $T$ has discrete spectrum.

\begin{cor}\label{AA1} Assume that $(X,\mathcal B,p,T)$
is an ergodic dynamical system. Then $T$ has discrete spectrum if
and only if for each weakly mixing rigid system
$(Y,\mathcal B_Y,p_Y,S)$ the Cartesian product system $T\times S$
remains rigid.
\end{cor}
\begin{proof} Assume that $T$ is an ergodic rotation and let $S$
be weakly mixing, $S^{n_m}\to Id$. By passing to a subsequence if
necessary, $T^{n_m}\to R\in C(T)$, and so $T^{n_{m+1}-n_m}\to Id$
and still $S^{n_{m+1}-n_m}\to Id$; thus $T\times S$ is rigid (see
also Proposition~\ref{sparse}).

To prove the converse, suppose that $T$ does not have discrete spectrum,
but $T\times S$ is rigid for each $S$ which is rigid and weakly
mixing. Then there is some continuous $\nu$ with $\nu\ll\nu^T$. It
follows that $\nu$ is a Dirichlet measure, and if
$\widehat{\nu^T}(n_m)\to\nu^T(\mathbb T)$ then
$\widehat{\nu}(n_m)\to \nu(\mathbb T)$. Consider the Gaussian system
$G_\nu$ given by $\nu$. Then $G_\nu$ is weakly mixing and each
sequence which is a rigid sequence for $T$ is also a rigid
sequence for $G_\nu$. It follows that $G_\nu\times S$ is rigid for
each weakly mixing rigid $S$ which is in conflict with from
Proposition~\ref{AA0}.
\end{proof}

\begin{rem}  Using the viewpoint in Bergelson and Rosenblatt~\cite{bergros},
it is clear that Corollary~\ref{AA1} has a unitary version.  That is,
a unitary operator $U$ on a Hilbert space $H_U$ has discrete spectrum if and only if for every
weakly mixing rigid unitary operator $V$ on a Hilbert space $H_V$, the
product $U\times V$ on $H_U\times H_V$ is a rigid unitary operator.
\end{rem}

The following is a folklore result.

\begin{prop}\label{A2} Assume that $T$ and $S$ are ergodic
transformations with discrete spectrums. Then they are isomorphic if and only if they have the same
rigidity sequences.
\end{prop}
\begin{proof} Consider $C(T)$ and $C(S)$ respectively. Both are
given as weak closure of powers. Take the map:
$$
F:C(T)\to C(S), \;\;F(T^n)=S^n,\;n\in\Z.$$
We easily show that
this extends to a homeomorphism equivariant with rotation by $T$
and $S$ respectively.
Now refer back to the comments about centralizers at the beginning
of Section~\ref{rigidonly}.  We know that $T$ is isomorphic to the translation by $T$
on $C(T)$ considered with Haar measure and
$S$ is isomorphic to the translation by $S$
on $C(S)$ considered with Haar measure. Hence, $T$ and $S$ are
isomorphic.
\end{proof}

\begin{cor}\label{AA2} Assume that $(X,\mathcal B_X,p_X,T)$
is an ergodic transformation with a discrete spectrum, and $(Y,\mathcal B_Y,p_Y,S)$ is ergodic and has
the same rigidity sequences as $T$. Then $S$ is isomorphic to $T$.
\end{cor}
\begin{proof} It follows directly from
Corollary~\ref{AA1} that $S$ has discrete spectrum. The result
follows from Proposition~\ref{A2}.
\end{proof}

\subsection{\bf Cocycle Methods}
\label{seccocycle}

\subsubsection{\bf Tools} \label{tools}
We will now describe tools using cocycles over rotations to
produce weakly mixing transformations with a prescribed sequence
as rigidity sequences. We will start with a transformation $T$ having
discrete spectrum and its sequence of rigidity. (In fact, for
applications, we will consider one dimensional
rotations by irrational $\alpha$ and the sequence given the denominators of
$\alpha$.) Then we will consider cocycles over $T$ with values in
locally compact Abelian groups. We will then pass to the
associated unitary operators (weighted operators) and we will try
to ``lift'' some rigidity sequences for the rotation to the
weighted operator. Once such an operator has continuous spectrum
we apply the GMC which preserves rigidity. Another option to
obtain ``good'' weakly mixing transformations will be to pass to
Poisson suspensions (in case we extend by a locally compact and
not compact group) - which in a sense will be even easier as
ergodicity of Poisson suspension is closely related to the
fact that the cocycles are not coboundaries. See Remark~\ref{Poisson},
and also \cite{CFS}, ~\cite{De-Fr-Le-Pa}
and ~\cite{Ro}, for details concerning
ergodic properties of Poisson suspensions.  Note: at times
in this section $\mathbb T$ denotes the unit circle in $\mathbb C$
with multiplicative notation, and at times it will mean $[0,1)$
with addition modulo one.  The reader will be able to distinguish
which model for the circle is being used by the context of the
discussion.

\subsubsection{\bf Compact Group Extensions and Weighted Operators}
Assume that $T$ is an ergodic transformation acting on a standard
Lebesgue space $\xbm$. Let $G$ be a compact metric
Abelian group with Haar measure $\la_G$.  We take the $\la_G$
measurable sets, denoted by $\mathcal G$ as our measurable sets
for $(G,\mathcal B_G,\la_G)$. A measurable map
$\va:X\to G$ generates a
{\em cocycle} $\va^{(\cdot)}(\cdot)$ which is given by $\va^{(n)}(\cdot):X\to G$,
$n \in \mathbb Z$, by the formula, for $x \in X$,
\beq\label{cocycle}
\va^{(n)}(x)=\left\{\begin{array}{ccc}\va(x)+\va(Tx)+\ldots+\va(T^{n-1}x)&\mbox{if}&
n>0,\\
0&\mbox{if}& n=0,\\
-(\va(T^{-n}x)+\ldots+\va(T^{-1}x))&\mbox{if}&n<0.\end{array}\right.\eeq
Using $T$ and $\va$ we define a {\em compact group extension}
$T_\varphi$ of $T$ which acts on the space $(X\times
G,\cb\ot\cb_G,p\ot\la_G)$ by the formula \beq\label{gext}
T_\varphi(x,g)=(Tx,\va(x)+g)\;\;\mbox{for}\;\;(x,g)\in X\times G.
\eeq Notice that for each $n\in \Z$ and $(x,g)\in X\times G$
\beq\label{gext0} (T_\varphi)^n(x,g)=(T^nx,\va^{(n)}(x)+g).\eeq
The natural decomposition of $L_2(G,\la_G)$ using the character
group $\widehat{G}$ yields the decomposition
\beq\label{gext1}
L_2(X\times G,p\ot\la_G)=\bigoplus_{\chi\in\widehat{G}}L_2(X,p)\ot
(\C\chi).\eeq
Here $\C\chi$ is the one-dimensional subspace spanned by the
character $\chi$.
To understand ergodic and other mixing properties of
$T_\varphi$ we need to study the associated Koopman operator
$U_{T_\varphi}$,
$$U_{T_\varphi}F=F\circ T_\varphi,\;\mbox{for}\;\;F\in L_2(X\times
G,p\ot\la_G).$$ As the (closed) subspaces $L_2(X,p)\ot(\C\chi)$
in~(\ref{gext1}) are $U_{T_\varphi}$-invariant, we can examine
those mixing properties separately on all such subspaces (notice
that for $\chi=1$ we consider the original Koopman operator
$U_T$). It is well-known and not hard to see that the map
$f\ot\chi\mapsto f$ provides a spectral equivalence of
$U_{T_\varphi}|_{L_2(X,p)\ot(\C\chi)}$ and the operator
$V_{\chi\circ\va}^T$ acting on $L_2\xbm$ by the formula
\beq\label{gext2} V^T_{\chi\circ\va}f=\chi\circ\va\cdot f\circ
T\;\;\mbox{for}\;\;f\in L_2(X,p).\eeq Each such operator is an
example of a {\em weighted operator} $V^T_\xi$ over $T$, where
$\xi:X\to \mathbb T$ is a measurable function with values in
the (multiplicative) circle $\mathbb T$ and
$V_\xi^Tf=\xi\cdot f\circ T$.

Assume now that $T$ is an ergodic transformation with discrete
spectrum, i.e.\ without loss of generality, we can assume that $X$
is a compact monothetic metric group with $p=\la_X$ Haar measure
on $X$, and $Tx=x+x_0$ where $x_0$ and $\{nx_0:\:n\in\Z\}$ is
dense in $X$. Assume that $\xi:X\to \mathbb T$ is measurable.
By Helson's analysis \cite{He} (see also e.g.\ \cite{Iw-Le-Ru}):

\begin{thm}[\cite{He}]\label{helson} For $T$ as above, the maximal
spectral type of $V^T_\xi$ is either discrete or continuous. If it
is continuous then either it is singular or it is
Lebesgue.\end{thm}

We will use only the first part of this theorem. An important
practical point that comes from this theorem is that once we find
a function $f\in L_2\xbm$ such that the spectral measure $\sigma_f =
\sigma_f^{V_\xi^T}$ is continuous, then
$V^T_\xi$ has purely continuous spectrum. Consider $f=1$.
Using~(\ref{gext0}), for each $n\in\Z$, we obtain
\beq\label{gext3} \langle
\left(V^T_\xi\right)^n1,1\rangle=\int_X\xi^{(n)}(x)\,dp(x).\eeq It
follows that if there is a subsequence $(n_m)_{m\geq1}$ such that
\beq\label{gext4} \int_X\xi^{(n_m)}(x)\,dp(x)\to 0\;\Rightarrow\;
V^T_\xi\;\mbox{has continuous spectrum}.\eeq It is also nice to
note in passing that we have a bit stronger
result:\beq\label{gext5}
\int_X\xi^{(n_m)}(x)\,dp(x)\to 0\;\Rightarrow\;
\left(V^T_\xi\right)^{n_m}\to 0\;\mbox{weakly},\eeq which follows
directly from~(\ref{gext4}); indeed,
$$\int_X\xi^{(n_m)}(x)f(T^{n_m}x)\ov{f(x)}\,dp(x)\to 0$$
whenever $f$ is a character of $X$.

It is well-known and easy to check that if $(n_m)$ is a rigidity
sequence for $T$
\beq\label{gext6}
\xi^{(n_m)}\to1\;\mbox{in measure}\Rightarrow\;
\left(V^T_{\xi}\right)^{n_m}\to Id\;\mbox{strongly}.\eeq

\subsubsection{\bf $\R$-extensions, Weighted Operators and Poisson
Suspensions} We will now consider the case $G=\R$. We assume now
that $f:X\to\R$ is a cocycle for $T$ acting ergodically on a
standard probability Borel space $\xbm$. We consider $T_f$
$$
T_f(x,r)=(Tx,f(x)+r)$$ acting on $(X\times\R,p\ot\la_{\R})$. Note
that we are now on a standard Lebesgue space with a $\sigma$-finite
(and not finite) measure. In particular, constants are not
integrable.

To study spectrally $T_f$ we will write it slightly differently,
namely
$$
T_f=T_{f,\tau}$$ where $\tau$ is the natural action of $\R$ on
itself by translations: $\tau_t(r)=r+t$ and
$$
T_{f,\tau}(x,r)=(Tx,\tau_{f(x)}(r)).$$ This transformation is a
special case of so called {\em Rokhlin extension} (of $T$), see for
example  \cite{Le-Le}, and the spectral analysis below is similar to the
one in \cite{Le-Le}. So let us just imagine a slightly more
general situation
$$T_{f,\cs}(x,y)=(Tx,S_{f(x)}(y))$$ where $\cs=(S_t)$ is a flow
acting on $\ycn$ ($\nu$ can be finite or infinite). We will denote
the spectral measure of $a\in L_2\ycn$ (for the Koopman
representation $t\mapsto U_{S_t}$ on $L_2\ycn$) by
$\sigma_{a,\cs}$.

The space $L_2(X\times Y,p\ot\nu)$ is nothing but a tensor product
of two Hilbert subspaces, so to understand spectral measures we
only need to study spectral measures for tensors $a\ot b$ and we
have

$$
\int_{X\times Y}(a\ot b)\circ (T_{f,\cs})^n\cdot\ov{a\ot
b}\,dp\,d\nu$$
$$=\int_X\int_Y a(T^nx)\ov{a(x)}b(S_{f^{(n)}(x)}y)
\ov{b(y)}\,dp(x)\,d\nu(y)$$
$$=\int_Xa(T^nx)\ov{a(x)}\left(\int_Y e^{2\pi
itf^{(n)}(x)}\,d\sigma_{b,\cs}(t)\right)\,dp(x)$$
$$= \int_Y\left(\int_X e^{2\pi itf^{(n)}(x)}
a(T^nx)\ov{a(x)}\,dp(x)\right)\,d\sigma_{b,\cs}(t).$$

\begin{prop}\label{ll1}
If $T^{n_k}\to Id$ and $f^{(n_k)}\to0$ in measure then $(n_k)$ is
a rigidity sequence for $\tfs$.\end{prop}
\begin{proof}Take $a\in L^\infty(X,p)$ and notice that
by assumption for each $t\in\R$
$$
\int_X e^{2\pi
itf^{(n_k)}(x)}a(T^{n_k}x)\ov{a(x)}\,dp(x)\to\int_X|a|^2\,dp.$$ By
the Lebesgue Dominated theorem
$$\int_Y\left(\int_X e^{2\pi itf^{(n_k)}(x)}
a(T^{n_k}x)\ov{a(x)}\,dp(x)\right)d\sigma_{b,\cs}(t)\to
\int_Y\left(\int_X|a|^2\,dp\right)d\sigma_{b,\cs}=\|a\ot
b\|_{L_2(p\ot\la_{\R})}^2.$$
\end{proof}

We need more information about sequences of the form $$\int_X
e^{2\pi itf^{(n)}(x)} a(T^nx)\ov{a(x)}\,dp(x),\,\,n\in\Z.$$ In
fact, they turn out to be again Fourier coefficients of some
spectral measures. Indeed, consider $V_t$ acting on $L_2\xbm$ by
the formula
$$
(V_ta)(x)=e^{2\pi itf(x)}a(Tx).$$ This is nothing but a weighted
unitary operator and
$$
\langle V_t^n a,a\rangle=\int_X e^{2\pi itf^{(n)}(x)}
a(T^nx)\ov{a(x)}\,dp(x).$$ (Notice that we came back to the finite
measure-preserving case.)

Clearly, for $b\in L_2(\R)$ with compact support and $\mathcal S = \tau$
$$
\widehat{\sigma}_{b,\cs}(t)=\int_{\R}b\circ S_t\cdot\ov{b}\,dr
$$
$$=\int_{\R}b(r+t)\ov{b(r)}\,dr=(b\ast\ov{b})(-t).$$
Hence, the Fourier transform of $\sigma_{b,\cs}$ is square
summable, and therefore this spectral measure is absolutely
continuous. In fact, the maximal spectral type of $\cs$ is
Lebesgue, and we can see the maximal spectral type of $U_{T_f}$ as
an integral (against ``Lebesgue'' measure) of the maximal spectral
types of the family indexed by $t\in\R$ of weighted operators.

Suppose now that $U_{T_f}$ has an eigenvalue $c$, $|c|=1$. Then we
cannot have that all spectral measures $\sigma_{a\ot b,T_f}$ are
continuous. In fact, we must have that $c$ appears as an
eigenvalue for ``many'' $V_t$ (on a set of $t\in\R$ of positive
Lebesgue measure), and the following result is well-known (it is
an exercise).
\begin{lem}
The scalar $c$ is an eigenvalue of $V_t$ if and only if we can solve the
following functional equation:
$$
e^{2\pi itf}=\frac {c\cdot \xi\circ T}{\xi}$$ in measurable functions
$\xi:X\to\mathbb T$.
\end{lem}

It follows that having an eigenvalue for $U_{T_f}$ means that we
can solve the above multiplicative equations on a set of positive
Lebesgue measure of $t\in\R$. We are now in the framework of the
classical Helson's problem (e.g.\ \cite{Mo-Sch}) of passing from
multiplicative coboundaries to additive coboundaries. Using known
results in this area (\cite{Mo-Sch}, and see also the appendix in
\cite{Le-Pa}) we obtain the following (remember that constant
functions are not elements of $L_2(X\times\R,p\otimes\la_{\R})$).

\begin{prop}\label{He}
If $U_{T_f}$ has an eigenvalue then $f$ is an additive
quasi-coboundary, that is there exist a measurable $g:X\to\R$ and
$r\in\R$ such that $f(x)=r+g(x)-g(Tx)$ for $p$-a.e.\ $x\in X$.
\end{prop}

\vspace{2mm}

Therefore, if $f$ is a non-trivial cocycle then automatically
$U_{T_f}$ has continuous spectrum and classically the Poisson
suspension over $T_f$ is ergodic, hence weakly mixing (see Remark~\ref{Poisson}). Recall
that from spectral point of view Poisson suspension over
$(X\times \R,p\ot\la_{\R},T_f)$ will be the same as Gaussian
functor over $(L_2(X\times\R,p\ot\la_{\R}),U_{T_f})$. In
particular, if $(T_f)^{n_t}\to Id$ on
$L_2(X\times\R,p\ot\la_{\R})$ then $(n_t)$ will be a rigidity
sequence for the suspension (in view of Proposition~\ref{ll1}).
 From the above discussion, it follows that to have a weakly mixing
 transformation $\widetilde {T }_f$ with a rigidity sequence $(N_t)$, we need:
(i) $f$ is not an additive coboundary,\, (ii) $T^{N_t}\to Id$,\, (iii) $f^{(N_t)}\to 0$ in measure.

\subsubsection{\bf Denominators of $\alpha$ and Rigidity} \label{contfrac}

We have already seen that the sequence $(2^n)$ is a rigidity
sequence for a weakly mixing transformation. We can construct some
other explicit examples of rigidity sequences by using known
results from the theory of ``smooth'' cocycles over
one-dimensional rotations. This will allow us to show that if
$\alpha$ is irrational, and $(q_n)$ stands for its sequence of
denominators then $(q_n)$ is also a rigidity sequence for a weakly
mixing transformation. The most interesting case is of course the
bounded partial quotient case (for example for the Golden Mean).

So assume that $\va:\T\to\R$ is a smooth mean-zero cocycle. We use
the term ``smooth'' here in a not very precise way; it may refer
to a good speed of decaying of the Fourier transform of $\va$.

We recall first that one of consequences of the Denjoy-Koksma
Inequality for $AC_0$ (absolutely continuous mean-zero) cocycles
is that \beq\label{DK} \va^{(q_n)}\to 0\;\;\mbox{uniformly}\eeq
for every irrational rotation by $\alpha$, see \cite{Herman}.
Another type of Denjoy-Koksma inequality has been proved in
\cite{Aa-Le-Ma-Na} for functions $\va$ whose Fourier transform is
of order $\mbox{O}(1/|n|)$ -- as its consequence we have the
following: \beq\label{DK1} \mbox{If
$\widehat{\va}(n)=\mbox{o}(\frac1{|n|})$, $\widehat{\va}(0)=0$
then $\va^{(q_n)}\to0$ in measure}\eeq for every
rotation by an irrational $\alpha$.

We would like also to recall another (unpublished) result by M.
Herman \cite{Herman1}.  While this may not be available, one
can see also Krzy\.zewski~\cite{Kr} for generalizations of Herman's result.

\begin{thm}\label{MH} Assume that a mean-zero $\va:\T\to\R$ is in
$L_2(\T,\la_{\T})$ and its Fourier transform is concentrated on a
lacunary subset of $\Z$. Suppose that
$$
\va(x)=g(x)-g(x+\alpha), \;\;\la_{\T}-\mbox{a.e.}$$ for some
irrational $\alpha\in[0,1)$. Then $g\in L_2(\T,\la_{\T})$.
\end{thm}

Fix $\alpha\in[0,1)$ irrational, and let
$\alpha=[0:a_1,a_2,\ldots]$ stand for the continued fraction
expansion of $\alpha$. Denote by $(q_n)$ the sequence of
denominators of $\alpha$: $q_0=1$, $q_1=a_1$ and
$q_{n+1}=a_{n+1}q_n+q_{n-1}$ for $n\geq2$. Then
$$
\frac{q_{n+2}}{q_{n}}=\frac{a_{n+2}q_{n+1}+q_n}{q_{n}}\geq
a_{n+2}+1\geq2.$$ It follows that \beq\label{lac1}
(q_{2n})\;\mbox{is lacunary}.\eeq Moreover, \beq\label{lac2}
q_n\|q_n\alpha\|\leq 1\;\mbox{for each}\;n\geq1.\eeq We define
$\va:\T\to\R$ by \beq\label{lac3} \va(x)=\sum\limits_{n=0}^\infty
a_{q_{2n}}\cos2\pi iq_{2n}x\eeq where for $n\geq 1$
\beq\label{lac3a} a_{q_{2n}}=\frac1{\sqrt n}\|q_{2n}\alpha\|.\eeq
We then have $\va:\T\to\R$, $\widehat{\va}(n)=\mbox{o}(1/|n|)$ and
$(a_{q_{2n}})\in l_2$ in view of~(\ref{lac2}). Now, suppose that
\beq\label{lac4}\va(x)=g(x)-g(x+\alpha)\eeq for a measurable
$g:\T\to\R$. In view of Theorem~\ref{MH} and~(\ref{lac1}), $g\in
L_2(\T,\la_{\T})$. Hence,
$$
g(x)=\sum\limits_{k=-\infty}^\infty b_k e^{2\pi ikx}.$$ Furthermore, by
comparing Fourier coefficients on both sides in~(\ref{lac4}),
$$
b_k=0\;\;\mbox{if}\;k\neq
q_{2n}\;\mbox{and}\;\;b_{q_{2n}}=a_{q_{2n}}/(1-e^{2\pi
iq_{2n}\alpha})$$ with $(b_{q_{2n}})\in l_2$. However
$$
|b_{q_{2n}}|=a_{q_{2n}}/\|q_{2n}\alpha\|=1/\sqrt n$$ which is a
contradiction. We hence proved the following.

\begin{prop}\label{lac6} For each irrational $\alpha\in[0,1)$
there is a mean-zero $\va:\T\to\R$ such that
$\widehat{\va}(n)=\mbox{o}(1/|n|)$ and $\va$ is not an additive
coboundary.\end{prop}

Using~(\ref{DK1}) and Proposition~\ref{ll1} we hence obtain the
following.

\begin{prop}\label{lac7} For $\va$ satisfying the assertion of
Proposition~\ref{lac6} the sequence $(q_n)$ of denominators of
$\alpha$ is a rigidity sequence of $T_\va$ on
$L_2(\T\times\R,\la_{\T}\ot\la_{\R})$ and $U_{T_{\va}}$ has
continuous spectrum.\end{prop}

By using GMC method or by passing to the relevant Poisson
suspension we obtain:

\begin{cor}\label{denominators} For each sequence $(q_n)$ of
denominators there exists a
weakly mixing transformation $R$ such that $R^{q_n}\to
Id$.\end{cor}

\begin{rem} \label{wazna}(i) We would like to emphasize that in general the
sequence $(q_n)$ of denominators of $\alpha$ is not lacunary.
Indeed, assume that $\alpha=[0:a_1,a_2,\ldots]$ stands for the
continued fraction expansion of $\alpha$. Suppose that for a
subsequence $(n_k)$ we have $a_{n_k+1}=1$, $a_{n_k}\to\infty$.
Then by the recurrence formula $q_{m+1}=a_{m+1}q_m+q_{m-1}$ we
obtain that
$$
\liminf_{n\to\infty}\frac{q_{n+1}}{q_n}=1.$$ So the
sequences of denominators are another type of non-lacunary
sequences which can be realized as rigidity sequences for weakly
mixing transformations, besides the ones in Section~\ref{ratesofgrowth}.
\medskip

\noindent (ii) As the above shows $\{q_n:\:n\geq1\}$ is always a Sidon set
(see~\cite{Ka},\cite{Ru}); indeed,
$$
\{q_n:\:n\geq1\}=\{q_{2n}:n\geq1\}\cup\{q_{2n+1}:\:n\geq 0\}.$$ It
follows that the set of denominators is the union of two lacunary
sets and hence is a Sidon set (\cite{Ka},\cite{Ru}).
\medskip

\noindent (iii) The assertion of Theorem~\ref{MH} is true for functions
whose Fourier transform is concentrated on a Sidon set and when
$T$ is an arbitrary ergodic rotation on a compact metric Abelian
group (by the proof of the main result in \cite{He} or by
\cite{Kr}).
\medskip

\noindent (iv) It follows that to prove Proposition~\ref{lac6} we could have
used all denominators, with (for example) $a_{q_n}=\frac1{\sqrt
n}\|q_n\alpha\|$.
\medskip

\noindent (v) Eisner and Grivaux~\cite{EG} obtain some results in the
direction of Corollary~\ref{denominators}, but their examples
are restricted to badly approximated irrational numbers.
\end{rem}

There are also more complicated constructions showing that for
each irrational $\alpha$ there is an absolutely continuous
mean-zero $\va:\T\to\R$ which is not a coboundary --
see~\cite{Li-Vo}. We can then use such cocycles and~(\ref{DK1})
for another proof of the above corollary.

Here is a concrete example of Corollary~\ref{denominators}
using the continued fraction expansion of the Golden Mean.

\begin{cor} The Fibonacci
sequence is a rigidity sequence for some weakly mixing transformation .
\end{cor}

\begin{rem} Suppose we take a
increasing sequence like the Fibonacci sequence, which is obtained by recursion.
That is, we have
$z = F(x_1,\ldots,x_K) = \sum\limits_{k=1}^K c_kx_k$ where $c_k$
are whole numbers, and we have
$n_{m+1} = F(n_m,\ldots,n_{m-K+1})$ for all $m$.
Is this always a rigidity sequence for a weakly mixing transformation?
\end{rem}

\begin{rem} It is not clear how to characterize rigidity sequences
that cannot be IP rigidity sequences. See
Proposition~\ref{powers}, and its generalization
Proposition~\ref{ipodometer} for examples of this phenomenon. In
reference to the above, it would be interesting to show that
$FS((q_n))$ is not a rigidity net.
\end{rem}

\begin{rem} The above results can be used to answer positively the
following question: {\em Given an increasing sequence $(n_m)$ of
integers is there a weakly mixing transformation $R$ such that
$R^{n_{m_k}}\to Id$ for some subsequence $(n_{m_k})$ of $(n_m)$?}
In fact, we have already answered this question (see
Proposition~\ref{folklore1}), but we will now take a very different
approach. We start with the following well-known lemma; see
for example ~\cite{Kw-Le-Ru}.
\begin{lem} Given an increasing sequence $(n_m)$ of natural
numbers, consider the set of $\alpha \in [0,1)$ such that a
subsequence of $(n_m)$ is a subsequence of denominators of
$\alpha$. This is a generic subset of $[0,1)$.
\end{lem}
\noindent Now, given $(n_m)$ choose any irrational $\alpha$ so that for some
subsequence $(n_{m_k})$ we have all numbers $n_{m_k}$ being
denominators of $\alpha$. Then use previous arguments to construct a weakly
mixing ``realization'' of the whole sequence of denominators of
$\alpha$.
\end{rem}

\subsubsection{\bf Integer Lacunarity Case}
We will now give an alternative proof of Proposition~\ref{integerratios}
using the cocycle methods that have been developed here.
Assume that $(n_m)_{m\geq0}$ is an increasing sequence of positive
integers such that $n_0=1$, $n_{m+1}/n_m\in\Z$ with
\beq\label{odo1} \rho_m:=n_{m+1}/n_m\geq2\;\;\mbox{for}\;m\geq0.
\eeq Notice that in view of~(\ref{odo1}) there exists a constant
$C>0$ such that \beq\label{odo2}
\frac1{n_{m+1}^2}+\frac1{n_{m+2}^2}+\ldots\leq \frac
C{n_{m}^2}\;\;\mbox{for each $m\geq0$}.\eeq

Let $X=\Pi_{m=1}^\infty\{0,1,\ldots,\rho_m-1\}$ which is a
metrizable compact group when we consider the product topology and
the addition is meant coordinatewise with carrying the remainder
to the right. On $X$ we consider Haar measure $p_X$ which is the
usual product measure of uniform measures. Define $Tx=x+\hat 1$,
where
$$
\hat 1=(1,0,0,\ldots).$$ The resulting dynamical system is called
the $(n_m)$-{\em odometer}.

For each $t\geq 0$ set
$$
D_0^{n_m}=\{x\in X:\:x_0=x_1=\ldots=x_{m-1}=0\}.$$ Note that
$\{D^{n_m}_0,TD^{n_m}_0,\ldots,T^{n_m-1}D^{n_m}_0\}$ is a Rokhlin
tower fulfilling the whole space $X$ and \beq\label{odo3}\hat 1\in
TD^{n_m}_0 \;\;\mbox{for each $m\geq0$}.\eeq

The character group $\widehat X$ of $X$ is discrete and is
isomorphic to the (discrete) group of roots of unity of
degree~$n_m$, $m\geq0$. More precisely, for $m\geq0$ set
$$
1_{n_m}(x)=\vep^j_{n_m}:=e^{2\pi ij/n_m}\;\;\mbox{for $x\in
T^jD^{n_m}_0$,}\;j=0,1,\ldots, n_m-1.$$ Then $\widehat
X=\{1_{n_m}^j:\:j=0,1,\ldots,n_m-1,m\geq0\}$.

 From now on we will consider $f\in L_{2,0}(X,p_X)$ whose Fourier
transform is ``concentrated'' on $\{1_{n_m}:\:m\geq0\}$. We
have \beq\label{odo4} f(x)=\sum\limits_{m=1}^\infty
a_{n_m}1_{n_m}(x),\;\sum\limits_{m=1}^\infty|a_{n_m}|^2<+\infty.\eeq

{\bf A) Small divisors.} Assume that $f$ satisfies~(\ref{odo4})
and suppose that \beq\label{odo5}
f(x)=g(x)-g(x+\hat 1)\;\;\mbox{for $p_X$-a.e.}\;x\in X.\eeq Suppose
moreover that $g\in L_2(X,p_X)$. Hence
$g(x)=\sum\limits_{\chi\in\widehat{X}}b_{\chi}\chi(x)$ and by comparison
of Fourier coefficients on both sides in~(\ref{odo4}) we obtain
$$
b_\chi=0\;\mbox{whenever $\chi\neq 1_{n_m}$ and
$a_{n_m}=b_{n_m}(1-1_{n_m}(\hat 1))$ for $m\geq1$}.$$
Using~(\ref{odo3}) we obtain that
$b_{n_m}=\frac{a_{n_m}}{1-\vep_{n_m}}$, so
$$
|b_{n_m}|^2=\frac{|a_{n_m}|^2}{|1-\vep_{n_m}|^2}=n_m^2|
a_{n_m}|^2\;\;\mbox{for $m\geq1$}.$$ We have proved the following
\beq\label{odo6} \mbox{(\ref{odo5}) has an $L_2$-solution if and
only if $(n_ma_{n_m})_m\in l_2$.}\eeq

{\bf B) Estimate of $L_2$-norms for the cocycle.} Assume that
$n\in\N$ then
$$
f^{(n)}(x)=\sum\limits_{m=1}^\infty a_{n_m}(1+1_{n_m}(\hat 1)+\ldots+
1_{n_m}((n-1)\hat 1))1_{n_m}(x)
$$$$=\sum\limits_{m=1}^\infty a_{n_m}\left(\sum_{j=0}^{n-1}
\vep_{n_m}^j\right)1_{n_m}(x).$$ Fix $n=n_{m_0}$. We then have
$$f^{(n_{m_0})}(x)=\sum_{m=1}^{m_0} a_{n_m}\left(
\sum\limits_{j=0}^{n_{m_0}-1} \vep_{n_m}^j\right)1_{n_m}(x)+
\sum\limits_{m=m_0+1}^\infty a_{n_m}\left(\sum_{j=0}^{n_{m_0}-1}
\vep_{n_m}^j\right)1_{n_m}(x)$$$$ =\sum\limits_{m=m_0+1}^\infty
a_{n_m}\frac{1-\vep_{n_m}^{n_{m_0}}}{1-\vep_{n_m}}1_{n_m}(x).$$
Since $|1-\vep_{n_m}|=1/n_m$ and $|1-\vep^{n_{m_0}}_{n_m}|\leq
n_{m_0}/n_t$, \beq\label{odo7}
\|f^{(n_{m_0})}\|^2_{L_2(X,p_X)}\leq
n_{m_0}^2\sum\limits_{m=m_0+1}^\infty |a_{n_m}|^2.\eeq

{\bf C) Sidon sets and a ``good'' function.} According to
\cite{Ru} (see Example 5.7.6 therein) every infinite subset of a
discrete group contains an infinite Sidon set. Hence we can choose
a subsequence $(n_{m_k})$ of $(n_m)$ so that \beq\label{odo8}
\mbox{$\{1_{n_{m_k}}:\:k\geq1\}$ is a Sidon subset of $\widehat
X$}.\eeq We set \beq\label{odo9} f(x)=\sum\limits_{k=1}^\infty
\frac1{\sqrt kn_{m_k}}1_{n_{m_k}}(x).\eeq Suppose now
that~(\ref{odo5}) has a measurable solution $g$ (we should
consider $f$ real valued, so in fact we should consider
$f+\overline{f}$ below). In view of~(\ref{odo8}) and~\cite{Kr},
$g\in L_2(X,p_X)$. But
$$\sum\limits_{m=1}^\infty|n_m a_{n_m}|^2=\sum\limits_{k\geq1}\frac1k,$$
so by~(\ref{odo6}), we cannot obtain an $L_2$-solution. This
means that $f$ is not a measurable coboundary. According
to~(\ref{odo7}), the definition of $f$ and~(\ref{odo1}) for each
$s\geq1$
$$
\|f^{(n_{s})}\|^2_{L_2(X,p_X)}\leq n_s^2\sum\limits_{m=s+1}^\infty
|a_{n_m}|^2= n_s^2\sum\limits_{k\geq1:\:n_{m_k}\geq n_{s+1}}^\infty
|a_{n_{m_k}}|^2=n_s^2\sum\limits_{k=k_s}^\infty |a_{n_{m_k}}|^2
$$$$
\leq \frac{n_s^2}{k_s}\sum\limits_{j=s+1}^\infty
\frac1{n_{j}^2}\leq\frac C{k_s}\to0$$ as clearly $k_s\to\infty$
when $s\to\infty$.

Using our general method we hence proved the following.

\begin{prop}\label{intlac} Assume that $(n_m)$ is an increasing
sequence of integers with $n_{m+1}/n_m$ being an integer at
least~2. Then there exists a weakly mixing transformation $R$ such
that $R^{n_m}\to Id$.\end{prop}

We will now discuss the problem of $IP$-rigidity along $(n_m)$.
First of all notice that $(n_m)$ is a sequence of $IP$ rigidity
for the $(n_m)$-odometer (indeed, $\sum\limits_{m=1}^\infty
|1_{n_m}(\hat 1)-1|<+\infty$).

Let us also notice that if $R$ is weakly mixing and $R^{n_m}\to
Id$ then by passing to a subsequence, we will get
$IP-R^{n_{m_k}}\to Id$. But if we then set $m_k=n_{m_k}$ then
$m_{k}$ divides $m_{k+1}$ and $(m_k)$ is a sequence of $IP$-rigidity
for $R$. It means that if the sequence $(n_{m+1}/n_m)$ is
unbounded then, at least in some cases, it is a sequence of
$IP$-rigidity for a weakly mixing transformation. On the other
hand we have already seen (Corollary~\ref{powers}) that when
$\rho_t=a$, $t\geq1$ then $IP$-rigidity does not take place. The
proposition below generalizes that result and shows that in the
bounded case an $IP$-rigidity is excluded.

\begin{prop}\label{ipodometer} Assume that
$n_{t+1}/n_t\in\N\setminus\{0,1\}$, $t\geq0$, and
$\sup_{t\geq0}n_{t+1}/n_t=:C<+\infty$. Then for any weakly mixing
transformation $R$ for which $R^{n_t}\to Id$, the sequence $(n_t)$
is not a sequence of $IP$-rigidity.
\end{prop}
\begin{proof} Recall that $Tx=x+\hat 1$ where $X$ stands for the
$(n_t)$-odometer. Each natural number $r\geq1$ can be expressed in
a unique manner as
$$
r=\sum\limits_{t=0}^Na_tn_t,\; 0\leq a_t<\rho_t=n_{t+1}/n_t.$$ Assume
that $T^{r_m}\to Id_X$. Write
$$
r_m=\sum\limits_{t=0}^{N_m}a^{(m)}_tn_t,\; 0\leq a^{(m)}_t<\rho_t$$ and
set $k_m=\max\{t\geq0:\:a^{(m)}_0=\ldots=a^{(m)}_t=0\}$ (so
$r_m=\sum\limits_{t=k_m}^{N_m}a^{(m)}_tn_t$). We claim that
\beq\label{cla1} k_m\to\infty\;\;\mbox{whenever}\;m\to\infty.\eeq
Indeed, suppose that the claim does not hold. Then without loss of
generality we can assume that there exists $t_0\geq0$ such that
$k_m=t_0$ for all $m\geq1$, that is
$$
r_m=a^{(m)}_{t_0}n_{t_0}+\sum\limits_{t=t_0+1}^{N_m}a^{(m)}_tn_t
\;\;\mbox{with}\;1\leq a_{t_0}^{(m)}<\rho_{t_0}. $$ Consider the
tower $\{D^{n_{t_0+1}}_0,\ldots,D^{n_{t_0+1}}_{n_{t_0+1}-1}\}$ and let
$A=D^{n_{t_0+1}}_0$. Notice that for each $i\geq0$ and $j\geq1$ we
have $T^{in_{t_0+j}}A=A$. It follows that
$$
T^{r_m}A=T^{a^{(m)}_{t_0}n_{t_0}}\left(T^{
\sum\limits_{t=t_0+1}^{N_m}a^{(m)}_tn_t}(A)\right)$$$$=
T^{a^{(m)}_{t_0}n_{t_0}}(A)\in
\{D^{n_{t_0+1}}_1,\ldots,D^{n_{t_0+1}}_{n_{t_0+1}-1}\},$$ where the latter
follows from the fact that $1\leq a_{t_0}^{(m)}<\rho_{t_0}$. So
$p_X(T^{r_m}(A)\triangle A)=2/n_{t_0+1}$ and hence
$(r_m)$ is not a rigidity sequence for $T$, a contradiction.
Thus~(\ref{cla1}) has been shown.

Assume now that $R$ is a weakly mixing transformation for which
$(n_t)$ is its $IP$-rigidity sequence. It follows that we have a
convergence along the net \beq\label{cla2}
R^{\sum\limits_{t=k}^N\eta_tn_t}\to
Id\;\;\mbox{whenever}\;\;k\to\infty\;\mbox{and}\;\eta_k=1,
0\leq\eta_t\leq1,t\geq k+1.\eeq We claim that also
\beq\label{cla3} R^{\sum_{t=k}^Na_tn_t}\to
Id\;\;\mbox{whenever}\;\;k\to\infty\;\;\mbox{and}\;1\leq
a_k\leq\rho_k-1, 0\leq a_t\leq\rho_t-1, t\geq k+1.\eeq Indeed, we
write $R^{\sum\limits_{t=k}^Na_tn_t}$ as the composition of at most
$S_1\circ\ldots \circ S_D$ with $D\leq C=\max_t\rho_t$
automorphisms of the form $R^{\sum\limits_{t=k}^N\delta_tn_t}$ with
$\delta_t\in\{0,1\}$ (to define the first automorphism $S_1$ we
put $\delta_t=1$ as soon as $a_t\geq1$ and $\delta_t=0$ elsewhere,
for the second automorphism $S_2$ we put $\delta_t=1$ as soon as
$a_t\geq 2$ and $\delta_t=0$ elsewhere, etc.). Notice that for
each $i=1,\ldots,D$
$$
S_i=R^{\sum_{t=k_i}^Nn_t}\;\;\mbox{with}\;\;k_i\geq k.$$
Therefore, if we assume that in~(\ref{cla2}),
$\|R^{\sum_{t=k}^Nn_t}\|<\vep$ for $k\geq K$ then
$\|S_1\circ\ldots\circ S_D\|<D\vep$ (see Remark~\ref{metric1}) and
thus~(\ref{cla3}) follows.

Combining (\ref{cla3}) and (\ref{cla1}) we see that each rigidity
sequence for $T$ is also a rigidity sequence for $R$. This
however contradicts Corollary~\ref{AA2} (or rather to its
proof).
\end{proof}

{\bf Question.} If $(\rho_t)$ is bounded, but not always a whole
number, can it still $(n_t)$ be a sequence of $IP$-rigidity for some weakly
mixing transformation? How fast does $(\rho_t)$ have to grow for
$(n_t)$ to be a sequence of $IP$-rigidity for some weakly
mixing transformation?

\section{\bf Non-Recurrence} \label{nonrecurrence}

In the previous sections, we have seen that the characterization
of which sequences $(n_m)$ exhibit rigidity for some ergodic, or
more specifically weakly mixing dynamical system will, most
certainly be difficult. The only aspect that is totally clear at this time is that
these sequences must have density zero because their gaps tend to
infinity. Lacunary sequences are always candidates for
consideration in such a situation. However, we have seen in Remark~\ref{linformeg} d)
that there are lacunary sequences which cannot be rigidity sequences for
even ergodic transformations, let alone weakly mixing ones.

In a similar vein, we would like to characterize which increasing
sequences $(n_m)$ in $\mathbb Z^+$ are not recurrent for some
ergodic dynamical system i.e. $(n_m)$ has the property that for some
ergodic system $(X,\mathcal B,p,T)$ and some set $A$ of positive
measure, the sets $T^{n_m}A$ are disjoint from $A$ for all $m \ge
1$. So we are taking {\em recurrence along $(n_m)$} here to mean that $p(T^{n_m}A \cap A) > 0$
for some $m$.  A central unanswered question is the following
\medskip

\noindent {\bf Question:} Is it the case that any lacunary sequence is
a sequence of non-recurrence for a weakly mixing system?
\medskip

\begin{rem} \label{unions} It is not hard to see that any lacunary sequence
fails to be a recurrent sequence for
some ergodic dynamical system. Indeed, this happens even with
ergodic rotations of $\mathbb T$. See Pollington~\cite{Poll},
de Mathan~\cite{deM}, and Furstenberg~\cite{Furst}, p. 220. They show that for any lacunary
sequence $(n_m)$ there is some $\gamma\in \mathbb T$ of infinite
order, and some $\delta > 0$, such $|\gamma^{n_m} -1| \ge \delta$
for all $m \ge 1$. The arguments there also give information
about the size of the set of rotations that work for a given
lacunary sequence. The constructions in these articles are made
more difficult, as with a number of other results about lacunary
sequences, by not knowing the degree or nature of lacunarity. But
if one just wants some ergodic dynamical system to exhibit
non-recurrence, then the construction is easier. This was observed
by Furstenberg~\cite{Furst}. In short, his argument goes like
this. Suppose
$(n_k)$ is lacunary, say $\frac {n_{k+1}}{n_k} \ge \lambda > 1$
for all $k \ge 1$. Depending only on $\lambda$, we can choose $K$
so that the subsequences $(p_{m,j}:m\ge 1)$ given by $p_{m,j} = n_{j+Km},
j = 0,\ldots,K-1$ each have lacunary constant $\inf\limits_{m \ge
1} \frac {p_{m+1,j}}{p_{m,j}} \ge 5$. Then for each $j$, a
standard argument shows that there is a closed perfect set $C_j$
such that for all $\gamma \in C_j$, we have $|\gamma^{p_{m,j}}- 1|
\ge \frac 1{100}$ for all $m\ge 1$. We can choose
$(\gamma_0,\ldots,\gamma_{K-1})$ with each $\gamma_j \in C_j$ and
such that $\gamma_0,\ldots,\gamma_{K-1}$ are independent. Then we
would know that the transformation $T$ of the $K$-torus $\mathbb
T^K$ given by $T(\alpha_0,\ldots,\alpha_{K-1}) =
(\gamma_0\alpha_0,\ldots,\gamma_{K-1}\alpha_{K-1})$ is ergodic
Also, for a sufficiently small $\epsilon$, the arc $I$ of radius
$\epsilon$ around $1$ in $\mathbb T$ will give a set $C =
I\times\ldots\times I \subset T^K$ such that for all $n_k$, we have
$T^{n_k}C$ and $C$ disjoint. Indeed, each $n_k$ is some $p_{m,j}$
and so $T^{n_k}C$ and $C$ are disjoint because in the $j$-th
coordinate $T^{n_k}$ corresponds to the rotation
$\gamma_j^{p_{m,j}} I$ which is disjoint from $I$.
\end{rem}

The main idea in Remark~\ref{unions} that appears in Furstenberg
~\cite{Furst} gives us this basic principle.

\begin{prop} \label{wmunions} If a sequence $\mathbf n$ is a finite
union of sequences $\mathbf n_i,i=1,\ldots, I$, each of which is a
sequence of
non-recurrence for some weakly mixing transformation $T_i$, then
$\mathbf n$ is also a sequence of non-recurrence for a weakly mixing
transformation.
\end{prop}
\begin{proof} Write $\mathbf n_i = (\mathbf n_i(j):j\ge 1)$.
We take $T$ = $T_1\times\ldots\times T_I$. This is a weakly
mixing transformation since each $T_i$ is weakly mixing. There is
a set $C_i$ such that $T_i^{\mathbf n_i(j)}C_i$ is disjoint from $C_i$
for all $j$. So $C = C_1\times\ldots\times C_i$ has the property
that $T^{\mathbf n_i(j)}C$ is disjoint from $C$ for all $i$ and all
$j$.  That is, the sequence $T_1\times\ldots\times T_I$ is not recurrent
along $\mathbf n$ for the set $C$.
\end{proof}

\begin{rem} \label{different} a) This property of non-recurrent sequences
does not hold for rigidity sequences. For example, consider $\mathbf n_1 = (2^{m^2})$ and
$\mathbf n_2 = (2^{m^2}+1)$.  By Proposition~\ref{fastworks} below, these
are both sequences of non-recurrence for a weakly mixing transformation
and hence by the above their union is too.  These two sequences are
also rigidity sequences for ergodic rotations and weakly
mixing transformations
by Proposition~\ref{ratiogrows}.
But  the union of these two sequences is not a rigidity
sequence for an ergodic transformation
because a rigidity sequence for an ergodic transformation cannot have infinitely many terms differing
by $1$.
\medskip

\noindent b) Here is a related example that shows how rigidity sequences
and non-recurrent sequences behave differently.  The sequence $A = (p: p\,\,\text {prime})$
is not recurrent but the sequence  $A = (p - 1: p\,\,\text {prime})$ is recurrent.  See
S\'ark\"ozy~\cite{sark} and apply the Furstenberg Correspondence Principle.  But we see from either
Proposition~\ref{unifdist} or Proposition~\ref{sumset}
that neither sequence is a rigidity sequence for an ergodic transformation.
\end{rem}

We see that the non-recurrence phenomenon is both pervasive
and not, depending on how one chooses the quantifiers. As usual,
the measure-preserving transformations of a non-atomic separable
probability space $(X,\mathcal B, p)$ can be given the weak topology,
and become a complete pseudo-metric group $\mathcal G$ in this
topology. By a generic transformation, we mean an element in
a set that contains some dense $G_\delta$ set in $\mathcal G$.

\begin{prop} \label{revise} The generic transformation is both weakly
mixing and rigid,
and moreover is not recurrent along some increasing sequence
in $\mathbb Z^+$.
\end{prop}
\begin{proof} Remark~\ref{wmandrigid} pointed out that
the generic transformation is
weakly mixing and rigid. Fix such a transformation $T$ and
some $(n_m)$ such that $\|f\circ
T^{n_m} - f\|_2 \to 0$ for all $f \in L_2(X,p)$. Then it
follows that for any set $A$, $\lim\limits_{m \to \infty}
p(T^{n_m}A \Delta A) = 0$. Since $T$ is ergodic, we can choose a
set $A$ with $p(A) > 0$ and $T A$ and $A$ disjoint. Let $B = T
A$. We have $B$ and $A$ disjoint, and $\lim\limits_{m\to \infty}
p(T^{n_m-1}B \Delta A) = 0$. Hence, we can pass to a subsequence
$(m_s)$ so that $\sum\limits_{s=1}^\infty p(T^{n_{m_s}-1}B \Delta
A) \le \frac 1{100} p(A)= \frac 1{100}p(B)$. So
$$C = B\backslash
\bigcup\limits_{s=1}^\infty T^{-(n_{m_s}-1)}(T^{n_{m_s}-1}B\Delta
A)$$ will have $p(C) > 0$. Also, $C \subset B=TA$ is disjoint
from $A$. But at the same time
\[T^{n_{m_s}-1}C\subset T^{n_{m_s}-1}B\setminus\left(T^{n_{m_s}-1}B\triangle
A\right)=T^{n_{m_s}-1}B\cap A\subset A\]
for all $s\geq 1$.

Thus, the generic transformation is weakly mixing and rigid, and
additionally for some
sequence of powers $T^{n_{m_s}}$ and some set $C$ of positive measure,
we have non-recurrence because $C$ is disjoint from all
$T^{n_{m_s}}C$.
\end{proof}

\begin{rem}\label{notrecshift} a) This argument can easily be
used to show that if $T$ is rigid along $(n_m)$, then for any
$K$, by passing to a subsequence $(n_{m_s})$, we can have for
each $k \not= 0, |k| \le K$, the transformation $T$ is non-recurrent
along $(n_{m_s} + k)$ for some set $C_k, p(C_k) > 0$.  By taking $S$ to be
a product of $T$ with itself $2K$ times, we can arrange that the
weakly mixing transformation $S$ is rigid along $(n_m)$ and
there is one set $C, p(C) > 0$, such that $S$ is non-recurrent for $C$
along each of the sequences $(n_{m_s} + k)$ with $0 < |k| \le K$.
\medskip

\noindent b) One cannot restrict the sequence along which the non-recurrence
is to occur.  It is not hard to see that the class of transformations that is
non-recurrent for some set along a fixed sequence is a meager set of transformations.
\end{rem}

We do have some specific, interesting examples of the failure of
recurrence for a weakly mixing dynamical system.  See Chacon~\cite{Chacon}
for the construction of the rank one Chacon transformation.  The important point
here is that the non-recurrence occurs along a lacunary sequence $(n_m)$ with
ratios $n_{m+1}/n_m$ bounded. See Remark~\ref{moreChacon} a) for more information
about this example.

\begin{prop}\label{Chacon}
The Chacon transformation is not recurrent for sequence $(n_m) = (\frac
{3^{m+1} -1}2 -1)$.
\end{prop}
\begin{proof} Let $n_m = \frac
{3^{m+1} -1}2 -1$. The transformation $T$ that we are using here is
constructed
inductively as follows. Take the current stack (single tower)
$T_m$ of interval and cut it in thirds $T_{m,1}$, $T_{m,2}$, and
$T_{m,3}$. Add a spacer $s$ of the size of levels above the
middle third $T_{m,2}$, and let the new stack $T_{m+1}$ consists of
$T_{m,1}$, $T_{m,2}$, $s$, and $T_{m,3}$ in that order from bottom
to top. We take $T_0$ to be $[0,1)$ to start this construction.
So it is easy to see that the height $h_m$ of our $m$-th tower
is $\frac {3^{m+1} -1}2$. To see the failure of recurrence, use the standard
symbolic dynamics for $T$ i.e. assign the symbol $1$ to all the
spacer levels and $0$ to the rest of the levels. Then let $B_m$
be the name of length $h_m$ of a point in the base of the $m$-th
tower $T_m$. Then $B_0 = 0$, $B_1 = 0010$, and $B_{m+1} = B_m\
B_m\ 1\ B_m$ in general. It is a routine check that when one
shifts $B_k$ by $n_m = h_m-1$ for $k$ larger than $m$, and compares this
with $B_k$, then they have no common occurrences of $1$. So if $A$
is the first added spacer level, then we have $T^{n_m}A$ and $A$
disjoint for all $m$.
\end{proof}

\begin{rem} \label{moreChacon} a)  The Chacon transformation is mildly mixing so it
cannot have rigidity sequences at all.  But there is partial
rigidity in that $T^{h_m} \to \frac 12(Id +T^{-1})$
in the strong operator topology.  So $(h_m)$ and  $(h_m +1)$ are
recurrent sequences for $T$ in a strong sense, while the above is
showing that $(h_m -1)$ is not recurrent for $T$.
\smallskip

\noindent b)The obvious question here is what other sequences, besides ones
like the one above for the Chacon transformation, can be show
to be sequences of non-recurrence for weakly mixing transformations
via classical cutting and stacking constructions?
\medskip

\noindent b) Friedman and King~\cite{FK} consider a class of weakly
mixing, but not strongly mixing, transformations constructed by
Chacon; they prove that these, unlike the Chacon transformation
above, are lightly mixing (see \cite{FK} for the definition)
and so there is always recurrence for
these transformations along any increasing sequence. Hence, these
transformations form a meager set by the category result in Proposition~\ref{revise}.
\end{rem}

If we have a sufficient growth rate assumed for $(n_m)$, we can
give a construction of a weakly mixing transformation which
exhibits non-recurrence along the sequence. The argument here
starts like the constructions in Section~\ref{diophantine}

\begin{prop}
\label{fastworks}
Suppose we have a sequence $(n_m)$ such that
$\sum\limits_{k=1}^\infty n_m/n_{m+1} < \infty$. Then there is a
weakly mixing transformation $T$ for which $(n_m)$ is an IP
rigidity sequence and such that $T$ is not recurrent for
$(n_m-1)$.
\end{prop}
\begin{proof}
Choose a non-decreasing sequence of whole numbers $(h_m)$ with
$1/h_m \ge 10(n_m/n_{m+1})$ and $\sum\limits_{m=1}^\infty
1/h_m < \infty$. We construct a Cantor set $\mathcal C$ with constituent intervals
at each level that are arcs of size $1/(n_mh_m)$ around some of
the $n_m$-th roots of unity (determined as part of the induction).
Our conditions allow us to find in each such constituent interval
many points $j/n_{m+1}$ because $1/n_{m+1}$ is sufficiently
smaller than $1/n_mh_m$, and then select in these constituent
intervals new ones of length $1/(n_{m+1}h_{m+1})$ around some of
the $n_{m+1}$-th roots of unity for $m \ge M$. The resulting
Cantor set $\mathcal C$ has the property that for all points $x$
in the set $n_mx$ is within
$1/h_m$ of an integer
for $m \ge M$. Also, it follows that if we take a continuous, positive
measure $\nu_0$ on $\mathcal C$, with $\nu_0([0,1)) = 1$, then
$|1- \widehat {\nu_0}(n_m)| \le 1/h_m$ for $m \ge M$.

Now we take the GMC construction corresponding to the
symmetrization $\omega$ of $\nu_0$.  We can use Proposition~\ref{OKforIP} and the
result from Erd\H{o}s and Taylor~\cite{ET} cited in Remark~\ref{IPinfo}
to conclude that $(n_m)$ is an IP rigidity sequence for the weakly
mixing transformation $T = G_{\omega}$.  This
gives us a weakly mixing dynamical system $(X,\mathcal B,p,T)$ and a
function $f$ of norm one in $L_2(X,p)$ such that $\|f\circ T^{n_m} -
f\|_2^2 \le C/h_m$ for $m \ge M$. The function $f$ here is the
one in the GMC such that $\nu_f^T = \omega$.
It follows immediately, by the fact that $\omega$ is symmetric and
by the symmetric Fock space construction in the GMC,
that $f$ is a Gaussian variable from the first chaos. So $f$ is real-valued,
and being a Gaussian random variable it takes both positive and negative values.
So we have a non-constant, real-valued function of norm
one such that $\|f\circ T^{n_m} - f\|_2^2 \le 4/h_m$ for $m \ge
M$.

Now we claim that both the positive part $f^+$ and the negative
part $f^-$ of $f$ satisfies the same inequality. To see this, write
\[\int |f\circ T^{n_m} - f|^2\,dp =
\int |f^+\circ T^{n_m} -f^-\circ T^{n_m} - f^+ + f^-|^2\, dp.\]
Expand this into the sixteen terms involved. Use that fact that
the terms $f^+f^-$ and $(f^+\circ T^{n_m})(f^-\circ T^{n_m})$ are
zero, and regroup terms to see that
\begin{eqnarray*} \int |f\circ T^{n_m} - f|^2\,dp &=&
\int |f^+\circ T^{n_m} - f^+|^2\,dp\\
&+&\int |f^-\circ T^{n_m} - f^-|^2\,dp\\
&+& \int 2(f^+\circ T^{n_m})f^- + 2(f^-\circ T^{n_m})f^+ \,dp.
\end{eqnarray*}
Because $2(f^+\circ T^{n_m})f^- + 2(f^-\circ T^{n_m})f^+ $ is
positive, we have
\[\int |f\circ T^{n_m} - f|^2\,dp
\ge \int |f^+\circ T^{n_m} - f^+|^2\,dp\] and
\[\int |f\circ T^{n_m} - f|^2\,dp
\ge \int |f^-\circ T^{n_m} - f^-|^2\,dp.\]

In addition, the same argument above shows that for every constant
$L$, the function $(f-L)^+$ also would satisfy this last estimate
too. Hence, taking $L = 1/2$, we would have $F= f1_{\{f \ge L\}}$
satisfies this inequality too and not being the zero function. But
let $A = \{f \ge L\}$. On $T^{n_k}A \backslash A$, we would have
$|F\circ T^{n_m} - F|^2\ge 1/4$. So $p(T^{n_m}A\backslash A) \le
64/h_m$. But similarly, on $A\backslash T^{n_m}A$, we would have
$|F\circ T^{n_m} - F|^2\ge 1/4$. So $p(A\backslash T^{n_m}A) \le
64/h_m$. The result is that from our original GMC construction, we
can infer the existence of a proper set $A$ of positive measure,
which depends on the original function $f$ and not on $m$, such
that $p(T^{n_m}A\Delta A) \le 64/h_m$ for all $m \ge M$. Hence,
$\sum\limits_{m=1}^\infty p(T^{n_m}A\Delta A) < \infty$.

We do have to also make certain that the set $A$ here is a proper
set i.e. $p(A) < 1$. But $f$ is a Gaussian random variable and
so both
$f^+$ and $f^-$ are non-trivial, and so it is easy to choose a value of $L$
so that the above construction gives us a proper set $A$ of positive measure.

Now we can use the convergence of $\sum\limits_{m=1}^\infty
p(T^{n_m}A\Delta A)$ to construct a set of positive measure $B$
for which $T^{n_m-1}B$ and $B$ are disjoint for all $m$. The
argument is a variation on the one given in
Proposition~\ref{revise}. There is the issue that $TA$ and $A$ are
not necessarily disjoint. But there is some subset $A_0$ of $A$
of positive measure such that $T A_0$ and $A$ are disjoint. So if
we take $B = T A_0\backslash \bigcup\limits_{m=M}^\infty T^{-(n_m
-1)}(T^{n_m}A\Delta A)$, for suitably large $M$, then we would
have $p(B) > 0$, $B \subset T A_0\backslash A$, and $T^{n_m-1}B
\subset A$ for all $m \ge M$. Hence $T^{n_m-1}B$ and $B$ disjoint
for all $m \ge M$.

Now we need to revise the result above so that we get the same
disjointness for all $m$. But here we know that $T$ is weakly
mixing, and consequently all of its powers are ergodic. So one
can inductively revise $B$ as follows. One takes a subset $B_1$
of $B$ such that $T^{n_1-1}B_1$ and $B_1$ are disjoint, then one
takes a subset $B_2$ of $B_1$ such that $T^{n_2 -1}B_2$ and $B_2$
are disjoint, and so on. After a finite number of steps one ends
up with a subset $B_{M-1}$ of $B$ such that $T^{n_m -1}B_{M-1}$
and $B_{M-1}$ are disjoint for all $m \ge 1$. Now, with $B_{M-1}$
replacing $B$, we have $T^{n_m-1}B$ and $B$ disjoint for all $m
\ge 1$. So this construction gives a weakly mixing transformation
that is not only IP rigid along $(n_m)$, but such that along
$(n_m-1)$ it is not recurrent.
\end{proof}

\begin{rem}\label{notrecshiftagain}  With the hypothesis of
Proposition~\ref{fastworks}, by taking $S$ to be
a product of $T$ with itself $2K$ times, we can arrange that the
weakly mixing transformation $S$ is rigid along $(n_m)$ and
there is one set $C, p(C) > 0$, such that $S$ is non-recurrent for $C$
along each of the sequences $(n_m + k)$ with $0 < |k| \le K$.
\end{rem}

\begin{rem} Proposition~\ref{wmunions} allows us to use Proposition~\ref{fastworks}
to give other examples of non-recurrent sequences. Again, as in Remark~\ref{different} a),
both $(2^{n^2})$
and $(2^{n^2}+1)$ satisfy the hypothesis of Proposition~\ref{fastworks},
so there is a weakly mixing transformation for which $\mathbf n =
(\ldots,2^{n^2}, 2^{n^2}+1,\ldots)$ is a sequence of non-recurrence,
even though $\mathbf n$ does not satisfy the hypothesis of
Proposition~\ref{fastworks}.
\end{rem}

\begin{rem} Using Proposition~\ref{specialinfrankone}, we can construct
examples of non-recurrence along $(n_m - 1)$ for $T$ which is weakly mixing
and rank one.  For
example, take $(n_m)$ such that $n_{m+1}/n_m \ge 2$ is a whole
number for all $m$ and such that $\sum\limits_{m=1}^\infty \frac {n_m}{n_{m+1}}
< \infty$.
\end{rem}

\begin{rem} a) In Proposition~\ref{fastworks},
replacing our original sequence by $(n_m +1)$, we
conclude this fact: whenever $(n_m)$ is increasing and
$\sum\limits_{m=1}^\infty \frac {n_m}{n_{m+1}} < \infty$, there exists a
weakly mixing transformation $T$ and a set of positive measure $B$
such that $T^{n_m}B$ and $B$ are disjoint for all $m$. Of course
now the transformation is IP rigid along $(n_m+1)$.
\smallskip

\noindent b) The condition we are using of course will not hold
for lacunary sequences like $n_m = 2^m$. But it is the case that
if $(k_m)$ is lacunary, and $\delta > 1$, then the subsequence
$(n_m) = (k_{\lfloor
m^\delta\rfloor})$ will have our series property. So speeding
up the exponent for a lacunary sequence slightly will give us the
type of non-recurrence that we want.
\end{rem}

\begin{rem} Consider the series
$\sum\limits_{m=1}^\infty p(T^{n_m}A\Delta A)$.  Can this be convergent for all
$A$? This is not obviously impossible, although it seems to us unlikely.
But if one replaces this by the corresponding functional version, then
it cannot be convergent for all functions. First,
one can see that $\sum\limits_{m=1}^\infty p(T^{n_m}A\Delta A)=
\sum\limits_{m=1}^\infty\|f\circ T^{n_m} - f\|_2^2$ if we take $f
= 1_A$. So the question is, can we have a dynamical system in
which the square function $Sf = \left
(\sum\limits_{m=1}^\infty|f\circ T^{n_m} - f|^2\right )^{1/2}$ is
always an $L_2$-function? But it is clear that this is not
possible. If it were, then one can show there is a homogeneous
inequality $\|Sf\|_2 \le K\|f\|_2$ for some constant $K$. Then
one takes again $f=1_A$, and sees that $\sum\limits_{m=1}^\infty
p(T^{n_m}A\Delta A)\le K^2 p(A)$. But this cannot be. Indeed,
just take a very long Rokhlin tower with $A$ as the base and the
left-hand side could exceed the right-hand side.
\end{rem}
\bigskip

\noindent {\bf Acknowledgements}: \,We would like to think S. Eigen, R. Kaufman,
J. King, V. Ryzhikov, S. Solecki, Y. Son, and B. Weiss for their input.

\medskip

{\small
\parbox[t]{3.5in}
{V. Bergelson\\
Department of Mathematics\\
Ohio State University\\
Columbus, OH 43210, USA\\
E-mail: vitaly@math.ohio-state.edu}
\bigskip

{\small
\parbox[t]{3.5in}
{A. del Junco\\
Department of Mathematics\\
University of Toronto\\
Toronto, M5S 3G3, Canada\\
E-mail: deljunco@math.toronto.edu}
\bigskip

{\small
\parbox[t]{5in}
{M. Lema\'nczyk\\
Faculty of Mathematics and Computer Science\\
Nicolaus Copernicus University,
Toru\'n, Poland, and\\
Institute of Mathematics\\
Polish Academy of Sciences, Warsaw, Poland\\
E-mail:
mlem@mat.uni.torun.pl}
\bigskip

{\small
\parbox[t]{5in}
{J. Rosenblatt\\
Department of Mathematics\\
University of Illinois at Urbana-Champaign\\
Urbana, IL 61801, USA\\
E-mail: rosnbltt@illinois.edu}
\bigskip

\end{document}